\documentclass[10pt, a4paper]{article}
\usepackage[margin=1in]{geometry}
\usepackage{color}

\usepackage[T1]{fontenc}

\usepackage{graphicx}

\usepackage{amsmath}
\usepackage{amsthm, amsfonts,amssymb} %amsmath,
\usepackage{lmodern}
\usepackage{mathtools}
\usepackage{tikz}
\newtheorem{thm}{Theorem}[section]
\newtheorem{lem}[thm]{Lemma}

\newtheorem{cor}[thm]{Corollary}

\newtheorem{mydef}{Definition}[section]

\newtheorem{rem}{Remark}[section]
\newcommand{\bFormula}[1]{
\begin{equation} \label{#1}}
\newcommand{\eF}{\end{equation}}

\newcommand{\vu}{{\bf u}}

\newcommand{\R}{{\mathbb{R}}}

\newcommand{\N}{{\mathbb{N}}}
\newcommand{\bfeta}{{\boldsymbol\eta}}
\DeclareMathOperator{\supp}{supp}
\DeclareMathOperator{\dist}{dist}

\newcommand{\norm}[2][]{\left\|#2\right\|_{#1}}

\theoremstyle{remark}
\newcommand\bb[1]{\mathbf{#1}}
\newcommand\ddfrac[2]{\frac{\displaystyle #1}{\displaystyle #2}}

\newcommand\bint[1]{\displaystyle\int #1}

\usepackage[font=footnotesize,labelfont=bf]{caption}
\numberwithin{equation}{section}

\usepackage{todonotes}

\date{}

\newcommand{\Addresses}{{% additional braces for segregating \footnotesize
  \bigskip
  \footnotesize

Malte Kampschulte, \textsc{Department of Mathematical Analysis, Faculty of Mathematics and Physics, Charles University, Prague, Czech Republic}\par\nopagebreak
  \textit{E-mail address}: \texttt{kampschulte@karlin.mff.cuni.cz}

\medskip

  Boris Muha, \textsc{Department of Mathematics, Faculty of Science, University of Zagreb}\par\nopagebreak
  \textit{E-mail address}: \texttt{borism@math.hr}

\medskip
Sr\dj{}an Trifunovi\'{c}, \textsc{Department of Mathematics and Informatics, Faculty of Sciences, University of Novi Sad}\par\nopagebreak
  \textit{E-mail address}: \texttt{srdjan.trifunovic@dmi.uns.ac.rs}

}}

\title{Global weak solutions to a 3D/3D fluid-structure interaction problem including possible contacts}
\author{Malte Kampschulte, Boris Muha, Sr\dj an Trifunovi\'c}
\begin{document}

\maketitle
\begin{abstract}
In this paper, we study an interaction problem between a $3D$ compressible viscous fluid and a $3D$ nonlinear viscoelastic solid fully immersed in the fluid, coupled together on the interface surface. The solid is allowed to have self-contact or contact with the rigid boundary of the fluid container. For this problem, a global weak solution with defect measure is constructed by using a multi-layered approximation scheme which decouples the body and the fluid by penalizing the fluid velocity and allowing the fluid to pass through the body, while the body is supplemented with a contact-penalization term. The resulting defect measure is a consequence of pressure concentrations that can appear where the fluid meets the (generally irregular) points of self-contact of the solid. Moreover, we study some geometrical properties of the fluid-structure interface and the contact surface. In particular, we prove a lower bound on area of the interface.
\end{abstract}
%181 words
\\${}$ 

%In this paper, we study an interaction problem between a $3D$ compressible viscous fluid and a $3D$ nonlinear viscoelastic solid, coupled on the interface via kinematic and dynamic boundary conditions prescribing the  the continuity of velocities and balance of contact forces, respectively. A multi-layered approximation scheme is constructed which decouples the body and the fluid by penalizing the kinematic boundary condition and allowing the fluid to pass through the body, while the body is supplemented with a contact-penalization term. In such a way, the fluid sub-problem is solved on a fixed domain by using the standard theory, while the structure sub-problem is solved by time-discretizing and applying the variational techniques on each time step. Then, a global weak solution with a defect measure is obtained as a limit of these approximate solutions which allows the body to have self-contact and contact with the rigid boundary. The defect measure is a consequence of pressure concentractions near the points where the interface is irregular. Moreover, we study some geometrical properties of the fluid-structure interface and the contact surface. In particular, we prove a lower bound on area of the interface.

%136 words in abstract
\section{Introduction}
\subsection{Problem definition}
We consider the motion of $3D$ viscoelastic solid fully immersed into a barotropic compressible fluid contained in a container with rigid walls. Let $\Omega_S$, $\Omega\subset\R^3$ be open, bounded, connected, $C^{1,\alpha}$ sets, $\alpha>0$, representing the reference position of the viscoelastic solid and the container, respectively. The unknowns of the system are the following:
\begin{enumerate}
  \item[] The elastic solid deformation  $\dotfill$ $\bfeta :(0,T)\times {\Omega_S}\to \Omega$;  $\qquad$  $\qquad$
  \item[] The fluid velocity $\dotfill$ $\mathbf{u}:(0,T)\times \Omega_F^{\bfeta }(t)\to \mathbb{R}^3$; $\qquad$  $\qquad$
  \item[] The fluid density $\dotfill$ $\rho:(0,T)\times \Omega_F^{\bfeta }(t)\to \mathbb{R}$;  $\qquad$  $\qquad$
  \item[] The elastic solid at time $t$ $\dotfill$ $\Omega^\bfeta_S(t):=\text{int}(\bfeta(t,\overline{\Omega_S}))$; $\qquad$  $\qquad$
   \item[] The fluid domain at time $t$ $\dotfill$ $\Omega_F^{\bfeta }(t):=\text{int}(\Omega\setminus \Omega_S^\bfeta(t))$. $\qquad$  $\qquad$
\end{enumerate}
Notice that such a non-standard definition for fluid and solid domains (see Figure \ref{fig:setup}) is used because we want to include the interior points of self-contact into the structure domain, rather than into the fluid domain. This happens when $\bfeta(t)$ is not injective on $\partial\Omega_S$. We assume $\Omega_S^\bfeta(0)\neq\Omega$, i.e.\ the solid does not completely fill the container at the initial time. We will prove that, as a consequence of conservation of mass and boundedness of energy, that the same is true at any time. Moreover, for time-dependent sets, by slight abuse of notation, we will write $I \times S(t):=\cup_{t\in I} \{t\}\times S(t)$, and some time-space domains which will be frequently used will be denoted as:
\begin{eqnarray*}%%%%%%%%%%%%%
     Q_{S,T}:=(0,T)\times {\Omega_S},\quad Q_{S,T}^\bfeta:=(0,T)\times \Omega_S^{\bfeta }(t), \quad  Q_{F,T}^{\bfeta }:=(0,T) \times \Omega_F^{\bfeta }(t), \quad Q_T:=(0,T)\times \Omega.
\end{eqnarray*}%%----------------------------%%

\begin{figure}[h!]
    \centering
    \includegraphics[width=\textwidth]{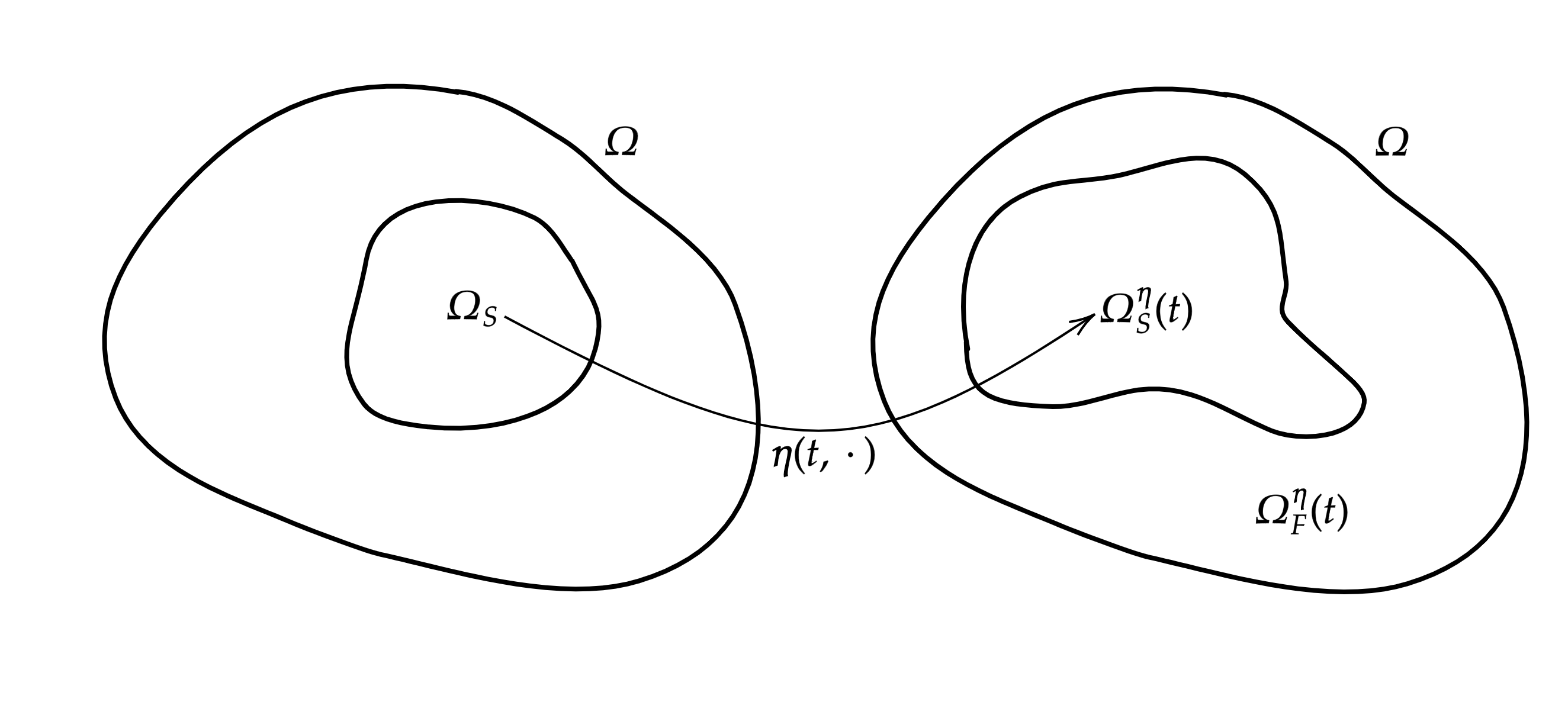}
    \caption{Lagrangian coordinates (left) and physical coordinates (right) at a time $t$.}
    \label{fig:setup}
\end{figure}
\subsubsection{Elastodynamics}\label{sec:el}
The motion of the solid is given by nonlinear elastodynamics equations:
\begin{eqnarray}%%%%%%%%%%%%%
    \partial_{t}^2 {\bfeta } =\nabla \cdot \boldsymbol\sigma (\bfeta,\partial_t\bfeta)\quad{\rm in}\quad  Q_{S,T},\label{elasticity}
    \end{eqnarray}%%----------------------------%
    where $\boldsymbol\sigma$ is the first Piola-Kirchhoff stress tensor. We require that deformation satisfies the Ciarlet-Ne\v cas condition introduced in \cite{ciarletInjectivitySelfcontactNonlinear1987}, i.e.\ it is orientation preserving and injective on the interior of $\Omega_S$. Therefore we define the following function space for the deformations:
\begin{eqnarray*}%%%%%%%%%%%%%
    \mathcal{E}^q:=\left\{ \bfeta\in W^{2,q}(\Omega_S):|\bfeta(\Omega_S)|=\int_{\Omega_s} \det \nabla\bfeta, \quad \inf_{\Omega_S} \det\nabla\bfeta >0,\quad  \overline{\Omega_S^\bfeta}\subseteq \overline{\Omega} \right\},
\end{eqnarray*}%%----------------------------%%
for $q>3$. We also prescribe the following boundary condition when the solid is in contact with the rigid boundary, which states that the solid does not penetrate or slip on the rigid boundary:
\begin{align}\label{elasticityBC}
    \partial_t\bfeta(t,x)=0, \quad{\rm for}\quad t\in (0,T) \text{ and } x\in \partial\Omega_S \text{ such that } \bfeta(t,x)\in \partial\Omega.
\end{align}
Here we consider nonlinearly elastic material of second order, i.e.\ a generalized standard material (see e.g. \cite{halphenMateriauxStandardGeneralises1975}, \cite[Sec. 2.5]{kruzik2019mathematical}) with constitutive law:
\begin{eqnarray*}%%%%%%%%%%%%%
    \nabla\cdot \boldsymbol\sigma (\bfeta,\partial_t\bfeta) := -DE(\bfeta)- D_2R(\bfeta,\partial_t \bfeta), \quad \text{ in } \mathcal{D}'(Q_{S,T}),
\end{eqnarray*}%%----------------------------%%
where $E$ is the energy functional, $R$ is the dissipation functional, $D$ is the Fr\'{e}chet derivative and $D_2$ is the Fr\'{e}chet derivative with respect to the second argument. Since we are working with a second-gradient material, we must impose another boundary condition on $\nabla^2 \bfeta$. This condition is chosen so that the energy in our coupled system is conserved properly. Since it is cumbersome to write it explicitly, we can define it implicitly by imposing that
\begin{eqnarray}%%%%%%%%%%%%%
    \int_{\Omega_S} \boldsymbol\sigma(\bfeta,\partial_t\bfeta): \nabla\boldsymbol\xi =\langle DE(\bfeta),\boldsymbol\xi \rangle+ \langle D_2R(\bfeta,\partial_t \bfeta), \boldsymbol\xi \rangle, \quad \text{ for all }\boldsymbol\xi\in C^\infty(\Omega_S). \label{2nd:order:bnd}
\end{eqnarray}%%----------------------------%%

Finally, we assume that $E$ and $R$ satisfy the following conditions:
\begin{itemize}
    \item[(E1)]  $E: W^{2,q}(\Omega_S)\to \mathbb{R} \cup \{+\infty\}$  is bounded from below, weakly lower-semicontinuous and coercive;
    \item[(E2)] For every $E_0 > 0$ there is a $\varepsilon_0 > 0$ such that $\det \nabla \bfeta(x) > \varepsilon_0$ for all $x \in \Omega_S$ and all $\boldsymbol \bfeta$ with $E(\boldsymbol \bfeta) < E_0$;
    \item[(E3)] Whenever it is finite, $E$ is Fréchet-differentiable with a derivative $DE: \{\bfeta \in W^{2,q}(\Omega_S): E(\bfeta) < \infty\} \to [W^{2,q}(\Omega_S)]'$ which is continuous with respect to weak topology;
    \item[(E4)] There is a Minty-type property: If $\bfeta_k \rightharpoonup \bfeta$ in $W^{2,q}(\Omega_S)$ and $E(\bfeta_k)$ is bounded, then $\liminf_{k\to\infty} \langle DE(\bfeta_k)-DE(\bfeta),\psi (\bfeta_k-\bfeta)\rangle\geq0$, where $\psi \in C_c(\Omega_S;[0,1])$ is an arbitrary cut-off function. Furthermore, if also $\limsup_{k\to \infty} \langle DE(\bfeta_k) - DE(\bfeta), \psi (\bfeta_k - \bfeta) \rangle \leq 0$ for all such $\psi$, then this implies $\bfeta_k \to \bfeta$ in $W^{2,q}(\Omega_S)$;
    \item[(R1)] $R: W^{2,q}(\Omega_S) \times W^{1,2}(\Omega_S) \to \mathbb{R}$ is weakly lower semicontinuous,  2-homogeneous in the second argument, i.e. $R(\boldsymbol \bfeta, \lambda \boldsymbol b) = \lambda^2 R(\boldsymbol \bfeta, \boldsymbol b)$ for all $\lambda \in \mathbb{R}$ and continuously Fréchet-differentiable with respect to the second argument;
    \item[(R2)] There exists a (possibly energy dependent) Korn-inequality: For any fixed $E_0 > 0$ and all $\bfeta \in \mathcal{E}^q$ with $E(\bfeta) \leq E_0$, we have $\|b\|_{W^{1,2}} \leq c(E_0) (\|b\|_{L^2} + R(\bfeta,b))$ for all $b \in W^{1,2}(\Omega_S)$.
\end{itemize}

A prototypical example of such an energy-dissipation pair is of the form
\begin{align*}
    R(\bfeta,\partial_t \bfeta) &:= \int_{\Omega_S} | \partial_t \nabla \bfeta^\tau \nabla \bfeta + \nabla \bfeta^\tau \partial_t \nabla \bfeta |^2, \\
    E(\bfeta) &:= \begin{cases} \int_{\Omega_S} \mathcal{C}(\nabla \bfeta^\tau \nabla \bfeta -I) + \frac{1}{\det (\nabla \bfeta)^a} + \frac{1}{q} |\nabla^2 \bfeta |^q, & \text{ if } \det \nabla \bfeta > 0 \text{ a.e. in } \Omega_S, \\ \infty, & \text{ otherwise} ,\end{cases}
\end{align*}
where $\mathcal{C}$ is a positive definite quadratic tensor and $a > \frac{3q}{q-3}$. For a much more detailed discussion, we refer to Sections 1.4 and 2.3 of \cite{benevsova2020variational}.

\subsubsection{Fluid equations}
The fluid flow is governed by the compressible Navier-Stokes equations:\footnote{One can also introduce a forcing term of the form $\rho \bb{F}$ on the right-hand side of the second equation ($\bb{F}$ can be gravity for example), which was done in \cite{breit2021compressible}. This however would not affect the results presented in this paper, so it is omitted for simplicity.}
\begin{eqnarray}%%%%%%%%%%%%%
    \partial_t \rho + \nabla \cdot (\rho \mathbf{u}) &=& 0\quad{\rm in}\quad   Q_{F,T}^{\bfeta },\label{conteq}\\[2mm]
\partial_t (\rho\mathbf{u}) + \nabla \cdot (\rho\mathbf{u}\otimes \mathbf{u}) &=& \nabla \cdot\mathbb{T}(\vu,p)\quad{\rm in}\quad   Q_{F,T}^{\bfeta }, \label{momeq} 
\\[2mm]
\vu &=&0\quad{\rm on}\; (0,T)\times(\partial\Omega\setminus \partial\Omega_S^\bfeta(t)),
\end{eqnarray}%%----------------------------%%
where $\mathbb{T}(\vu,p)=\mathbb{S}(\vu)-p\mathbb{I}$ is the fluid Cauchy stress tensor. Here $\mathbb{S}$ is the viscous stress tensor whose constitutive relation is given by standard Newton rheological law
\begin{eqnarray*}%%%%%%%%%%%%%
\mathbb{S}(\nabla \bb{u}):=\mu \big( \nabla \bb{u} + \nabla^\tau \bb{u}-\frac{2}{3} (\nabla \cdot \bb{u}) \mathbb{I}\big) + \zeta  (\nabla \cdot \bb{u}) \mathbb{I},\quad \mu,\zeta>0,
\end{eqnarray*}%%----------------------------%%
where $\mu$ is the shear viscosity coefficient and $\zeta$ the bulk viscosity coefficient. We assume following state equation for the pressure
\begin{align}\label{PressureLaw}
    p=p(\rho)=\rho^{\gamma},
\end{align}
where $\gamma>1$ is the adiabatic exponent.

\subsubsection{Coupling conditions}
Since we consider a coupled problem, we need to prescribe two types of coupling conditions for two types of interactions. They read as follows:

\bigskip

\noindent \textbf{Fluid-structure coupling} on the Lagrangian fluid-structure interface $Q_\bfeta^I$:
\begin{align}%%%%%%%%%%%%%
    &\text{Kinematic\;coupling\;condition:}& ~ \partial_t \bfeta &= \mathbf{u} \circ \bfeta, \label{kinc}\\
    &\text{Dynamic\;coupling\;condition:}&~ \boldsymbol\sigma(\bfeta,\partial_t\bfeta)\bb{n}&=  \det \nabla\bfeta\left [\mathbb{T}(\vu,p)\circ\bfeta\right ]\nabla\bfeta^{-\tau}\bb{n}, \label{dync}
\end{align}%%---------------------------
where $Q_\bfeta^I\subseteq (0,T)\times \partial\Omega_S$ is defined in Section \ref{weak:sol:sec} and $\bb{n}$ is the unit outer normal vector on $\partial\Omega_S$; \\ \\
\noindent \textbf{Structure-structure coupling} on the set of self-contact
\begin{align}%%%%%%%%%%%%%
    &\text{Kinematic\;coupling\;condition:}& ~ \partial_t \bfeta(t,x) &= \partial_t \bfeta(t,y), \quad \label{kinca}\\
    &\text{Dynamic\;coupling\;condition:}&~ \boldsymbol\sigma(\bfeta,\partial_t\bfeta)(t,x)\bb{n}(x)dS(x) &= - \boldsymbol\sigma(\bfeta,\partial_t\bfeta)(t,y)\bb{n}(y)dS(y) , \quad \label{dynca}
\end{align}%%----------------------------%%
for all $t\in (0,T)$ and $x\neq y\in \partial\Omega_S$ such that  $\bfeta(t,x)=\bfeta(t,y)$, where $dS:\partial\Omega\to \mathbb{R}^+$ is the surface element. \\

Notice that since we prescribe coupling conditions on the reference domain, we needed to apply the Piola transform to the Cauchy stress tensor of the fluid to get the first Piola-Kirchhoff stress tensor in \eqref{dync}. Next, we prescribe the inital data:
\begin{eqnarray}\label{initialdata}
{\bfeta }(0,\cdot)={\bfeta }_0,\quad \partial_t {\bfeta }(0,\cdot)=\bb{v}_0,\quad ~\rho(0,\cdot) = \rho_0, \quad (\rho\mathbf{u})(0, \cdot)=(\rho\bb{u})_0,
\end{eqnarray}
and assume that it satifies the following compatibility conditions:
\begin{eqnarray}\nonumber
    &&\rho_0\in L^{\gamma}(\Omega^{\bfeta_0}), \quad \rho_0 \geq 0, \quad \rho_0 \not\equiv 0, \quad \rho_{0}|_{\mathbb{R}^3 \setminus \Omega^{\bfeta_0}} =0,\\ \label{IC_Compatibility}
&&\rho_0>0 \text{ in } \{x\in \Omega_F^{\bfeta_0}: (\rho\bb{u})_0(x) >0\}, \quad  \frac{(\rho \bb{u})_0^2}{\rho_0} \in L^1(\Omega_F^{\bfeta_0}).
\end{eqnarray}%%----------------------------%%

Finally, let us point out that a smooth solution of the problem \eqref{elasticity}-\eqref{initialdata} satisfies the following energy identity in $(0,T]$:
\begin{align}%%%%%%%%%%%%%
&\frac{d}{dt} \int_{\Omega_F^{\bfeta }} \Big( \frac{1}{2} \rho |\bb{u}|^2 + \frac{\rho^\gamma}{\gamma-1}   \Big)  + \int_{\Omega_F^{\bfeta }} \mathbb{S}(\nabla \mathbf{u}):\nabla \mathbf{u} \nonumber\\
&+ \frac{d}{dt} \Big(\int_{{\Omega_S}} \frac{1}{2}|\partial_t {\bfeta }|^2 + E(\nabla\bfeta )  \Big)+ 2R(\bfeta,\partial_t \bfeta) =0. \label{enid}
\end{align}%%----------------------------%%
This identity is formally obtained at the end of Section \ref{derivation}. The goal of this paper is to study global weak solutions with a pressure defect measure to problem \eqref{elasticity}-\eqref{initialdata} which satisfy the corresponding energy inequality.

\subsection{Main results and literature review}

The main result of the paper is the following existence theorem:
\begin{thm}\label{main:thm}
Let $\Omega$, $\Omega_S$ be open bounded sets of class $C^{1,\alpha}$, $\alpha>0$, $T>0$ and let structure energy and dissipation functionals $E$ and $R$ satisfy conditions (E1)-(E4) and (R1)-(R2) respectively. Moreover, let $\bfeta_0\in \mathcal{E}^q$ for $q>3$, ${\bf v}_0\in L^2(\Omega_S)$, $\rho_0\in L^\gamma(\Omega_F^{\bfeta_0})$ for $\gamma>12/7$ and $(\rho \bb{u})_0 \in L^{\frac{2\gamma}{\gamma+1}}(\Omega_F^{\bfeta_0})$ are such that compatibility conditions $\eqref{IC_Compatibility}$ are satisfied and $\Omega^\bfeta_S(0)\neq \Omega$. Then there exists a global weak solution with defect measure to the problem \eqref{elasticity}-\eqref{initialdata} on $(0,T)$ in the sense of Definition \ref{weaksol}.
\end{thm}

One of the main difficulties of the considered problem is that the fluid-solid interface can be quite irregular. Namely, while the space $\mathcal{E}^q$ is embedded in $C^{1,\alpha}(\Omega_S)$ and the surface deformation $\bfeta|_{\partial\Omega_S}$ is a $C^{1,\alpha}$ local diffeomorphism, it is in general not injective. Therefore one cannot avoid irregularities in the fluid domain such as cusps or even surfaces that are not representable by a union of overlapping graphs, which gives rise to defect measures that correspond to concentrations of the fluid pressure. Please see Section \ref{rem:cantor} for more details about possible irregularities of the fluid domain. In this sense, our result is optimal and one cannot expect more regular solutions without extra assumptions. However, we can prove that the interface corresponding to a weak solution has certain nice topological properties, see Theorem \ref{lemma:contact2}. In particular we prove the following result, obtained there in Claim 4:
\begin{thm}
Let $(\bfeta,\rho,\vu)$ be a weak solution to problem \eqref{elasticity}-\eqref{initialdata} in the sense of Definition \ref{weaksol}. Then there exits a constant $A_0>0$, depending only on the initial energy and the domains $\Omega,\Omega_S$, such that area of the fluid-solid interface is bounded below by $A_0$.
\end{thm}

The main novelty of these results is that we consider \textit{global} solutions to a $3D$/$3D$ fluid-structure interaction (FSI) problem that allow possible self-contact of the solid and solid-rigid boundary contact. To the best of our knowledge, this is the first result where such global solutions are constructed and where contact conditions are formulated in the context of fluid-structure interaction. 

The problem of describing contact is one of the most challenging problems in the analysis of the FSI problems and majority of results in the area do not allow contact, i.e.\ analyse the problem only before the contact occurs. So far there are only a few results that consider global-in-time solution in the context of FSI. In the simplest case when the structure is rigid, the existence of a global-in-time weak solution was proved for compressible fluid and no-slip kinematic coupling conditions by Feireisl \cite{Feireisl03} and for the incompressible fluid with slip kinematic coupling conditions by Chemetov and Ne\v casova \cite{ChemetovNecasova}. An analogous global existence result for a FSI system consisting of $2D$ incompressible viscous fluid and a $1D$ elastic beam was proved in \cite{CasanovaGrandmontHillairet}. In the context of strong solutions, there are several small data global-time existence results. In \cite{Leq13} such a result was proved for the incompressible Navier-Stokes equations coupled with a damped wave equation, while in \cite{IKLT14,IKLT17} analogous results were proved for an FSI problem between $3D$ fluid and $3D$ linearly elastic solid, but with some additional interface or structure damping terms.

Another approach to proving global-in-time existence results for FSI problems is to show that in certain situation the contact will not occur in finite time -- the so-called ``no contact paradox''. This phenomenon is by now well understood in the case where solid is rigid, see e.g.\ \cite{HT09,GVH10}. Recently, significant progress was made by Grandmont and Hillairet  \cite{GraHill16} who proved the existence of a global-in-time strong solution to a 2D FSI problem involving a viscoelastic beam, where they also showed that the contact does not occur in finite time. We emphasize that this result essentially uses the global strong solution with large initial data and therefore cannot be directly applied to $3D$ case. 

In contrast to this, in solid mechanics without intervening fluids, contact has been long established as something that will occur and many methods to deal with it have been established, from classic results, e.g.\ \cite{MR118021,ballGlobalInvertibilitySobolev1981,ciarletInjectivitySelfcontactNonlinear1987}, going all the way to most recent studies of hard dynamic self contact such as \cite{kromerQuasistaticViscoelasticitySelfcontact2019,cesikInertialEvolutionNonlinear2022}. As we have the fluid acting as a kind of physical buffer to the contact, our approximation relies on penalization by a soft potential, something that has been well studied in particular in the context of numerics (see e.g.\ \cite{kromerGlobalInjectivitySecondgradient2019}).

The second significant novelty of our result is that we consider a nonlinear $3D$ viscoelastic solid. The vast majority of papers that analyse FSI problems involving $3D$ elastic solids work with local-in-time regular solutions, see e.g.\ \cite{CS06,KT12,RV14,BG17}. For the analysis of the static case we refer reader to \cite{Grandmont02,GaldiKyed}. The standard elastic energy (e.g.\ Saint Venant–Kirchhoff) does not provide enough regularity for even defining weak solution of Leray-Hopf type. Therefore in order to work with weak solution one needs to use an elasticity model with better regularity properties. To the best of our knowledge there are two approaches in the literature. The first one is to consider an elastic interface with mass, e.g.\ \cite{multilayered,SunMarBor}. In this paper we adapt the second approach, i.e.\ we use a second order nonlinear viscoelastic model for the solid. For such a model the existence a weak solution to  FSI problem with incompressible and compressible fluid has been obtained in \cite{benevsova2020variational,breit2021compressible}. We note that viscoelastic regularization is needed in our approach in order to deal with the nonlinearities that arise when studying large deformations and moving boundaries. However, viscoelastic terms are not merely a mathematical regularization, but rather they are derived from the physical characteristics of viscoelastic solids such as biological materials and polymers. 

We finish this brief literature review by mentioning that well-posedness theory for weak solutions to FSI problems with simpler models for solid is by now well-developed: existence theory when the solid is rigid e.g.\ \cite{Gunzburger00,Tucsnak02}, uniqueness theory when the solid is rigid e.g.\ \cite{Glass15,Bravin19,Radosevic21}, existence theory when the solid is described by a lower dimensional model (plate/shell/beam etc.) e.g.\ \cite{MR2438783,SubBorARMA,compressible,BorSeb,srdjan1}, uniqueness theory when the solid is described by a lowerdimensional model e.g. \cite{SS22,SrdjanWS}. Finally, we refer the interested reader to a nice review of the current state-of-the-art in FSI \cite{SunBulletin}, and \cite{FSIBioMed} for a discussion of applications of FSI in biomedicine.

\subsection{The strategy of the proof and discussion of some resulting implications}
The weak solution is obtained as a limit of approximate solutions which are constructed as follows. The problem is decoupled by allowing the fluid to pass through the elastic solid and extending the fluid equations to the entire domain $\Omega$, while the fluid and solid velocities are penalized by each other on the solid domain. Moreover, the elastic solid is supplemented with a contact-penalization term which prevents contact with the rigid boundary $\partial\Omega$ and self-contact. The initial fluid density is zero on the solid domain. Finally, some regularizations are added to the problem to ensure that the scheme and the convergences work properly. Let us point out the main advantages in this construction. First, the fluid and structure sub-problems are solved separately, sending data to one another through a time-stepping scheme. The fluid equations are solved by using a Galerkin-based approach on a fixed domain rather than on a moving domain, while the structure equations are in contrast solved by time-discretizing and using the variational techniques. This is possible only because of the decoupling and it cannot be overstated how much it simplifies the analysis. Moreover, the flexibility of this construction paves the road for future extensions, where more complex fluid or structure models can be used. 

The second part of the proof are the convergences. Here, the kinematic coupling is recovered due to existence of $H_x^1$ global velocity field (which is a consequence of viscosity in both fluid and solid) and the fluid cannot pass through the elastic solid anymore, so the initial zero density of the fluid in the solid domain is proved to stay this way at all times. As the contact-penalization vanishes, the structure can have contact with the rigid boundary and with itself, while overlap is still prevented, thus giving us a global solution. The proof of this part is mostly standard, as the methods used for compressible viscous fluids and second-gradient elastic solids can be used here as well. The key difference is the convergence of pressure. The appearance of contact leads to the formation of an irregular boundary in the fluid domain. At the point of contact, the fluid domain is bounded by two graphs that are tangent to each other. As a result, the boundary of the fluid domain will have cusps, and it is no longer Lipschitz, which means that it is not differentiable at some points. This irregularity can result in concentrations of pressure, which correspond to the defect measure in the definition of our weak solution.
Note that such behaviour is to be expected because fluids are in general strongly affected by the shape of the domain and the theory for compressible viscous fluids developed so far gives us no information about the pressure near such irregular boundaries. The only result that touches this issue was obtained by Kuku\v{c}ka in \cite{kukucka}, where pressure at cusps of $W^{1,s}$-regularity for $s\geq 2$ can be controlled provided $\gamma \geq \frac{3s}{2s-3}$. Note that this result cannot be used in our problem due to the reasons presented above. To conclude, the defect measure in the weak solution definition allows us to capture concentrations of pressure that appear because of the  irregularities of the fluid domain due to the contact. Therefore, it seems that the weak solution with defect measure is the optimal framework to study our problem.

%Note that this information is not of particular use for us, due to reasoning given above.
%Therefore, the weak solution with defect measure we obtain here seems to be optimal.

At the end, let us point out that one might consider the incompressible fluid model instead, in order to avoid having to deal with defect measures directly. However incompressible fluids do not allow us to study the properties of the pressure nearly as well as the compressible case does, as it is only recovered as a Lagrange multiplier instead of a quantity given by constitution law \eqref{PressureLaw}. Proving existence of solutions would require the test functions $\boldsymbol\xi$ to be divergence-free on the fluid domain and thus even further dependent on the solution. This was done without contact in \cite[Lemma A.6]{benevsova2020variational} by constructing approximate test functions which are additionally divergence-free inside a $\varepsilon$-neighbourhood of the interface and could likely be extended to contact as well. However the resulting pressure will still be a distribution and thus not an improvement on the solutions obtained for the compressible case.

\bigskip

The paper is organized as follows. In Section $\ref{weak:sol:sec}$, the concept of weak solution is introduced. In Section $\ref{sec:Top}$, some topological properties of fluid and structure domains for weak solutions are studied. Section $\ref{form:eq}$, it is proved that smooth solution are also weak solutions, and that weak solutions which are smooth satisfy the original problem pointwise. In Section $\ref{sec:app:constr}$, the approximate solutions are constructed via decoupling time-stepping penalization approximation scheme. In Section $\ref{sec:h:lim}$, a first limit passage is performed in order to obtain a solution to a fully coupled problem without contact. Finally, in Section $\ref{sec:last:limit}$, the weak solution to the original problem is obtained by removing the contact-penalization, thus allowing the contact to happen.

\section{Weak solution}\label{weak:sol:sec}
First, note that the elements of the set $\mathcal{E}^q$ of possible deformations can still have self-contact or contact with the rigid boundary. Therefore, it is useful to characterize the physical boundary for such deformations:
\begin{eqnarray*}%%%%%%%%%%%%%
    &&C_\bfeta:=\{x\in \partial\Omega_S: \bfeta(x)\in \partial\Omega \},\\
    &&I_\bfeta:=\{x\in \partial\Omega_S\setminus C_\bfeta:\bfeta \text{ is injective at }x  \},\\
    &&N_\bfeta:=\{x\in \partial\Omega_S\setminus C_\bfeta: \bfeta \text{ is not injective at }x  \}.
\end{eqnarray*}%%----------------------------%%
Here injective at $x$ means: $(\forall y\in \overline{\Omega_S}\setminus\{x\})\;\bfeta(x)\neq\bfeta(y)$. Notice that since $\bfeta\in\mathcal{E}^q$ and therefore satisfies the Ciarlet-Ne\v cas condition, injectivity at $x$ is equivalent to: $(\forall y\in \partial\Omega_S\setminus\{x\})\;\bfeta(x)\neq\bfeta(y)$.
By definition we have $\partial\Omega_S=C_\bfeta\cup I_\bfeta\cup N_\bfeta$ and $C_\bfeta\cap I_\bfeta=C_\bfeta\cap N_\bfeta=I_\bfeta\cap N_\bfeta=\emptyset$. 
\begin{figure}[h!]
    \centering
    \includegraphics[width=\textwidth]{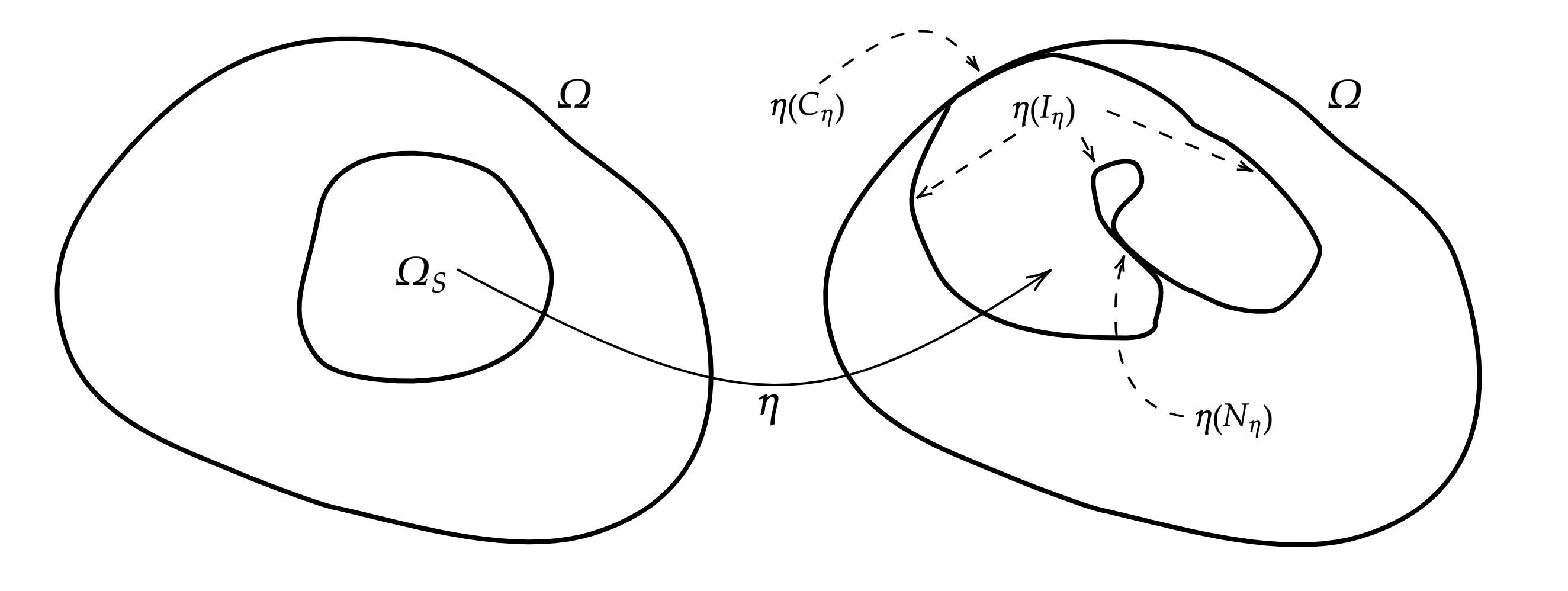}
    \caption{Sets $\bfeta (C_\bfeta), \bfeta(I_\bfeta)$ and $\bfeta(N_\bfeta)$ for a given deformation $\bfeta\in \mathcal{E}^q$.}
    \label{fig:my_label}
\end{figure}
The boundaries of the sets $C_\bfeta,I_\bfeta, N_\bfeta$ in the topology of $\partial \Omega_S$ will be denoted with the standard notation
\begin{eqnarray*}%%%%%%%%%%%%%
    \partial C_\bfeta,\quad \partial I_\bfeta, \quad \partial N_\bfeta.
\end{eqnarray*}%%----------------------------%%
Next, let us denote the time-dependent sets as
\begin{eqnarray*}%%%%%%%%%%%%%
    &&C_\bfeta(t):=C_{\bfeta(t)},\qquad I_\bfeta(t):=I_{\bfeta(t)}, \qquad N_\bfeta(t):=N_{\bfeta(t)}
\end{eqnarray*}%%----------------------------%%
and the time-space cylinders 
\begin{eqnarray*}%%%%%%%%%%%%%
     Q_\bfeta^C:=[0,T]\times C_\bfeta(t), \qquad  Q_\bfeta^I:=[0,T]\times I_\bfeta(t), \qquad  Q_\bfeta^N:=[0,T]\times N_\bfeta(t).
\end{eqnarray*}%%----------------------------%%
We will study topological properties of these sets in Section \ref{sec:Top}. \\

To define weak solutions we need to introduce appropriate functional framework.
For a given $\bfeta$, let us introduce the functional spaces defined on time-dependent domains:
\begin{eqnarray*}%%%%%%%%%%%%%
    L^p(0,T; L^q(\Omega_F^\bfeta(t))):=\left\{f \in L^1(Q_{F,T}^\bfeta):\left(\int_{\Omega_F^\bfeta(t)}|f(t,x)|^q dx\right)^{1/q} \in L^p(0,T) \right\}
\end{eqnarray*}%%----------------------------%%
and
\begin{eqnarray*}%%%%%%%%%%%%%
     L^p(0,T; W^{1,q}(\Omega_F^\bfeta(t))):=\{f \in L^1(Q_{F,T}^\bfeta): f,\nabla f\in L^p(0,T; L^q(\Omega_F^\bfeta(t))) \}.
\end{eqnarray*}%%----------------------------%%
Now, let us introduce the concept of weak solution we will study (see Section \ref{derivation} for its formal derivation):
\begin{mydef}\label{weaksol}
Let $\Omega^\bfeta_S(0)\neq \Omega$, $\bfeta_0 \in \mathcal{E}^q$ for $q>3$, ${\bf v}_0\in L^2(\Omega_S)$, $\rho_0\in L^\gamma(\Omega_F^{\bfeta_0})$ and $(\rho \bb{u})_0 \in L^{\frac{2\gamma}{\gamma+1}}(\Omega_F^{\bfeta_0})$ such that compatibility conditions $\eqref{IC_Compatibility}$ are satisfied. We say that $({\bfeta },\rho,\bb{u})$ is a weak solution to the problem \eqref{elasticity}-\eqref{initialdata} with defect measure $\pi\in L_{w^*}^\infty(0,T; \mathcal{M}^+(\partial N_\bfeta(t)))$ if:
\begin{enumerate}
\item $\bfeta \in L^\infty(0,T; \mathcal{E}^q)$, $\partial_t \bfeta \in L^2(0,T;H^1(\Omega_S))$,\\
$\bb{u} \in  L^2(0,T; H^1(\Omega_F^\bfeta(t)))$,\\
$\rho \in L^\infty(0,T;L^\gamma(\Omega_F^\bfeta(t)))\cap L^{\gamma+\theta}\big(0,T;L_{loc}^{\gamma+\theta}(\Omega_F^\bfeta(t))\big)$, for  $\theta=\frac23\gamma -1$,\\
$\rho\bb{u} \in L^\infty(0,T; L^{\frac{2\gamma}{\gamma+1}}(\Omega_F^\bfeta(t)))$, $\rho|\bb{u}|^2 \in L^\infty(0,T; L^{1}(\Omega_F^\bfeta(t)))$;
\item The kinematic coupling $\partial_t \bfeta = \mathbf{u} \circ \bfeta$ holds in the sense of traces on $Q_\bfeta^I$;
\item The structure velocities match at points of contact in the sense of traces, i.e.
\begin{eqnarray*}%%%%%%%%%%%%%
    \partial_t \bfeta(t,x)=\partial_t \bfeta(t,y), \quad \text{for a.a. } t\in [0,T] \text{ and } x,y\in N_\bfeta(t) \text{ such that }\bfeta(t,x)=\bfeta(t,y);
\end{eqnarray*}%%----------------------------%%
\item The renormalized continuity equation is satisfied in the weak sense
\begin{eqnarray}\label{reconteqweak}%%%%%%%%%%%%%
    \int_{Q_{F,T}^\bfeta} \rho B(\rho)( \partial_t \varphi +\bb{u}\cdot \nabla \varphi) -\int_{Q_{F,T}^\bfeta} b(\rho)(\nabla\cdot \bb{u}) \varphi =\int_{\Omega_F^{\bfeta_0}}\rho_0 B(\rho_0)\varphi(0,\cdot)
\end{eqnarray}%%----------------------------%%
for all $\varphi \in C_0^\infty([0,T)\times \mathbb{R}^3)$ and any $b\in L^\infty (0,\infty) \cap C[0,\infty)$ such that $b(0)=0$ with $B(\rho)=B(1)+\int_1^\rho \frac{b(z)}{z^2}dz$. %and any $b\in L^\infty\cap C[0,\infty)$ such that $b(0)=0$ with $B(\rho)=B(1)+\int_1^\rho \frac{b(z)}{z^2}dz$.

 %   \item The renormalized continuity equation is satisfied in the weak sense
%\begin{eqnarray}\label{reconteqweakeps}%%%%%%%%%%%%%
%    \int_{Q_T^{\bfeta }} \rho B(\rho)( \partial_t \varphi +\bb{u}\cdot \nabla \varphi) =\int_{Q_T^{\bfeta }} b(\rho)(\nabla\cdot \bb{u}) \varphi -\int_0^T \frac{d}{dt}\int_{\Omega_F^{\bfeta }(t)} \rho B(\rho) \varphi,
%\end{eqnarray}%%----------------------------%%
%for all $\varphi \in C^\infty([0,T]\times \overline{\Omega_F^{\bfeta }(t)})$ and any $b\in L^\infty\cap C[0,\infty)$ such that $b(0)=0$ with $B(\rho)=B(1)+\int_1^\rho \frac{b(z)}{z^2}dz$.
%There exists a Young measure $(\nu_{t,x})_{(t,x)\in \beta_T}$ such that \int_{Q_{S,T}} \langle\nu_{t,x}, g_{i\alpha}\rangle \cdot \nabla \phi
\item The coupled momentum equation
\begin{eqnarray}%%%%%%%%%%%%%
    &&\int_{Q_{F,T}^{\bfeta }} \rho \bb{u} \cdot\partial_t \boldsymbol\xi + \int_{Q_{F,T}^{\bfeta }}(\rho \bb{u} \otimes \bb{u}):\nabla\boldsymbol\xi +\int_{Q_{F,T}^{\bfeta }}p(\rho)(\nabla \cdot \boldsymbol\xi)+\int_0^T \langle \pi,\nabla \cdot \boldsymbol\xi \rangle_{[\mathcal{M}^+,C](\partial N_\bfeta(t))}   \nonumber\\
    &&\quad-  \int_{Q_{F,T}^{\bfeta }} \mathbb{S}(\nabla\bb{u}): \nabla \boldsymbol\xi +\int_{Q_{S,T}} \partial_t {\bfeta }\cdot \partial_t \boldsymbol\phi  -\int_0^T\langle DE(\bfeta),\boldsymbol\phi \rangle-\int_0^T\langle D_2R(\bfeta,\partial_t \bfeta), \boldsymbol\phi \rangle \nonumber\\
    &&=\int_{\Omega^{\bfeta_0}}(\rho\mathbf{u})_0\cdot\boldsymbol\xi(0,\cdot)+ \int_{\Omega_S} \bb{v}_0 \cdot \boldsymbol\phi(0,\cdot)\nonumber \\ \label{momeqweak}
\end{eqnarray}%%----------------------------%% 
holds for all $\boldsymbol\xi \in C_c^\infty([0,T)\times \Omega)$, where $ \boldsymbol\phi= \boldsymbol\xi\circ \bfeta$ on\footnote{Note that $\boldsymbol\phi$ inherits the regularity from $\bfeta$ so $\boldsymbol\phi \in L^2(0,T; W^{2,q}(\Omega_S))\cap H^1(0,T; H^1(\Omega_S))$ and therefore all the integral terms are well-defined. See also \cite[Prop. A.4]{benevsova2020variational}.} $Q_{S,T}$. 
\item Energy inequality 
\begin{eqnarray}%%%%%%%%%%%%%
    &&\int_{\Omega_F^{\bfeta}(t)} \Big( \frac{1}{2} \rho |\bb{u}|^2 + \frac{\rho^\gamma}{\gamma-1}   \Big)(t) + \frac{1}{\gamma-1}\int_{\partial I_\bfeta(t)}d\pi(t) + \int_0^t\int_{\Omega_F^{\bfeta}(\tau)} (\mathbb{S}(\nabla \mathbf{u}):\nabla \mathbf{u})(\tau)~ d\tau \nonumber \\
    && \quad+  \frac{1}{2}\int_{\Omega_S} |\partial_t {\bfeta }|^2(t) + E(\bfeta(t)) + 2\int_0^t R(\bfeta,\partial_t \bfeta)(\tau)~d\tau \nonumber \\
    &&\leq \int_{\Omega_F^{\bfeta_0}} \Big( \frac{1}{2\rho_0} |(\rho\bb{u})_0|^2 + \frac{\rho_0^\gamma}{\gamma-1}   \Big) + \int_{\Omega_S} \frac{1}{2}|\bb{v}_0|^2 + E(\bfeta_0) \label{enineq}
    \end{eqnarray}%%----------------------------%%
holds for all $t\in (0,T]$.
\end{enumerate}

\end{mydef}

\begin{rem} %\label{rem:testCoupling}
(1) If there is no contact, i.e.\ $Q_\bfeta^C\cup Q_\bfeta^N=\emptyset$, then the weak solution enjoys better properties and coincides with the one in \cite{breit2021compressible}.\\
(2) Note that the defect measure $\pi$ is supported on $(0,T)\times\mathcal{M}^+(\partial N_\bfeta(t))$. It is a consequence of lack of uniform integrability of pressure near the domain irregularities caused by self-contact. While contact can appear with the rigid boundary $\partial \Omega$ which also breaks uniform integrability, the test functions are compactly supported and are thus zero on $\partial\Omega$, so the possible measures supported on $\partial\Omega$ are irrelevant. The physical interpretation of this is that there is no boundary condition on $\partial\Omega$ for the fluid stress, while on the fluid-structure interface there is the fluid-structure dynamic coupling $\eqref{dync}$ which needs to be preserved in the test functions. Let us point out that if the solution is smooth, both this condition and the structure-structure dynamic coupling $\eqref{dynca}$ can be recovered in the pointwise sense, and the defect measure $\pi$ vanishes. This is shown in Section $\ref{derivation}$.\\
(3) Let us point out that there is a possibility of vacuum appearing in weak solutions to our problem. There are a few downsides to it. First, the viscosity coefficients are constant which implies that even in vacuum there can be fluid velocity which satisfies the following ``ghost elliptic equation'' $\nabla\cdot\mathbb{S}(\nabla\bb{u})=0$. In our framework, this means that the vacuum can produce stress which acts onto the structure. One way to avoid all these issues is to use a model which has a density dependent viscosity coefficients $\mu=\mu(\rho)$ and $\lambda=\lambda(\rho)$ which vanish inside vacuum (we refer to \cite[Chapter 30]{handbook} for detailed coverage of this topic). However, the analysis of this fluid model has not been done on bounded domains because no boundary conditions conserve both energy and so called ``BD entropy'' - a quantity which provides additional information that ensures the convergence of pressure and stress. Note that in this case, the strong convergence of density based on the convergence of viscous effective flux (for the constant viscosity case) is out of reach, so the improved regularity coming from the BD entropy compensates for this deficiency.
\end{rem}

\begin{rem} \label{rem:testCoupling}
 The coupling of test functions between an Eulerian and a Lagrangian test-function, as we have done here in the weak solution leads to some consequences for the expected class of solutions. Indeed, at any pair $(x,y)$ of pre-images of a contact point, i.e.\ with $\bfeta(x) =\bfeta(y)$, the corresponding values $\boldsymbol\phi(x),\boldsymbol\phi(y)$ of the Lagrangian test-function must coincide with that of the Eulerian test function $\boldsymbol\xi(\bfeta(x))$ and thus with each other. As a result, at that point, it seems like it is not possible to discern any information about the relative behaviour of the different parts, in particular the forces expressed by the momentum balance.
 
 Indeed this has been a well-known problem in the case of rigid bodies immersed in a fluid, where this uncertainty can lead to non-uniqueness of solutions \cite{starovoitovNonuniquenessSolutionProblem2005}. However in our case, we are dealing with an elastic body. For a rigid body, the same argument as above immediately fixes the test-function on the whole body. In contrast, for an elastic body we have the freedom to choose the test-function differently everywhere else, which gives us full information about all the stresses and forces within the solid. And while those could potentially form concentrations at the boundary, the interface itself has zero mass and thus carries no momentum. That means that knowing the solid momentum in the bulk is indeed sufficient to determine also the behaviour of the interface. 
 
 %We would also like to highlight that the coupling of velocities might look like our contact is ``sticky'', i.e.\ once two parts of solid come into contact, the coupling condition require that they stay that way. However even neglecting issues, such as velocity being only meaningfully defined on sets of non-zero capacity, this is in fact not the case. \textcolor{red}{Consider the two functions $f(t):= t^2$ and $g(t) := 0$. Then $\partial_t f(t) = \partial_t g(t)$ whenever $f(t)=g(t)$, but the two meet and then separate. This can easily be extended into a comparable situation for our definition of weak solution. Thus the only consequence is that the contact is in some sense slow, which considering the fluid in between is the expected behaviour.}

 Additionally, while the velocities will coincide at the contact itself, this still allows for separation. While at first glance this seems counterintuitive, consider e.g.\ the 2 dimensional case where locally $\Omega$ and $\Omega_S$ consist of the upper half-space and $\eta(x)=(x_1,x_2+x_1^2+t^2)$. Then at time $t=0$, the fluid is in contact with the obstacle at point $(0,0)$ and their velocities align. Still there is separation for any future time and in particular it is possible to construct a matching velocity field $v\in L^2(0,T;W^{1,2}(\Omega_f^\bfeta))$ for the fluid with the right kinematic conditions. While this does not prove the actual existence of a corresponding solution to the equations, it at least shows that this is not prevented by our weak formulation.
\end{rem}

\section{Topological properties of fluid and structure domains and their interface for deformations with allowed contact}\label{sec:Top}
In order to study some topological and geometrical properties of the interface, first we prove a Lemma that gives us some nice properties of functions from $\mathcal{E}^q$.

\begin{lem}\label{lemma:contact1}
Let $\bfeta \in \mathcal{E}^q$. Then, one has:
\begin{enumerate}
    \item If $x,y \in C_\bfeta$ and $\bfeta(x)=\bfeta(y)$, then $x=y$, i.e. $\bfeta$ is injective on $C_\bfeta$;
    \item If $x,y,z\in N_\bfeta$ and $\bfeta(x)=\bfeta(y)=\bfeta(z)$, then at least one of $x=y$, $x=z$ and $y=z$ is true;
    \item $\partial\Omega_S^\bfeta=\bfeta(I_\bfeta)\cup \bfeta( C_\bfeta)\cup \bfeta( \partial N_\bfeta)$;
    \item $\partial \Omega_F^\bfeta=\bfeta(\overline{I_\bfeta})\cup \partial\Omega\setminus \bfeta(\operatorname{Int}(C_\bfeta))$;
    \item The fluid-structure interface equals $\bfeta(\overline{I_\bfeta})$, i.e. $\partial \Omega_S^\bfeta\cap \partial \Omega_F^\bfeta=\bfeta(\overline{I_\bfeta})$;
    \item Let $\bfeta_n\to \bfeta\; {\rm in}\; \mathcal{E}^q$. Then $|I_\bfeta|\leq \liminf\limits_{n\to \infty}|I_{\bfeta_n}|$.
\end{enumerate}
\end{lem}
\begin{proof}
\textbf{Claim 1.} Assume the opposite, i.e.\ there exists $x,y\in C_\bfeta$, $x\neq y$ such that $\bfeta(x)=\bfeta(y)\in\partial\Omega.$ Since $\bfeta\in\mathcal{E}^q$ we have $\det\nabla\bfeta>c$ and therefore there exists $\delta>0$ such that the restrictions $\bfeta_{| B_{\delta}(x)}$ and $\bfeta_{| B_{\delta}(y)}$ are each bijective. Moreover, since $\Omega$ is a $C^{1,\alpha}$ domain, there exists $l>0$ such that $B_l(\bfeta(x))\cap\partial\Omega$ can be represented as a graph of a function $\phi:V_l\to \mathbb{R}$, where $V_l$ is an open and connected set in $\mathbb{R}^2$, so that $\Omega$ lies below the graph of $\phi$. Now we consider sets $\partial\bfeta(B_{\delta}(x))$ and $\partial\bfeta(B_{\delta}(y))$. By assumption both sets contain point $\bfeta(x)$ and since $\bfeta\in C^1(\Omega_S,\R^3)$ tangents at that point coincide with the tangent to $\partial\Omega$. Therefore, by taking $\delta$ and $l$ smaller if needed, both $B_l(\bfeta(x))\cap\partial\bfeta(B_{\delta}(x))$ and $B_l(\bfeta(x))\cap\partial\bfeta(B_{\delta}(y))$ can be represented as graphs of functions $\phi_1,\phi_2:V_l \to \mathbb{R}$, respectively, in the same coordinate system such that $\phi(x_0)=\phi_1(x_0)=\phi_2(x_0)$, for some $x_0\in V_{l}$. Furthermore, since $\Omega_S^\bfeta\subseteq \Omega$ we have $\phi_1,\phi_2\leq \phi$ on $V_l$. However, this means that the area that lies below the function $\phi_m:=\min\{\phi_1,\phi_2\}$ belongs to $\bfeta (B_{\varepsilon}(x))\cap \bfeta (B_{\varepsilon}(y))$, which is a contradiction with $\bfeta \in \mathcal{E}^q$ (more precisely it is in contradiction with the Ciralet-Ne\v cas condition). \\

\noindent
\textbf{Claim 2.} This claim is similar to the first one (see also \cite[Theorem 2]{ciarletInjectivitySelfcontactNonlinear1987}). Assume the opposite, i.e. there exists $x,y,z\in N_\bfeta$ such that $x\neq y\neq z\neq x$ and $\bfeta(x)=\bfeta(y)=\bfeta(z)$. Again we can choose $\delta>0$ such that the corresponding restriction of $\bfeta$ to $B_{\delta}(x)$, $B_{\delta}(y)$ and $B_{\delta}(z)$ are bijective. By reasoning in the same way as in the proof of Claim 1, we conclude that $\partial\bfeta(B_\delta(x))$, $\partial\bfeta(B_\delta(y))$ and $\partial\bfeta(B_\delta(z))$ that can be represented as graphs of functions in some neighbourhood of $\bfeta(x)$ in the same coordinate system (notice that in this case $\Omega_S^\bfeta$ can lie on both sides of the graphs). However, this is again in contradiction with the Ciralet-Ne\v cas condition.\\

\noindent
\textbf{Claim 3.} First note that $\bfeta(C_\bfeta)\cup\bfeta(I_\bfeta)\subseteq\partial\Omega_S^\bfeta\subseteq\bfeta(\partial\Omega_S)=\bfeta(I_\bfeta)\cup\bfeta(C_\bfeta)\cup\bfeta(N_\bfeta)$. Therefore it is enough to prove $\partial\Omega_S^\bfeta\setminus\bfeta(I_\bfeta\cup C_\bfeta)=\bfeta(\partial N_\bfeta).$

Let $x\in\partial\Omega_S^\bfeta\setminus\bfeta(I_\bfeta\cup C_\bfeta)$. Then there exists $y\in N_\bfeta$ such that $\bfeta(y)=x.$ Moreover, by definition of $N_\bfeta$ and Claim 2, there exists unique $z\in N_\bfeta$ such that $\bfeta(y)=\bfeta(z)$. We will prove that at least one of $y,z$ is in $\partial N_\bfeta$. Assume the opposite, i.e. $y,z\in {\rm int}N_\bfeta$. Then there exist open sets $O_1$ and $O_2$ in $N_\bfeta$ such that $y\in O_1$ and $z\in O_2$ and $\bfeta(O_1)\subseteq \bfeta(O_2)$. In this case, we can represent $\bfeta(O_1)\cap B_\delta(\bfeta(x))$ as a graph of a function $\phi:V_\delta\to \mathbb{R}$, for $\delta>0$ small enough, so that above and below is $\Omega_S^\bfeta$. This means that $\bfeta(O_1)\cap B_\delta(\bfeta(x))\subseteq \text{int}(\bfeta(\overline{\Omega_S}))=\Omega_S^\bfeta$, which is a contradiction with $x\in \partial\Omega_S^\bfeta$. Therefore $x\in \bfeta(\partial N_\bfeta)$ and the first inclusion is proved.

Let $x\in \bfeta(\partial N_\bfeta)$ and let $z\in\partial N_\bfeta$ be such that $\bfeta(z)=x$. Then there is a sequence  $z_n\in I_\bfeta$ such that $z_n\to z$ (since otherwise $z_n \in C_\bfeta$ would imply that $z\in C_\bfeta$, due to closedness of $C_\bfeta$ which is proved in Theorem \ref{lemma:contact2} Claim 3). Then $\bfeta(z_n)\in\partial\Omega_S^\bfeta$ and $\bfeta(z_n)\to \bfeta(z)=x$. Since $\partial\Omega_S^\bfeta$ is closed we have $x\in \partial\Omega_S^\bfeta$. Therefore we have proved the second inclusion and thus the claim. \\

\noindent
\textbf{Claim 4.} First we prove inclusion $\bfeta(\overline{I_\bfeta})\cup (\partial\Omega\setminus \bfeta(\text{int}(C_\bfeta)))\subseteq\partial \Omega_F^\bfeta$. Since $\bfeta$ is an open mapping (because $\bfeta\in \mathcal{E}^q$) and $\partial\Omega_F^\bfeta$ is a closed set it is enough to prove 
$\bfeta(I_\bfeta)\cup (\partial\Omega\setminus \bfeta(C_\bfeta))\subseteq\partial \Omega_F^\bfeta$. Let $x\in \bfeta(I_\bfeta)\cup (\partial\Omega\setminus \bfeta(C_\bfeta))$ and denote $y = \bfeta(x)$. Then, for every $\delta< \text{dist}(\bfeta(x),\partial\Omega)$, one has $B_\delta(\bfeta(x))\setminus\Omega_S^\bfeta\neq \emptyset$ and $B_\delta(\bfeta(x))\cap \Omega^c=\emptyset$, which then implies that $(B_\delta(\bfeta(x))\setminus\Omega_S^\bfeta)\subseteq \overline{\Omega_F^\bfeta}$ (since  $\overline{\Omega} = \overline{\Omega_F^\bfeta}\cup \overline{\Omega_S^\bfeta}$). Therefore, in this case $\bfeta(x)\in \partial\Omega_F^\bfeta$. If on the other hand $y\in \partial\Omega\setminus \bfeta(C_\bfeta)$, then for every $\delta< \text{dist}(y,\partial\Omega_S^\bfeta)$, one has $B_\delta(y)\cap\Omega_S^\bfeta=\emptyset$ and $B_\delta(y)\cap \Omega\neq\emptyset$, so $B_\delta(y)\cap \Omega\subseteq \overline{\Omega_F^\bfeta}$ and consequently $y\in \partial\Omega_F^\bfeta$.

To prove other inclusion first we notice that $\partial\Omega_F^\bfeta\subseteq \partial\Omega\cup \bfeta(\partial\Omega_S)$. Let $x\in \partial\Omega_F^\bfeta$. If $x\in \partial\Omega$, by definition of $C_\bfeta$, one has $x\notin {\rm Int}C_{\bfeta}$. Similarly, by definition of $N_\bfeta$ it immediately follows $x\notin{\rm Int}N_\bfeta$. Therefore if $x\notin\partial\Omega$, we have $x\in \overline{I_\bfeta}$ which proves the second inclusion.\\

\noindent
\textbf{Claim 5.} Is a direct consequence of previous two claims.\\

\noindent
\textbf{Claim 6.}
First we prove $|\Omega_S\setminus C_\bfeta|\leq \limsup_{n}|\Omega_S\setminus C_{\bfeta_n}|$. Let $\chi:=\chi|_{\Omega_S\setminus C_\bfeta}$ and $\chi_n:=\chi|_{\Omega_S\setminus C(\bfeta_n)}$ be characteristic functions of sets $\Omega_S\setminus C(\bfeta_n)$ and $\Omega_S\setminus C(\bfeta_n)$, respectively.

Let $X\in{\rm int}(\Omega_S\setminus C(\bfeta_n))$. By definition of $C(\bfeta)$ it means that $d(\bfeta(x),\partial\Omega)>0.$ Morever, by uniform convergence of $\bfeta_n$ there exists $n_x\in\N$ such that $d(\bfeta_n(x),\partial\Omega_S)>0$, $n\geq n_x$, i.e. $1=\chi(x)=\chi_n(x)$, $n\geq n_x$. Therefore, we proved:
$$
(1-\chi)\leq \liminf\limits_{n\to\infty}(1-\chi_n)\quad {\rm a.\;e.\; in}\quad \Omega_S.
$$
Now, by Fatou's Lemma we have
$$
|\Omega\setminus C(\bfeta)|=\int_{\partial\Omega_S}(1-\chi)
\leq \int_{\partial\Omega_S}\liminf\limits_{n\to\infty}(1-\chi_n)
\leq \liminf\limits_{n\to\infty} \int_{\partial\Omega_S}(1-\chi_n)
=\liminf\limits_{n\to\infty} |\Omega\setminus C_{\bfeta_n}|.
$$
By working additionally with characteristic function for $\Omega_S\setminus N_\bfeta$ and $\Omega_S\setminus N_\bfeta$, we analogously prove $|\Omega_S\setminus N_\bfeta|\leq \limsup_{n}|\Omega_S\setminus N_{\bfeta_n}|$, so
\begin{align*}
|I_\bfeta|
&=|\Omega\setminus \left (C_\bfeta\cup N_\bfeta\right )|
\leq \liminf\limits_{n\to\infty} |\Omega\setminus \left (C_{\bfeta_n}\cup N_{\bfeta_n}\right )|
=\liminf\limits_{n\to\infty} |I_{\bfeta_n}|. \qedhere
\end{align*}
\end{proof}

Now we are in the position to prove a theorem that gives us some geometrical properties of the fluid and structure domains and their interface,  associated to a weak solution.
\begin{thm}\label{lemma:contact2}
Let $(\bfeta,\rho,\bb{u})$ satisfy the energy inequality $\eqref{enineq}$ and let $\int_{\Omega_F^\bfeta(t)} \rho(t)=m>0$ for all $t\in [0,T]$. Then, one has:
\begin{enumerate}
    \item  $|\Omega_F^\bfeta(t)|\geq c(E_0,m,\gamma)>0$ and $|\Omega_S^\bfeta(t)|\geq c(E_0)>0$, for all $t\in[0,T]$;
     \item $|\partial\Omega_F^\bfeta(t)|\geq c(E_0,m,\gamma)>0$ and $|\partial\Omega_S^\bfeta(t)|\geq c(E_0)>0$, for all $t\in[0,T]$;
     \item $Q_\bfeta^I$ is open and $Q_\bfeta^C,Q_\bfeta^N$ are closed in $[0,T]\times \partial\Omega_S$;
    \item There exists a constant $c=c(E_0)$ independent of $\bfeta$ such that we have the following lower bound on the area of the fluid-structure interface: $|\partial\Omega_F^\bfeta(t)\cap \partial\Omega_S^\bfeta(t)|\geq c(E_0),\quad t\in [0,T].$
\end{enumerate}

\end{thm}
\begin{proof}

\noindent
\textbf{Claim 1.}
Since mass of the fluid is fixed, one has
\begin{eqnarray*}%%%%%%%%%%%%%
    m=||\rho||_{L^1(\Omega_F^\bfeta(t))}\leq ||\rho||_{L^\gamma(\Omega_F^\bfeta(t))}||1||_{L^\frac{\gamma}{\gamma-1}(\Omega_F^\bfeta(t))}\leq E_0^{\frac1\gamma}|\Omega_F^\bfeta|^{\frac{\gamma-1}{\gamma}},
\end{eqnarray*}%%----------------------------%%
so
\begin{eqnarray*}%%%%%%%%%%%%%
    |\Omega_F^\bfeta(t)|\geq \frac{m^{\frac{\gamma}{\gamma-1}}}{E_0^{\frac1{\gamma-1}}},
\end{eqnarray*}%%----------------------------%%
for all $t\in [0,T]$. Note that $E_0>0$ since $m>0$. Now, since $\det\nabla\bfeta\geq c$ on $[0,T]\times \Omega_S$, one has $|\Omega_S^\bfeta(t)| = \int_{\Omega_S}\det\nabla\bfeta(t)\geq c |\Omega_S|>0$.\\

\noindent
\textbf{Claim 2.} First, since $\text{int}(I_\bfeta(t))$ and $\partial\Omega\setminus C_\bfeta(t)$ have a 
well-defined surface element, by Lemma $\ref{lemma:contact1}$ Claim 3 and continuity, one has that $\partial\Omega_F^\bfeta(t)$ has a well-defined surface element everywhere and therefore it is a measurable surface for all $t\in[0,T]$ (and similarly for the surface $\partial\Omega_S^\bfeta(t)$). From the isoperimetric inequality we have
\begin{eqnarray*}%%%%%%%%%%%%%
    |\partial\Omega_F^\bfeta(t)|\geq 3 |\Omega_F^\bfeta(t)|^{\frac23}\left (\frac{4}{3}\pi\right )^{\frac13}\geq c(E_0,m,\gamma),
\end{eqnarray*}%%----------------------------%%
and similarly for $|\partial\Omega_S^\bfeta(t)|$. \\

\noindent
\textbf{Claim 3.} Closedness of $Q^C_\bfeta$ is a direct consequence of continuity of $\bfeta$. Let $(t_n,x_n)$ be a convergent sequence in $Q^C_\bfeta$ and $(t,x)$ its limit. Then we have
\begin{eqnarray*}%%%%%%%%%%%%%
    0=\lim\limits_{n\to\infty} \text{dist}(\bfeta(t_n,x_n),\partial\Omega)=\text{dist}(\bfeta(\lim\limits_{n\to\infty} (t_n,x_n)),\partial\Omega) = \text{dist}(\bfeta(t,x),\partial\Omega),
\end{eqnarray*}%%----------------------------%%
and therefore $(t,x)\in Q^C_\bfeta.$

To prove that $Q^N_\bfeta$ is closed, we take a convergent sequence $(t_n,x_n)$ in $Q^N_\bfeta$ and $(t,x)=\lim\limits_{n\to\infty} (t_n,x_n)$. By Lemma \ref{lemma:contact1}, Claim 2), and the definition of $Q^N_\bfeta$ for every $n$ there exists a unique $y_n\in N_{\bfeta(t_n)}$ such that $y_n\neq x_n$ and $\bfeta(t_n,x_n)=\bfeta(t_n,y_n)$. Since $y_n$ is a sequence in $\partial\Omega_S$ which is compact, there exists a convergent subsequence which we will still denote by $y_n$. 
By the continuity of $\bfeta$, one concludes
\begin{eqnarray*}%%%%%%%%%%%%%
    \bfeta(t,x)=\bfeta\left(\lim\limits_{n\to\infty} (t_n,x_n)\right) = \bfeta\left(\lim\limits_{n\to\infty} (t_n,y_n)\right) = \bfeta(t,y)
\end{eqnarray*}%%----------------------------%%
where $y:=\lim\limits_{n\to\infty} y_n$. Since $\bfeta$ is locally bijective for on balls of uniform radius $\varepsilon$, one must have $|x_n-y_n|\geq \varepsilon$ for all $n\geq n_0\in \mathbb{N}$, so $x\neq y$, which proves that $Q_\bfeta^N$ is closed. Finally, for the last set, since $Q_\bfeta^C\cup Q_\bfeta^N$ is now closed, one gets that $Q_\bfeta^T=[0,T]\times\partial\Omega_S\setminus(Q_\bfeta^C\cup Q_\bfeta^N)$ is open. \\

\noindent
\textbf{Claim 4.} Let $X=\{\bfeta\in \mathcal{E}^q:\|\bfeta\|_{W^{2,q}(\Omega_S)}\leq E_0, ~|\Omega_S^\bfeta|\leq C(E_0)<|\Omega|\}$. First, note that by Lemma \ref{lemma:contact1}, Claim 5 we have
$$
\partial\Omega_F^\bfeta\cap \partial\Omega_S^\bfeta=\bfeta(\overline{I_\bfeta}).
$$
Moreover, there exists a constant $C_0$ such that
$$
|\bfeta(I_\bfeta)|=\left |\int_{I_\bfeta}\det\nabla\bfeta|\nabla\bfeta^{-\tau}{\bf n}|\right |\geq C_0|I_\bfeta|, \quad\bfeta\in X.
$$
Therefore, since $|\bfeta(\overline{I_\bfeta})| \geq |\bfeta(I_\bfeta)|$, it is enough to prove that there exists a constant $c=c(E_0)$ such that 
$$
|I_\bfeta|\geq c(E_0)>0, \quad \text{for all } \bfeta\in X.
$$
%$$
%|\partial\Omega_F^\bfeta\cap \partial\Omega_S^\bfeta|\geq C(E_0).
%$$
Assume the opposite, i.e. there exists a sequence $\bfeta_n\in X$ such that $|I_{\bfeta_n}|\to 0.$ 
By Lemma \ref{lemma:contact1}, Claim 6 we have:
$$
|I_\bfeta|\leq \liminf\limits_{n\to\infty} |I_{\bfeta_n}|=0.
$$
which is a contradiction, since $I_\bfeta$ is an open set and it must have a positive measure. 
%Therefore, $|\partial\Omega_F^\bfeta\cap \partial\Omega_S^\bfeta|=0$ which contradicts Lemma \ref{lemma:contact1}, Claim 6.
\end{proof}

\subsection{Closed components of the the fluid domain: a surprising example}\label{rem:cantor}
%\subsection{Number of closed components of the fluid domain and some unusual examples} \label{rem:cantor}
 While we have all these conditions, the boundary of the fluid region can still be highly irregular. Assuming sufficient regularity of the fixed reference domain $\Omega_S$ and finite elastic energy, the surface deformation $\bfeta|_{\partial \Omega_S}$ will be a regular (i.e.\ local diffeomorphism) $C^{1,\alpha}$-map of a $C^{1,\alpha}$-domain. However it will not be injective or even an immersion.
 
 Here, in particular the latter is what makes the situtation interesting, but difficult. While it can be already seen from simple two dimensional examples (cmp. Figure \ref{fig:my_label}), that this will lead to irregularities in the fluid domain, this still leaves the question of how many of these irregularities can appear. In other words, what is the cardinality and the measure of $\partial I_\bfeta \subset \partial \Omega_S$? The somewhat surprising answer is that even in an otherwise benign situation, this can be a set of full, i.e.\ $d-1=2$-dimensional measure.
 
 To see this, let us first note, that locally the fluid domain always looks like one of three possible configurations: the bulk, the area above a $C^{1,\alpha}$-graph or the area between two such graphs. The former two are well understood. Only in the latter configuration can these cusps appear. For our example we can further simplify the situation by considering a problem with a vertical symmetry, i.e.\ the functions defining the two graphs differ by only a sign. Then the question reduces to the following: Given a smooth, non-negative function $f:Q \subset \mathbb{R}^{2} \to \mathbb{R}$, what can (locally) be said about the size of the set $\partial \{x \in Q: f(x) > 0\}$?
 
 It turns out, that the answer is, not much, as the following Cantor-type construction shows. Take any fixed non-negative function $g: C^{\infty}(\mathbb{R})$ with $\supp g = [-1,1]$. Then for a sequence of vertical and horizontal scalings $(a_k)_{k\in\mathbb{N}},(b_k)_{k\in\mathbb{N}}$, for each $k$, we add $2^k$ copies of $x\mapsto a_k g(b_kx)$ to our function, shifted in such a way that they sit precisely in the gaps of all the previous copies (see Figure \ref{fig:cantor}). The final result of this construction will be our function $f$. As the choice of $a_k$ does not influence the set $\partial \{x \in Q: f(x) > 0\}$, this can be done in a way such that $f$ is smooth and or arbitrary small norm. If we now choose $b_k = \frac{1}{4^k}$, it is easy to show that $\partial \{x \in Q: f(x) > 0\}$ ends up being the well known ``fat''-Cantor set, which has non-zero $1$-dimensional measure. Similar constructions can be done in higher dimensions, where it is even possible for the resulting fluid domain to still be connected.
 
 Nevertheless, one should keep in mind that such a situation is not necessarily unnatural. If two mostly flat elastic solids (or parts of the same solid) are pushed against each other in a fluid, then small inclusions of the fluid between them are to be expected and there is no a-priori way to bound their size or number. In fact, the solid itself could already have this type of roughness in its relaxed configuration. It is thus impossible to exclude such a behaviour in the generic setup that we study here.

\begin{figure}[ht]
 \begin{center}
 \begin{tikzpicture}
  \draw (-2,0) .. controls +(.5,0) and +(-.5,0) .. (0,.7) .. controls +(.5,0) and +(-.5,0) .. (2,0);
  \draw (-2,0) .. controls +(.5,0) and +(-.5,0) .. (0,-.7) .. controls +(.5,0) and +(-.5,0) .. (2,0);
  \foreach \a in {-1,1} {
    \begin{scope}[shift={(4*\a,0)}, scale=.333]
      \draw (-2,0) .. controls +(.5,0) and +(-.5,0) .. (0,.7) .. controls +(.5,0) and +(-.5,0) .. (2,0);
      \draw (-2,0) .. controls +(.5,0) and +(-.5,0) .. (0,-.7) .. controls +(.5,0) and +(-.5,0) .. (2,0);
      \foreach \b in {-1,1} {
        \begin{scope}[shift={(4*\b,0)}, scale=.333]
          \draw (-2,0) .. controls +(.5,0) and +(-.5,0) .. (0,.7) .. controls +(.5,0) and +(-.5,0) .. (2,0);
          \draw (-2,0) .. controls +(.5,0) and +(-.5,0) .. (0,-.7) .. controls +(.5,0) and +(-.5,0) .. (2,0);
            \foreach \c in {-1,1} {
            \begin{scope}[shift={(4*\c,0)}, scale=.333]
            \draw (-2,0) .. controls +(.5,0) and +(-.5,0) .. (0,.7) .. controls +(.5,0) and +(-.5,0) .. (2,0);
            \draw (-2,0) .. controls +(.5,0) and +(-.5,0) .. (0,-.7) .. controls +(.5,0) and +(-.5,0) .. (2,0);
            \draw (-6,0) -- (-2,0); \draw (2,0)--(6,0);
            \end{scope}
        }
        \end{scope}
      }
    \end{scope}
 }
 \end{tikzpicture}
 \end{center}

 \caption{A sketch of the construction given above. Vertical and horizontal scaling of the individual domain components are exagarated for better visibility.}
 \label{fig:cantor}
\end{figure}
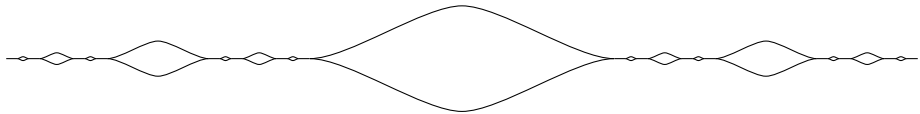

\section{Formal equivalence between weak and smooth solutions}\label{form:eq}
In this section, the goal is to show that our concept of weak solution is justified. This will be done in two steps. First we will show that all smooth solutions $\eqref{elasticity}-\eqref{initialdata}$ are in fact weak solutions in the sense of Definition $\ref{weaksol}$. Then, we will show that a weak solution $(\rho,\bb{u},\bfeta)$, which is smooth, with defect measure $\pi$, satisfies the problem $\eqref{elasticity}-\eqref{initialdata}$ pointwise and that the defect measure vanishes.

\subsection{Derivation of weak formulation for smooth solutions}\label{derivation}
Let $(\rho,\bb{u},\bfeta)$ be a smooth solution to the problem $\eqref{elasticity}-\eqref{initialdata}$. The goal is to prove that this is also a weak solution in the sense of Definition \ref{weaksol}. Let us first observe the fluid domain $\Omega_F^\bfeta$. Due to Lemma $\ref{lemma:contact1}$ Claim 4, one has that $\partial \Omega_F^\bfeta=\bfeta(\overline{I_\bfeta})\cup \partial\Omega\setminus \bfeta(\text{Int}(C_\bfeta))$. Assuming that $|\bfeta(\overline{I_\bfeta})| = |\bfeta(I_\bfeta)$| and $|\partial\Omega\setminus \bfeta(\text{Int}(C_\bfeta))| = |\partial\Omega\setminus \bfeta(C_\bfeta)|$, one has that both $\bfeta(I_\bfeta)$ and $\partial\Omega\setminus \bfeta(C_\bfeta)$ have a well-defined trace and normal vector, so one can use the divergence theorem on $\Omega_F^\bfeta$
\begin{eqnarray*}%%%%%%%%%%%%%
    \int_{\Omega_F^\bfeta} \nabla \cdot\bb{f} = \int_{\bfeta(I_\bfeta)} \bb{f} \cdot \bb{n} dS+\int_{\partial\Omega\setminus \bfeta(C_\bfeta)}\bb{f} \cdot \bb{n}^\bfeta dS^\bfeta,
\end{eqnarray*}%%----------------------------%%
where $\bb{n}^\bfeta$ and $dS^\bfeta$ are the normal vector and the surface element on $\partial \Omega_S^\bfeta$, respectively. On the other hand, for any function $\bb{f}\in C^{\infty}([0,T]\times \Omega)$, one has the Reynolds transport theorem on $\Omega_S^\bfeta(t)$
\begin{eqnarray*}%%%%%%%%%%%%%
    \frac{d}{dt}\int_{\Omega_S^\bfeta(t)} \bb{f} = \int_{\Omega_S^\bfeta(t)}\partial_t\bb{f}+\int_{\partial\Omega_S^\bfeta(t)} \bb{f} (\partial_t\bfeta \cdot \bb{n}^\bfeta)\circ \bfeta^{-1}dS^\bfeta,
\end{eqnarray*}%%----------------------------%%
so
\begin{eqnarray*}%%%%%%%%%%%%%
     &&\frac{d}{dt}\int_{\Omega_F^\bfeta(t)} \bb{f} = \frac{d}{dt}\int_{\Omega} \bb{f} -  \frac{d}{dt}\int_{\Omega_S^\bfeta(t)} \bb{f}\\
     &&= \int_{\Omega}\partial_t\bb{f}- \int_{\Omega_S^\bfeta(t)}\partial_t\bb{f}-\int_{\partial\Omega_S^\bfeta(t)} \bb{f} (\partial_t\bfeta \cdot \bb{n}^\bfeta)\circ \bfeta^{-1} dS^\bfeta \\
    && = \int_{\Omega_F^\bfeta(t)}\partial_t\bb{f}-\int_{\partial\Omega_S^\bfeta(t)} \bb{f} (\partial_t\bfeta \cdot \bb{n}^\bfeta)\circ \bfeta^{-1} dS^\bfeta.
\end{eqnarray*}%%----------------------------%%
Finally, since $\partial_t \bfeta(t,x)=\partial_t \bfeta(t,y)$ whenever $\bfeta(t,x)=\bfeta(t,y)$, the last integral vanishes on $N_\bfeta(t)$, so one concludes that the Reynolds transport theorem holds on $\Omega_F^\bfeta(t)$ in the following form
\begin{eqnarray}%%%%%%%%%%%%%
    &\frac{d}{dt}\int_{\Omega_F^\bfeta(t)} \bb{f} = \int_{\Omega_F^\bfeta(t)}\partial_t\bb{f}-\int_{\bfeta(t,I_\bfeta(t))} \bb{f} (\partial_t\bfeta \cdot \bb{n}^\bfeta)\circ \bfeta^{-1} dS^\bfeta \label{reynolds:omega:f}
\end{eqnarray}%%----------------------------%%
for all functions $\bb{f}$ defined on $(0,T)\times \Omega$ for which the above terms are well-defined. We will assume that both $\rho$ and $\rho \bb{u}$ can be represented as restrictions of functions defined on $(0,T)\times \Omega$ which are regular enough so that $\eqref{reynolds:omega:f}$ holds for them.

Multiplying the fluid momentum equation with a test function $\boldsymbol\xi \in C_0^\infty([0,T)\times \overline{\Omega})$, and then integrating over $(0,T)\times \Omega_F^\bfeta(t)$ gives us
\begin{eqnarray}%%%%%%%%%%%%%
    \int_0^T \int_{\Omega_F^\bfeta(t)}\partial_t (\rho\mathbf{u})\cdot \boldsymbol\xi + \int_0^T\int_{\Omega_F^\bfeta(t)} (\nabla \cdot (\rho\mathbf{u}\otimes \mathbf{u}))\cdot \boldsymbol\xi &=&  \int_0^T \int_{\Omega_F^\bfeta(t)}(\nabla \cdot\mathbb{S}(\vu))\cdot \boldsymbol\xi-\int_0^T \int_{\Omega_F^\bfeta(t)} \nabla\rho^\gamma \cdot \boldsymbol\xi. \nonumber\\
    \label{fl}
\end{eqnarray}%%----------------------------%%
First, by $\eqref{reynolds:omega:f}$, one has
\begin{eqnarray*}%%%%%%%%%%%%%
    &&\int_0^T \int_{\Omega_F^\bfeta(t)}\partial_t (\rho\mathbf{u})\cdot \boldsymbol\xi = \int_0^T \int_{\Omega_F^\bfeta(t)}\partial_t (\rho\mathbf{u}\cdot \boldsymbol\xi) - \int_0^T \int_{\Omega_F^\bfeta(t)}\rho\mathbf{u}\cdot \partial_t \boldsymbol\xi\\
    &&=\int_0^T\underbrace{\frac{d}{dt}\int_{\Omega_F^\bfeta(t)} \rho\mathbf{u}\cdot \boldsymbol\xi}_{\mathclap{=-\int_{\Omega_F^\bfeta(t)} (\rho\mathbf{u})_0\cdot \boldsymbol\xi(0,\cdot)}}+\int_0^T\int_{\bfeta(t,I_\bfeta(t))}\rho\mathbf{u}\cdot \boldsymbol\xi (\partial_t\bfeta \cdot \bb{n})\circ \bfeta^{-1} dS -\int_0^T \int_{\Omega_F^\bfeta(t)}\rho\mathbf{u}\cdot \partial_t \boldsymbol\xi.
\end{eqnarray*}%%----------------------------%%
Next
\begin{eqnarray*}%%%%%%%%%%%%%
    &&\int_0^T\int_{\Omega_F^\bfeta(t)} (\nabla \cdot (\rho\mathbf{u}\otimes \mathbf{u}))\cdot \boldsymbol\xi\\
    &&=    \int_0^T\int_{\bfeta(t,I_\bfeta(t))}  (\rho\mathbf{u}\cdot \boldsymbol\xi) (\bb{u}\cdot \bb{n}^\bfeta)dS + \underbrace{\int_0^T\int_{\partial\Omega\setminus \bfeta(t,C_\bfeta(t))}(\rho\mathbf{u}\cdot \boldsymbol\xi) (\bb{u}\cdot \bb{n})dS}_{=0} -    \int_0^T\int_{\Omega_F^\bfeta(t)}  (\rho \bb{u}\otimes \bb{u}):\nabla\boldsymbol\xi
\end{eqnarray*}%%----------------------------%%
since $\boldsymbol\xi=0$ on $\partial\Omega$, so combining the above identites and using the kinematic fluid-structure coupling condition $\eqref{kinc}$ to cancel the boundary terms, one obtains
\begin{eqnarray}%%%%%%%%%%%%%
    &&\int_0^T \int_{\Omega_F^\bfeta(t)}\partial_t (\rho\mathbf{u})\cdot \boldsymbol\xi+ \int_0^T\int_{\Omega_F^\bfeta(t)} (\nabla \cdot (\rho\mathbf{u}\otimes \mathbf{u}))\cdot \boldsymbol\xi \nonumber \\
    &&= -\int_{\Omega_F^\bfeta(t)} (\rho\mathbf{u})_0\cdot \boldsymbol\xi(0,\cdot) -\int_0^T \int_{\Omega_F^\bfeta(t)}\rho\mathbf{u}\cdot \partial_t \boldsymbol\xi-    \int_0^T\int_{\Omega_F^\bfeta(t)}  (\rho \bb{u}\otimes \bb{u}):\nabla\boldsymbol\xi.
    \label{fl:calc0}
\end{eqnarray}%%----------------------------%%
Next, one has
\begin{eqnarray}%%%%%%%%%%%%%
    &&\int_0^T \int_{\Omega_F^\bfeta(t)}(\nabla \cdot\mathbb{S}(\vu))\cdot \boldsymbol\xi-\int_0^T \int_{\Omega_F^\bfeta(t)} \nabla\rho^\gamma \cdot \boldsymbol\xi \nonumber\\
    &&= \int_0^T \int_{\bfeta(t,I_\bfeta(t))}(\mathbb{S}(\vu)\bb{n}^\bfeta)\cdot \boldsymbol\xi dS^\bfeta + \underbrace{\int_0^T\int_{\partial\Omega\setminus \bfeta(t,C_\bfeta(t))}(\mathbb{S}(\vu)\bb{n})\cdot \boldsymbol\xi dS}_{=0} \nonumber\\
    &&\quad- \int_0^T \int_{\bfeta(I_\bfeta(t))}\rho^\gamma\bb{n}^\bfeta\cdot \boldsymbol\xi dS^\bfeta- \underbrace{\int_0^T\int_{\partial\Omega\setminus \bfeta(C_\bfeta)}\rho^\gamma\bb{n}\cdot \boldsymbol\xi dS}_{=0} \nonumber\\ 
    &&\quad -\int_0^T \int_{\Omega_F^\bfeta(t)} \mathbb{S}(\nabla\bb{u}): \nabla \boldsymbol\xi +\int_0^T \int_{\Omega_F^\bfeta(t)}\rho^\gamma(\nabla \cdot \boldsymbol\xi)\nonumber\\
    &&=\int_0^T \int_{I_\bfeta(t)} \det \nabla\bfeta\left(\left [\mathbb{T}(\vu,p)\circ\bfeta\right ]\nabla\bfeta^{-\tau}\bb{n}\right)\cdot \boldsymbol\xi\circ\bfeta dS -\int_0^T \int_{\Omega_F^\bfeta(t)} \mathbb{S}(\nabla\bb{u}): \nabla \boldsymbol\xi +\int_0^T \int_{\Omega_F^\bfeta(t)}\rho^\gamma(\nabla \cdot \boldsymbol\xi), \nonumber \\ &&\label{fl:calc} 
\end{eqnarray}%%----------------------------%%
where the last term is obtained by a coordinate transformation by composing with $\bfeta^{-1}$. Next, multiplying the structure momentum equation with $\boldsymbol\phi:=\boldsymbol\xi \circ \bfeta$ and integrating over $(0,T)\times \Omega_S$ gives us
\begin{eqnarray}%%%%%%%%%%%%%
    &&\int_0^T \int_{\Omega_S} \partial_t^2 \bfeta \boldsymbol\phi - \int_0^T \int_{\Omega_S} \nabla \cdot \boldsymbol\sigma (\bfeta,\partial_t\bfeta)\cdot\boldsymbol\phi \nonumber \\
    &&= - \int_0^T \int_{\Omega_S}\partial_t \bfeta\cdot \partial_t \boldsymbol\phi +\int_{\Omega_S} \bb{v}_0 \cdot\boldsymbol\phi(0,\cdot) - \int_0^T \int_{\partial \Omega_S} (\boldsymbol\sigma (\bfeta,\partial_t\bfeta)\bb{n})\cdot \boldsymbol\phi+ \underbrace{\int_0^T \int_{\Omega_S}\boldsymbol\sigma (\bfeta,\partial_t\bfeta):\nabla\boldsymbol\phi}_{=\int_0^T\langle DE(\bfeta),\boldsymbol\phi \rangle+ \int_0^T\langle D_2R(\bfeta,\partial_t \bfeta), \boldsymbol\phi \rangle}. \nonumber \\
    \label{str:calc}
\end{eqnarray}%%----------------------------%%
by $\eqref{2nd:order:bnd}$. Noticing that
\begin{eqnarray*}%%%%%%%%%%%%%
    &&\int_{\partial \Omega_S} (\boldsymbol\sigma (\bfeta,\partial_t\bfeta)\bb{n})\cdot \boldsymbol\phi\\
    &&= \int_{I_\bfeta(t)} (\boldsymbol\sigma (\bfeta,\partial_t\bfeta)\bb{n})\cdot \boldsymbol\phi dS + \int_{N_\bfeta(t)} (\boldsymbol\sigma (\bfeta,\partial_t\bfeta)\bb{n})\cdot \boldsymbol\phi dS+ \underbrace{\int_{C_\bfeta(t)} (\boldsymbol\sigma (\bfeta,\partial_t\bfeta)\bb{n})\cdot \boldsymbol\phi dS}_{=0},
\end{eqnarray*}%%----------------------------%%
where the second term on RHS also vanishes due to Lemma $\ref{lemma:contact1}$ Claim 2 and the structure-structure coupling $\eqref{dynca}$, and that
\begin{eqnarray*}%%%%%%%%%%%%%
    \int_{I_\bfeta} (\boldsymbol\sigma (\bfeta,\partial_t\bfeta)\bb{n})\cdot \boldsymbol\phi dS = \int_{I_\bfeta} \det \nabla\bfeta\left [\mathbb{T}(\vu,p)\circ\bfeta\right ]\nabla\bfeta^{-\tau}\bb{n} dS ,
\end{eqnarray*}%%----------------------------%%
by the dynamic fluid-structure coupling $\eqref{dync}$, one can use the identities $\eqref{fl:calc0}$, $\eqref{fl:calc}$ and $\eqref{str:calc}$ in $\eqref{fl}$ to obtain the coupled momentum equation $\eqref{momeqweak}$.

\bigskip

Next, the renormalized continuity equation $\eqref{reconteqweak}$ in obtained in a similar fashion, i.e.\ we multiply $\eqref{conteq}$ with $(B(\rho)+\rho B'(\rho))\varphi$, then integrate over $(0,T)\times \Omega_F^\bfeta(t)$ and use the integration by parts and Reynolds transport theorem $\eqref{reynolds:omega:f}$.

\bigskip 

Finally, the energy (in)equality is obtained by choosing 
\begin{eqnarray*}%%%%%%%%%%%%%
    \boldsymbol\xi =       \begin{cases}        \bb{u},& \quad \text{ on } (0,T)\times\Omega_F^\bfeta(t),\\
        \partial_t \bfeta \circ \bfeta^{-1},&\quad  \text{ on } (0,T)\times\Omega_S^\bfeta(t),
    \end{cases}
\end{eqnarray*}%%----------------------------%%
in $\eqref{momeqweak}$, then $B(\rho)=\frac{1}{\gamma-1}\rho^\gamma,~ b(\rho) = \rho^\gamma$ and $\varphi=1$ in $\eqref{reconteqweak}$, and $B(\rho)=1,~ b(\rho)=0$ and $\varphi = -\frac{1}{2}|\bb{u}|^2$ and then summing up these 3 identities.

\subsection{Consistency}
Let us now assume that $({\bfeta },\rho,\bb{u})$ is a weak solution in the sense of Definition $\eqref{weaksol}$ which is smooth and let $\pi\in L_{w^*}^\infty(0,T; \mathcal{M}^+(\partial N_\bfeta(t)))$ be the corresponding defect measure. Now, by choosing test functions $\boldsymbol\xi$ with compact supports in $(0,T)\times\Omega_F^\bfeta(t)$ and $(0,T)\times\Omega_S^\bfeta(t)$, one directly recovers the momentum equations $\eqref{momeq}$ and $\eqref{elasticity}$. The fluid-structure and structure-structure kinematic coupling conditions $\eqref{kinc}$ and $\eqref{kinca}$ are automatically satisfied by definition. By similar calculation as above, one can conclude that 
\begin{eqnarray}%%%%%%%%%%%%%
    &&\int_0^T \int_{I_\bfeta(t)} \det \nabla\bfeta\left(\left [\mathbb{T}(\vu,p)\circ\bfeta\right ]\nabla\bfeta^{-\tau}\bb{n}\right)\cdot \boldsymbol\phi dS+\int_0^T \langle \pi,\nabla \cdot \boldsymbol\xi \rangle_{[\mathcal{M}^+,C](\partial I_\bfeta(t))}\nonumber\\
    &&\quad-  \int_0^T\int_{I_\bfeta(t)} (\boldsymbol\sigma (\bfeta,\partial_t\bfeta)\bb{n})\cdot \boldsymbol\phi -   \int_0^T\int_{N_\bfeta(t)} (\boldsymbol\sigma (\bfeta,\partial_t\bfeta)\bb{n})\cdot \boldsymbol\phi =0 \label{test:eq}
\end{eqnarray}%%----------------------------%%
for any $\boldsymbol\xi \in C_0^\infty([0,T)\times \overline{\Omega})$, where $\boldsymbol\phi = \boldsymbol\xi\circ\bfeta$. Now, for every $t_0\in [0,T]$ and $x_0\in I_\bfeta(t_0)$, there is an interval $I\ni t_0$ and a ball $B=B(\bfeta(t_0,x_0),r)$, for some $r>0$, such that for all $t\in I$, one has $B\cap \partial\Omega_F^\bfeta(t)\subset \bfeta(t,I_\bfeta(t))$. Then, we define a sequence of smooth test functions $\psi_n$ supported on $I\times B$ so that $\psi_n \to \delta_{(t_0,x_0)}$, so by testing $\eqref{test:eq}$ with $\psi_n \boldsymbol\xi$ and letting $n\to \infty$, one has that $\left(\det \nabla\bfeta\left(\left [\mathbb{T}(\vu,p)\circ\bfeta\right ]\nabla\bfeta^{-\tau}\bb{n}\right)\cdot \boldsymbol\phi\right)(t_0,x_0) = \left((\boldsymbol\sigma (\bfeta,\partial_t\bfeta)\bb{n})\cdot \boldsymbol\phi \right)(t_0,x_0)$. Since $(t_0,x_0)$ and $\boldsymbol\xi$ were arbitrary, one obtains that the fluid-structure dynamic condition $\det \nabla\bfeta\left(\left [\mathbb{T}(\vu,p)\circ\bfeta\right ]\nabla\bfeta^{-\tau}\bb{n}\right) = \boldsymbol\sigma (\bfeta,\partial_t\bfeta)\bb{n}$ holds on $(0,T)\times I_\bfeta(t)$, and by continuity on $[0,T]\times\overline{I_\bfeta(t)}$. By a similar approach, one can also conclude that the structure-structure dynamic condition $\eqref{dynca}$ holds on $[0,T]\times N_\bfeta(t)$ (here the key point is that due to the coupling $\boldsymbol\phi = \boldsymbol\xi\circ\bfeta$ on $\Omega_S$, one has that $\boldsymbol\phi(t,x) = \boldsymbol\phi(t,y)$ whenever $\bfeta(t,x)=\bfeta(t,y)$), so one has that all $(\bfeta,\rho,\bb{u})$ satisfy $\eqref{elasticity}-\eqref{initialdata}$ pointwise. This implies that we can show that the energy identity (without the defect measure) holds
\begin{eqnarray*}%%%%%%%%%%%%%
    &&\int_{\Omega_F^{\bfeta}(t)} \Big( \frac{1}{2} \rho |\bb{u}|^2 + \frac{\rho^\gamma}{\gamma-1}   \Big)(t)+ \int_0^t\int_{\Omega_F^{\bfeta}(\tau)} (\mathbb{S}(\nabla \mathbf{u}):\nabla \mathbf{u})(\tau) ~d\tau \nonumber \\
    &&\quad +  \frac{1}{2}\int_{\Omega_S} |\partial_t {\bfeta }|^2(t) + E(\bfeta(t)) +\int_0^t 2R(\bfeta,\partial_t \bfeta)(\tau)~d\tau \nonumber \\
    &&= \int_{\Omega_F^{\bfeta_0}} \Big( \frac{1}{2\rho_0} |(\rho\bb{u})_0|^2 + \frac{\rho_0^\gamma}{\gamma-1}   \Big) + \int_{\Omega_S} \frac{1}{2}|\bb{v}_0|^2 + E(\bfeta_0), 
    \end{eqnarray*}%%----------------------------%%
for all $t\in (0,T]$, which compared to the energy inequality $\eqref{enineq}$, by taking into consideration that $\pi \geq 0$, implies that $\pi \equiv 0$.

\section{Construction of approximate solutions}\label{sec:app:constr}
The goal of this section is to obtain approximate solutions to our problem which depend on given parameters $h,\varsigma,\varepsilon>0$, where $h = T/N$ for $N\in \mathbb{N}$. As these parameters vanish (in certain order), we will prove that the approximations converge to a weak solution in sense of Definition $\ref{weaksol}$. \\

\subsection{Time-stepping decoupling approximation scheme}
Before we start, let us point out that for our scheme to work, we must avoid self-contact or contact with the rigid boundary of the elastic body initially, at least for $\varepsilon>0$, so the initial contact can be recovered as $\varepsilon\to 0$. The key idea here is to use the bulk nature of the solid and to ``pull inside'' using the reference configuration. For this we have to construct a smooth diffeomorphism $\Phi: \Omega_S \to \Omega_S$ which mostly is equal to the identity, but maps the contact set $C(\bfeta_0) \cup N(\bfeta_0)$ to the interior of $\Omega_S$. A simple example would be
\begin{align*}
     \Phi_{\varepsilon}(x) := x- \varepsilon \boldsymbol\nu (x) \psi(x),
\end{align*}
where $\boldsymbol\nu :\overline{\Omega_S} \to B_1(0)$ is a sufficiently smooth continuation of the unit normal and $\psi(x): \overline{\Omega_S} \to [0,1]$ is a cut-off equal to one on $C(\bfeta_0)\cup N(\bfeta_0)$. Since we are dealing with a fixed, smooth enough domain $\Omega_S$, for $\varepsilon$ small enough, then $\Phi_{\varepsilon}$ is injective and has a uniformly positive Jacobian. With this, we can then define
\begin{align*}
    \bfeta_{0,\varepsilon}(x) := \bfeta_0 \circ \Phi_{\varepsilon}(x)
\end{align*}
which itself has the same regularity as $\bfeta_0$, satisfies $\det \nabla \bfeta_{\varepsilon} = \det \nabla \bfeta_0 \det \nabla \Phi_\varepsilon > c > 0$ for some $c$ and per construction $C(\bfeta_{0,\varepsilon}) = N(\bfeta_{0,\varepsilon}) = \emptyset$. Additionally we have $\bfeta_{0,\varepsilon} \to \bfeta_0$ in all relevant norms. Note that this approximate data can then be further regularized, e.g.\ by a standard convolution argument, without making it inadmissible by causing overlap. For more details also compare Step 1 of the proof of Theorem 2.4 in \cite{cesikInertialEvolutionNonlinear2022}.\\

Next, we define the contact-penalization term\footnote{Contact penalization functionals are a standard tool, in particular in the numerics of solids (see e.g.\ \cite{kromerGlobalInjectivitySecondgradient2019}). Note however that commonly one penalizes overlap of two different parts of the solid, while we instead have to resort to penalizing already a short distance before contact, in order to preserve the fluid domain.} given by
\begin{align*}
    K_\varepsilon(\bfeta) := \int_{\Omega^S} \kappa_\varepsilon(\operatorname{dist}(\bfeta,\partial \Omega)) + \iint_{\Omega^S \times \Omega^S \setminus \{|x-y| < \sqrt{\varepsilon}\}} \kappa_\varepsilon(|\bfeta(x)-\bfeta(y)|)
\end{align*}
with $\kappa_\varepsilon(r) = \kappa(r/\varepsilon)$ and $\kappa$ such that $\kappa(r) :=  +\infty$ for $r\leq 0$, $\kappa(r) = 0$ for $r\geq 1$ and convex in between with $\lim_{r\to 0}\kappa(r)=+\infty$ and $\lim_{r\to1}\kappa'(r) = 0$. Also we assume that there is a constant $c$ such that $|\kappa'(r) r | \leq c(\kappa(r)+1)$ for all $r>0$. Note that the last translates into
\begin{align} \label{eq:kappaGrowth}
 |\kappa_\varepsilon'(r) r| \leq c\varepsilon (\kappa_\varepsilon(r)+1).
\end{align}
A canonical example of such a function would be $\kappa(r) := \frac{1}{r} +r -2$ for $r\in (0,1)$ and $\kappa(r)=0$ for $r\geq 1$. Additionally, the following regularisation of elastic and viscoelastic energy terms is introduced:
\begin{align*}
 E_\varepsilon(\bfeta) := E(\bfeta)+ \varepsilon^{a_0}\norm[W^{k_0,2}]\bfeta^2, \quad  R_\varepsilon(\bfeta,\boldsymbol b) := R(\bfeta_k,\boldsymbol b) + \varepsilon\norm[W^{k_0,2}]{\boldsymbol b}^2,
\end{align*}
and the approximate stress tensor $\boldsymbol\sigma_\varepsilon$ is defined as
\begin{eqnarray*}%%%%%%%%%%%%%
    \nabla\cdot \boldsymbol\sigma_\varepsilon:= -DE_\varepsilon
    (\bfeta)-D_2R_\varepsilon(\bfeta,\partial_t \bfeta), \quad \text{ in } \mathcal{D}'(Q_{S,T})
\end{eqnarray*}%%----------------------------%%
and the exponent $a_0$ is needed for convergence later. Finally, define the time delay operator $\tau_h$ as
\begin{eqnarray*}%%%%%%%%%%%%%
   \tau_h f(t):= f(t-h). 
\end{eqnarray*}%%----------------------------%%
\noindent
\underline{\textbf{The structure sub-problem (SSP)}:} \\
By induction on $0\leq n\leq N-1$, assume that:
\begin{itemize}
    \item[] \textbf{Case} $n=0$: $\bfeta^0(0,\cdot):=\bfeta_{0,\varepsilon}$,~ $\partial_t\bfeta^0(0,\cdot): = \bb{v}_0$;
    \item[] \textbf{Case} $n\geq 1$: the solution $\bfeta^{n}$ of $(SSP)$ and the solution $(\rho^n,\bb{u}^n)$ of $(FSP)$ (defined below) are already obtained.
\end{itemize}

\noindent Find $\bfeta^{n+1}$ so that:

\begin{enumerate}
\item $\bfeta^{n+1} \in L^\infty(nh,(n+1)h; W^{k_0,2}(\Omega_S)),~ \partial_t \bfeta^{n+1} \in L^2(nh,(n+1)h;W^{k_0,2}(\Omega_S))$;

\item The structure equation
\begin{eqnarray}%%%%%%%%%%%%%
   &&\int_{nh}^{(n+1)h}\int_{\Omega_S} \frac{\partial_t \bfeta^{n+1}-\tau_h\bb{U}^n}h\cdot \boldsymbol\phi+\int_{nh}^{(n+1)h}\langle DE_\varepsilon(\bfeta^{n+1}),\boldsymbol\phi \rangle\nonumber\\
   &&\quad+\int_{nh}^{(n+1)h}\langle D_2R_\varepsilon(\bfeta^{n+1},\partial_t \bfeta^{n+1}), \boldsymbol\phi \rangle + \int_{nh}^{(n+1)h} \langle DK_\varepsilon (\bfeta^{n+1}),\boldsymbol\phi \rangle=0 \label{momeqSSP}
\end{eqnarray}%%----------------------------%%
holds for all $\boldsymbol\phi \in C^\infty( [nh,(n+1)h]\times \overline{\Omega_S})$;
\item Structure energy inequality
\begin{eqnarray}%%%%%%%%%%%%%
    &&E_\varepsilon(\bfeta^{n+1}(t)) + K_\varepsilon(\bfeta^{n+1}(t))+2\int_{nh}^t R_\varepsilon( \bfeta^{n+1},\partial_t \bfeta^{n+1}) \nonumber\\
    &&\quad+ \frac{1}{2h}\int_{nh}^t\int_{\Omega_S}\Big( | \partial_t \bfeta^{n+1}-\tau_h\bb{U}^{n}|^2  +| \partial_t \bfeta^{n+1}|^2\Big) \nonumber\\
    && \leq  E_\varepsilon(\bfeta^{n+1}(nh))+K_\varepsilon(\bfeta^{n+1}(nh))+ \frac{1}{2h}\int_{nh}^t\int_{\Omega_S} | \tau_h\bb{U}^{n}|^2(nh)  \label{SSPen}
\end{eqnarray}%%----------------------------%%
for all $t\in(nh,(n+1)h]$.

\end{enumerate}

${}$ \\ ${}$\\

\noindent
We extend the fluid domain to the entire $\Omega$ by extending the initial data and the fluid dissipation. More precisely, we take sufficiently regular $\rho_{0,\varepsilon}$ and $(\rho\bb{u})_{0,\varepsilon}$ such that
\begin{eqnarray*}%%%%%%%%%%%%%
    &&\rho_{0,\varepsilon}\geq 0,~~ \rho_{0,\varepsilon}|_{\mathbb{R}^3\setminus \Omega^{\bfeta_{0,\varepsilon}}}=0, ~~\int_\Omega\rho_{0,\varepsilon}^\gamma+\varepsilon\rho_{0,\varepsilon}^{\beta}\leq c,\\
    &&(\rho\bb{u})_{0,\varepsilon}=0 \text{ wherever } \rho_{0,\varepsilon}=0,\\
    &&\rho_{0,\varepsilon}\to \rho_0 \text{ in }L^{\gamma}(\Omega^{\bfeta_0}) \text{ and } (\rho\bb{u})_{0,\varepsilon}\to (\rho\bb{u})_0 \text{ in } L^{\frac{2\gamma}{\gamma+1}}(\Omega^{\bfeta_0}), \text{ as } \varepsilon \to 0,
\end{eqnarray*}%%----------------------------%%
where $\beta>0$ is a sufficiently large number, and for given $\varepsilon>0$ and $\bfeta$, we define
\begin{eqnarray*}%%%%%%%%%%%%%
      \mathbb{S}_{\varepsilon, \bfeta}(\nabla\bb{u}):=\mathbb{S}(\nabla\bb{u})\times\begin{cases}
    1, \quad \text{on } \Omega^\bfeta, \\
    \varepsilon,\quad \text{elsewhere}.\end{cases}
\end{eqnarray*}%%----------------------------%%
We now introduce:\\

\noindent
\underline{\textbf{The fluid sub-problem (FSP)}:} \\
By induction on $0\leq n\leq N-1$, assume that:
\begin{itemize}
    \item[] \textbf{Case} $n=0$: $\rho^0(0,\cdot):=\rho_{0,\varepsilon}(\cdot)$,~ $(\rho\bb{u})^0(0,\cdot): = (\rho\bb{u})_{0,\varepsilon}(\cdot)$;
    \item[] \textbf{Case} $n\geq 1$: the solution $(\rho^n,\bb{u}^n)$ of $(FSP)$ and the solution $\bfeta^{n+1}$ of $(SSP)$ are already obtained.
\end{itemize}
Find $(\rho^{n+1}, \bb{u}^{n+1})$ so that:
\begin{enumerate}
    \item $\rho^{n+1} \in L^\infty(nh,(n+1)h; L^\beta(\Omega))$, $\bb{u}^{n+1} \in L^2(nh,(n+1)h; H_0^1(\Omega))$;
    \item $\rho^{n+1}(nh,\cdot)= \rho^n(nh,\cdot)$ and $(\rho\bb{u})^{n+1}(nh,\cdot) = (\rho\bb{u})^n(nh,\cdot)$ in weakly continuous sense in time;
    \item The following damped continuity equation
\begin{align*}%%%%%%%%%%%%%
    &\int_{nh}^{(n+1)h} \bint_\Omega \rho^{n+1} ( \partial_t \varphi +\bb{u}^{n+1}\cdot \nabla \varphi) - \int_{nh}^{(n+1)h} \bint_{\Omega_S^{\bfeta^{n+1}}} (\rho^{n+1})^{2} \varphi - \varsigma\int_{nh}^{(n+1)h} \bint_\Omega  \nabla\rho^{n+1}\cdot\nabla\varphi  \nonumber \\
    &=- \int_{nh}^{(n+1)h}\frac{d}{dt}\int_{\Omega} \rho^{n+1} \varphi  
\end{align*}%%----------------------------%%
holds for all $\varphi \in C^\infty([nh ,(n+1)h]\times \mathbb{R}^3)$;
\item The following momentum equation
\begin{align}%%%%%%%%%%%%%
    &\int_{nh}^{(n+1)h} \bint_\Omega \rho^{n+1} \bb{u}^{n+1} \cdot\partial_t \boldsymbol\xi + \int_{nh}^{(n+1)h} \bint_\Omega (\rho^{n+1} \bb{u}^{n+1} \otimes \bb{u}^{n+1}):\nabla\boldsymbol\xi\nonumber\\
    & \quad +\int_{nh}^{(n+1)h} \bint_\Omega \big((\rho^{n+1})^\gamma + \varepsilon (\rho^{n+1})^\beta \big)\nabla \cdot \boldsymbol\xi-  \int_{nh}^{(n+1)h} \bint_\Omega   \mathbb{S}_{\varepsilon, \bfeta^{n+1}}(\nabla\bb{u}^{n+1}): \nabla \boldsymbol\xi \nonumber \\
    &\quad- \frac{1}2 \int_{nh}^{(n+1)h}\int_{\Omega_S^{\bfeta^{n+1}}(t)} (\rho^{n+1})^2 \bb{u}^{n+1} \cdot \boldsymbol\xi-\varsigma\int_{nh}^{(n+1)h} \int_{\Omega}(\nabla\rho^{n+1}\cdot\nabla\bb{u}^{n+1})\cdot\boldsymbol\xi\nonumber\\
    & \quad-
    \int_{nh}^{(n+1)h}  \bint_{\Omega_S}\ddfrac{\bb{U}^{n+1}-\partial_t \bfeta^{n+1}}{h} \cdot\boldsymbol\phi  \nonumber\\
    &= \int_{\Omega}\rho^{n+1}\mathbf{u}^{n+1}\cdot\boldsymbol\xi~\Big|_{nh}^{(n+1)h}\label{momeqFSP}
\end{align}%%----------------------------%%
holds for all $\boldsymbol\xi\in C^\infty([nh,(n+1)h]\times \Omega)$, where 
\begin{eqnarray*}%%%%%%%%%%%%%
    \bb{U}^{n+1}:=\bb{u}^{n+1}\circ\bfeta^{n+1}, \quad \boldsymbol\phi:=\boldsymbol\xi\circ\bfeta^{n+1};
\end{eqnarray*}%%----------------------------%%
\item The fluid energy inequality 
\begin{eqnarray}%%%%%%%%%%%%%
    &&\int_{\Omega} \Big( \frac{1}{2} \rho^{n+1} |\bb{u}^{n+1}|^2 + \frac{(\rho^{n+1})^\gamma}{\gamma-1}+\varepsilon \frac{(\rho^{n+1})^\beta}{\beta-1}\Big)(t)  + \int_{nh}^t\int_{\Omega} \mathbb{S}_{\varepsilon,\bfeta^{n+1}}(\nabla \mathbf{u}^{n+1}):\nabla \mathbf{u}\nonumber \\
    &&\quad+\int_{nh}^t \int_{\Omega_S^{\bfeta^{n+1}}(t)}\left( \frac{1}{\gamma-1}(\rho^{n+1})^{\gamma+1}+\frac{\varepsilon}{\beta-1}(\rho^{n+1})^{\beta+1}\right) \nonumber \\
    &&\quad+ \varsigma  \int_{nh}^t |\nabla \rho^{n+1}|^2\left( \gamma(\rho^{n+1})^{\gamma-2}+ \varepsilon\beta (\rho^{n+1})^{\beta-2} \right)\nonumber \\
    &&\quad+  \frac{1}{2h}\int_{nh}^t\int_{\Omega_S}\Big( |\bb{U}^{n+1} - \partial_t \bfeta^{n+1}|^2  +| \bb{U}^{n+1}|^2\Big) \nonumber \\
    &&\leq \int_\Omega \Big(  \rho^n|\bb{u}^n|^2 +  \frac{(\rho^n)^\gamma}{\gamma-1}  \Big)(nh) + \frac{1}{2h}\int_{\Omega_S} | \partial_t \bfeta^{n+1}|^2(nh) \label{FSPen}
\end{eqnarray}%%----------------------------%%
holds for all $t\in (nh,(n+1)h]$.
\end{enumerate}

\subsection{Discussion about the scheme}
 The fluid equations are extended to the entire domain $\Omega$ and the fluid is allowed to pass through the fluid-structure interface and occupy the structure domain. This way, we obtain a problem on a fixed domain. This idea combines decoupling, penalization and domain extension and was first done in \cite{heatFSI} in the context of interaction between a heat-conducting fluid and a thermoelastic shell. The penalization and domain extension approach for compressible viscous fluids on moving domains without the structure comes from \cite{feireisl2011convergence}. Next, the mentioned decoupling method in fluid-structure interaction which splits the coupled momentum equation in the structure acceleration term was first done in \cite{SubBorARMA}, where the fluid is incompressible and the elastic solid is a beam. The constructed scheme was stationary both for fluid and structure, and they are solved through a time-marching scheme. Later, in \cite{srdjan1}, this scheme was adjusted to the case where the structure is a nonlinear plate, and a hybrid scheme was constructed which is stationary for the fluid and time-continuous for the plate. This way, the energy for the plate (which is not discretizable in time) is properly conserved for the approximate solution. Finally, in \cite{trwa3}, the interaction between a compressible viscous fluid and a nonlinear thermoelastic plate is studied, and the decoupling scheme is fully time-continuous, which allowed the standard theory for solving both sub-problems to be used. 

The solutions to the solid equations in turn are constructed using a method related to that employed in \cite{benevsova2020variational} where the effects of inertia are initially dealt with on a different time-scale than those resulting from energy and dissipation. This allows us to begin by treating the problem as a gradient flow with forces. For this type of problems we can then employ the variational method of minimizing movements. In contrast to other approaches relying on linearization or finite dimensional approximations, this allows us to directly work with nonlinearities, both in the equation as well as in that the space of admissible deformations which for this kind of problem is generally not closed under any form of linear interpolation.

We deviate from \cite{benevsova2020variational} in two key aspects though. Firstly, in favor of the more simple approach described above, we choose to not construct an approximation to the fluid equations using the same method (compare \cite{breit2021compressible} to see how this can be done). Secondly, while the approach can in fact deal with ``hard'' collisions of solids (see \cite{cesikInertialEvolutionNonlinear2022}), we instead employ a soft penalization potential to avoid collisions in our approximation, as is well established, in particular in numerics (see e.g.\ \cite{kromerGlobalInjectivitySecondgradient2019} for a modern approach). This allows us to have regular approximate fluid domains, which makes working with them easier. To the best of our knowledge this is  the first time this type of approach has been applied to the fluid-structure interactions.

\noindent
Let us now point out the roles of the added terms in the equations. 
        \begin{enumerate}
            \item The contact-penalization $K_\varepsilon$ is added to ensure that the deformations satisfy the Ciarlet-Ne\v cas condition;
            \item The approximate energy functionals $E_\varepsilon$ and $R_\varepsilon$ provide stronger regularity which ensures that we can use $\partial_t \bfeta$ as a test function in the structure sub-problem and derive the energy inequality;
            \item The term $\smash{\chi_{\Omega_S^{\bfeta^{n+1}}} (\rho^{n+1})^2}$ in the continuity equation ensures that pressure is bounded in $L_{t,x}^p$ for $p>1$ on the structure domain $\smash{\Omega_S^{\bfeta^{n+1}}}$, since we cannot obtain the improved pressure estimates inside the structure domain uniformly w.r.t. parameter $h$, due to structure terms. Then, the term $\smash{-\frac{1}2 \chi_{\Omega_S^{\bfeta^{n+1}}} (\rho^{n+1})^{2} \bb{u}^{n+1}}$ in the momentum equation is added to correct the energy. Both of these terms will converge to zero weakly as $h,\varsigma\to 0$, due to Lemma $\ref{vanish:density}$. Since the solution to the fluid sub-problem is obtained as a limit of a system in a finite basis, the term $-\varsigma \Delta \rho^{n+1}$ is added (or rather kept) to ensure that we can identify the mentioned nonlinear terms, while once again $-\varsigma \nabla\rho^{n+1}\cdot \nabla\bb{u}^{n+1}$ is added to the fluid momentum equation to correct the energy;
            \item The artificial pressure $\varepsilon(\rho^{n+1})^\beta$ is standard in the theory of compressible Navier-Stokes equations. It is added to ensure that effective viscous flux converges properly as artificial density damping vanishes $-\varsigma \Delta\rho\to 0$;
            \item The terms $\chi_{\Omega_S}\frac{\bb{U}^{n+1}-\partial_t \bfeta^{n+1}}{h}$ and $\frac{\partial_t \bfeta^{n+1}-\tau_h\bb{U}^n}h$ allow us to decouple the coupled momentum equation $\eqref{momeqweak}$ and thus decouple the problem into fluid and structure sub-problems. Moreover, they act as a two-sided penalization, and force the limiting functions $\partial_t \bfeta$ and $\bb{u}\circ \bfeta$ to become equal on the solid domain as $\varsigma,h\to 0$, which will ensure the kinematic coupling condition. The coupled momentum equation will be recovered by summing up $\eqref{momeqSSP}$ and $\eqref{momeqFSP}$ with a common test function, so the mentioned terms, when summed up, will give
            $\frac{\bb{U}^{n+1}-\tau_h \bb{U}^n}h$, a term which converges to $\partial_t \bb{U} = \partial_t^2 \bfeta$ -- the acceleration of the solid, as $h\to 0$.
        \end{enumerate}

\subsection{Solving the fluid sub-problem}
The solution of the fluid sub-problem can be obtained as in standard theory of compressible viscous fluids. The idea is to span the fluid velocity in a finite Galerkin basis on $\Omega$. The solution of such problem is obtained by Schauder's fixed point theorem, while the estimates are obtained by testing the continuity equation with $\frac1{\gamma-1}\rho^\gamma+\varepsilon \frac1{\beta-1}\rho^{\beta-1}$ and $-\frac12|\bb{u}|^2$ and the momentum equation with $\bb{u}$ and summing these 3 identities together. Afterwards, the limit in the number of basis functions is passed to, so the solution is obtained. For more details, the reader is referred to \cite[Section 7]{NovStr}.

\subsection{Solving the structure sub-problem}
We introduce a time-discretized version. First, let us split the time interval $[nh,(n+1)h]$ into $M$ time intervals of length $\Delta t$, that is $[nh,nh+\Delta t], [nh+\Delta t, nh+2\Delta t],...,[nh+(M-1)\Delta t, (n+1)h]$. Finally, we discretize the function $\bb{U}^n$ in time and denote it with
\begin{eqnarray*}%%%%%%%%%%%%%
    \bb{U}_k^n := \frac{1}{\Delta t }\int_{nh+k \Delta t}^{nh+(k+1) \Delta t} \bb U^n dt, \quad k = 0,1,2,...,M-1.
\end{eqnarray*}%%----------------------------%%
For $1\leq k\leq M$, we define the problem as
\begin{eqnarray*}%%%%%%%%%%%%%
    \frac1h\int_{\Omega_S}\left( \frac{\bfeta_k^{n+1}-\bfeta_{k-1}^{n+1}}{\Delta t} -\bb{U}_k^n\right)\cdot\boldsymbol\phi+\langle DE_\varepsilon(\bfeta_k^{n+1}), \boldsymbol\phi \rangle+\langle D_2R_\varepsilon(\bfeta_k^{n+1},\partial_t \bfeta_k^{n+1}),\boldsymbol\phi \rangle + \langle DK_\varepsilon(\bfeta_k^{n+1}),\boldsymbol\phi \rangle=0
\end{eqnarray*}%%----------------------------%%
for all $\boldsymbol\phi\in C^\infty(\Omega_S)$, where $\bfeta_0^{n+1} := \bfeta^{n}(nh)$. \\

The approach we use for this part of the proof is quite similar to \cite[Theorem 3.5]{benevsova2020variational}.
However for the sake of completeness and because of the additional collision potential $K_\varepsilon$ we will still carry out the details. The idea is that if $\boldsymbol U_k^n$ is treated as given data, then at this point the problem has the structure of a gradient flow with forces. It can thus be approached by a minimizing movements scheme. As $n$ will be fixed in this section and apart from initial and right hand side data, all quantities have the same index $n+1$, we will drop this superscript everywhere where it is not necessary.

Starting with $\bfeta_0^{n+1}:= \bfeta^n(nh)$ we begin by considering the iterative problem of minimizing
\begin{align*}
 \bfeta \mapsto E_\varepsilon(\bfeta) + K_\varepsilon(\bfeta) + \Delta t R_\varepsilon\left(\bfeta_k, \frac{\boldsymbol \bfeta - \boldsymbol \bfeta_k}{\Delta t}\right) + \frac{\Delta t}{2h} \norm[L^2(\Omega_S)]{\frac{\boldsymbol \bfeta - \boldsymbol \bfeta_k}{\Delta t}-\bb U_k}^2
\end{align*}
among the set $\mathcal{E}^q\cap W^{k_0,2}(\Omega_S)$ to obtain $\bfeta_{k+1}$.
Here $\bfeta_k$ is the solution obtained in the previous $\Delta t$-step.

We first note that this problem has a minimizer as a consequence of the direct method. Specifically, since all terms are non-negative, the growth condition on $E$ gives us coercivity in $W^{k_0,2}$ and by our assumptions or respectively their specific form, all terms are weakly lower-semicontinuous in $\bfeta$ with respect to that space. Thus any minimizing sequence has a converging subsequence and a limit, which then has to be a minimizer (cmp.\ also \cite[Prop. 2.13]{benevsova2020variational} for details). However one should keep in mind that such a minimizer is not necessarily unique.

Let us now denote the resulting piecewise constant and piecewise affine approximations by
\begin{align*}
 \overline\bfeta_{\Delta t} (t) &:= \bfeta_{k+1}, &\text{ for }&t\in[nh+k\Delta t,nh+(k+1)\Delta t) ,\\
 \underline\bfeta_{\Delta t} (t) &:= \bfeta_{k}, &\text{ for }&t\in[nh+k\Delta t,nh+(k+1)\Delta t), \\
 \hat\bfeta_{\Delta t} (t) &:= \frac{t-nh-k\Delta t}{\Delta t} \bfeta_{k+1} +\frac{nh-(k+1)\Delta t-t}{\Delta t} \bfeta_k, &\text{ for }&t\in[nh+k\Delta t,nh+(k+1)\Delta t).
\end{align*}
Additionally since we know that $\bfeta_{k+1}$ is a minimizer, we can compare it with the choice $\bfeta_{k}$ in the minimization to obtain the initial estimate:
\begin{align*}
 & E_\varepsilon(\bfeta_{k+1}) + K_\varepsilon(\bfeta_{k+1}) + 2\Delta t R_\varepsilon\left(\bfeta_k, \frac{\boldsymbol \bfeta_{k+1} - \boldsymbol \bfeta_k}{\Delta t}\right) + \frac{\Delta t}{2h} \norm{\frac{\boldsymbol \bfeta_{k+1} - \boldsymbol \bfeta_k}{\Delta t}-\bb U_k}^2 \\
 &\leq E_\varepsilon(\bfeta_k) + K_\varepsilon(\boldsymbol \bfeta_k) + \frac{\Delta t}{2h} \norm{\bb U_k}^2.
\end{align*}
Summing up this estimate over $k$ and replacing the $\bfeta_k$ and the difference quotients by the respective approximations then results in
\begin{align*}
 &  E_\varepsilon(\overline\bfeta_{\Delta t} (t)) + K_\varepsilon(\overline\bfeta_{\Delta t} (t)) + 2\int_{nh}^t R_\varepsilon\left(\underline\bfeta_{\Delta t} (t), \partial_t \hat\bfeta_{\Delta t} (t) \right) + \frac{1}{2h} \norm{\partial_t \hat\bfeta_{\Delta t} (t)-\underline{\bb U}_{\Delta t}(t)}^2 dt \\
 &\leq E_\varepsilon(\bfeta_0^{n+1}) + K_\varepsilon(\boldsymbol \bfeta_0^{n+1}) + \int_{nh}^t \frac{1}{2h} \norm{\underline{\bb U}_{\Delta t}(t)}^2 dt
\end{align*}
for all $t \in [nh,(n+1)h]$ where $\underline{\bb U}_{\Delta t} (t) := \bb U_k$ for $t\in[nh+k{\Delta t},nh+(k+1){\Delta t})$ similar to before.

Per Jensen's inequality we additionally have $ \int_{nh}^t \norm{\underline{\bb U}_{\Delta t}(t)}^2 dt \leq \int_{nh}^t \norm{ \bb U^n(t)}^2 dt$. So in particular this gives us ${\Delta t}$-independent a-priori estimates in the form of 
\begin{align*}
  & \sup_{t\in [nh,(n+1)h]} E_\varepsilon(\overline\bfeta_{\Delta t} (t)) \leq C, & \quad \sup_{t\in [nh,(n+1)h]} K_\varepsilon(\overline\bfeta_{\Delta t} (t)) \leq C,\\ & \quad \int_{nh}^{(n+1)h} R_\varepsilon\left(\underline\bfeta_{\Delta t} (t), \partial_t \hat\bfeta_{\Delta t} (t) \right)dt \leq C, & \quad \int_{nh}^{(n+1)h} \frac{1}{2h} \norm{\partial_t \hat\bfeta_{\Delta t} (t)}^2 dt \leq C.
\end{align*}
Using the coercivity of $E_\varepsilon$, the fact that $\overline\bfeta_{\Delta t}$ and $\underline\bfeta_{\Delta t}$ only differ by a small shift in time and that $\hat\bfeta_{\Delta t}$ is their linear interpolation we then get that
\begin{align*}
 \overline\bfeta_{\Delta t},\underline\bfeta_{\Delta t},\hat\bfeta_{\Delta t} &\text{ are bounded in } L^\infty(nh,(n+1)h; W^{k_0,2}(\Omega;\mathbb{R}^n)), \\
 \hat\bfeta_{\Delta t} &\text{ is bounded in } W^{1,2}(nh,(n+1)h; W^{k_0,2}(\Omega;\mathbb{R}^n)).
\end{align*}

Using the classical Aubin-Lions Lemma and the Banach-Alaoglu theorem, this allows us to take a subsequence and a limit $\boldsymbol \bfeta^n$ such that
\begin{align*}
 \hat\bfeta_{\Delta t} \to \boldsymbol \bfeta^n \text{ in } C^0([nh,(n+1)h]; W^{k_0,2}(\Omega;\mathbb{R}^n)) \cap W^{1,2}(nh,(n+1)h; W^{k_0,2}(\Omega;\mathbb{R}^n)),
\end{align*}
as $\Delta t\to 0$. Additionally the closeness of the three interpolations then also implies
\begin{align*}
 \overline\bfeta_{\Delta t},\underline\bfeta_{\Delta t} \rightharpoonup^* \boldsymbol \bfeta^n &\text{ in } L^\infty(nh,(n+1)h; W^{k_0,2}(\Omega;\mathbb{R}^n)), \quad \text{ as }\Delta t\to 0.
\end{align*}

Next we go back to the initial minimization and note that as a minimizer $\bfeta_{k+1}$ fulfills the Euler-Lagrange equation
\begin{align*}
    \langle DE_\varepsilon(\bfeta_{k+1}), \boldsymbol\phi \rangle+ \langle DK_\varepsilon(\bfeta_{k+1}),\boldsymbol\phi \rangle + \left \langle D_2R_\varepsilon\left(\bfeta_k,\frac{\bfeta_{k+1}-\bfeta_{k}}{\Delta t} \right),\boldsymbol\phi \right\rangle + \frac1h\int_{\Omega_S}\left( \frac{\bfeta_{k+1}-\bfeta_{k}}{\Delta t} -\bb{U}_k\right)\cdot\boldsymbol\phi = 0
\end{align*}
for any $\boldsymbol \phi \in W^{k_0,2}(\Omega_S)$,
which in terms of the previous approximations can be restated as
\begin{align*}
    \langle DE_\varepsilon(\overline\bfeta_{\Delta t} (t)), \boldsymbol\phi \rangle+ \langle DK_\varepsilon(\overline\bfeta_{\Delta t} (t)),\boldsymbol\phi \rangle + \left \langle D_2R_\varepsilon\left(\underline\bfeta_{\Delta t},\partial_t \hat\bfeta_{\Delta t} (t) \right),\boldsymbol\phi \right\rangle + \frac1h\int_{\Omega_S}\left( \partial_t \hat\bfeta_{\Delta t} (t) -\underline{\bb U}_{\Delta t}(t)\right)\cdot\boldsymbol\phi = 0
\end{align*}
for almost all $t \in [nh,(n+1)h]$. Integrating in time and taking the limit ${\Delta t}\to 0$ then results in %TODO: details
\begin{align*}
    & \int_{nh}^{(n+1)h} \langle DE_\varepsilon(\bfeta^{n+1} (t)), \boldsymbol\phi \rangle+ \langle DK_\varepsilon(\bfeta^{n+1} (t)),\boldsymbol\phi \rangle \\
    &\quad+ \left \langle D_2R_\varepsilon\left(\bfeta^{n+1},\partial_t \bfeta^{n+1} (t) \right),\boldsymbol\phi \right\rangle + \frac1h\int_{\Omega_S}\left( \partial_t \bfeta^{n+1} (t) -\bb U^n(t)\right)\cdot\boldsymbol\phi\, dt = 0.
\end{align*}

Finally we derive an energy inequality for this equation. For this we test with $\boldsymbol \phi = \partial_t \bfeta^{n+1}$ to obtain via the chain-rule
\begin{align*}
    0&= \int_{nh}^{(n+1)h} \langle DE_\varepsilon(\bfeta^{n+1} (t)), \partial_t \bfeta^{n+1} \rangle+ \langle DK_\varepsilon(\bfeta^{n+1} (t)),\partial_t \bfeta^{n+1}\rangle  \\
    &\quad+ \left \langle D_2R_\varepsilon\left(\bfeta^{n+1},\partial_t \bfeta^{n+1} (t) \right),\partial_t \bfeta^{n+1} \right\rangle + \frac1h\int_{\Omega_S}\left( \bfeta^{n+1} (t) -\bb U^n(t)\right)\cdot\partial_t \bfeta^{n+1}\, dt \\
    &= E_\varepsilon(\bfeta^{n+1} ((n+1)h)) - E_\varepsilon(\bfeta^{n+1} (nh)) +  K_\varepsilon(\bfeta^{n+1} ((n+1)h)) - K_\varepsilon(\bfeta^{n+1} (nh)) \\ 
    &\quad+ \int_{nh}^{(n+1)h} 2R_\varepsilon\left(\bfeta^{n+1},\partial_t \bfeta^{n+1} (t) \right) + \frac1h\int_{\Omega_S}\left( \bfeta^{n+1} (t) -\bb U^n(t)\right)\cdot\partial_t \bfeta^{n+1}\, dt,
\end{align*}
so by using the identity $2a(a-b) = a^2 +(a-b)^2 - b^2$, one obtains the (in)equality $\eqref{SSPen}$.
%Using Young's inequality on the mixed part of the last term and reordering then finally results in the energy inequality
%\begin{align*}
% & E_\varepsilon(\bfeta^{n+1} ((n+1)h)) +  K_\varepsilon(\bfeta^{n+1} ((n+1)h)) 
 %   + \int_{nh}^{(n+1)h} 2R_\varepsilon\left(\bfeta^{n+1},\partial_t \bfeta^{n+1} (t) \right) + \frac1{2h} \norm{\bfeta^{n+1} (t)}^2 \, dt \\
  %  &\leq E_\varepsilon(\bfeta^{n+1} (nh)) +  K_\varepsilon(\bfeta^{n+1} (nh)) + \int_{nh}^{(n+1)h} \norm{\bb U^n(t)}^2 dt.
%\end{align*}

\subsection{Coupled back problem and uniform estimates}
Here, we sum up equations $\eqref{momeqSSP}$ and $\eqref{momeqFSP}$ with a common test function $\boldsymbol\phi$, then the energy inequalities $\eqref{SSPen}$ and $\eqref{FSPen}$ and finally sum over $n=0,...,N-1$. Denoting $f(t):=f^{n}(t)$ for $t\in [(n-1)h,nh)$, we obtain that the constructed approximate solutions $(\rho,\bb{u},\bfeta)$ satisfy:
\begin{mydef}\label{coupledbackproblem}
We say that $(\rho,\bb{u},\bfeta)$ is a weak solution to the regularized problem on an extended domain with decoupled velocities if:
\begin{enumerate}
    \item The following damped continuity equation is satisfied in the weak sense
\begin{eqnarray}\label{reconteqweakapp}%%%%%%%%%%%%%
    \int_0^T \int_\Omega \rho ( \partial_t \varphi +\bb{u}\cdot \nabla \varphi) -  \int_{0}^{T} \bint_{\Omega_S^\bfeta(t)}\rho^{2} \varphi  - \varsigma \int_0^T \nabla \rho \cdot \nabla \varphi =\int_{\Omega_F^{\bfeta_0}}\rho_0 \varphi(0,\cdot)
\end{eqnarray}%%----------------------------%%
for all $\varphi \in C_0^\infty([0,T)\times \mathbb{R}^3)$;
\item The coupled momentum equation
\begin{eqnarray}%%%%%%%%%%%%%
    &&\int_{Q_T} \rho \bb{u} \cdot\partial_t \boldsymbol\xi + \int_{Q_T}(\rho \bb{u} \otimes \bb{u}):\nabla\boldsymbol\xi +\int_{Q_T}(\rho^\gamma+\varepsilon \rho^\beta )\nabla \cdot \boldsymbol\xi -  \int_{Q_T} \mathbb{S}_{\varepsilon,\bfeta}(\nabla\bb{u}): \nabla \boldsymbol\xi\nonumber\\
    &&\quad- \frac{1}2 \int_{0}^{T}\int_{\Omega_S^\bfeta(t)} \rho^{2} \bb{u} \cdot \boldsymbol\xi-\varsigma \int_0^T \int_{\Omega}(\nabla\rho\cdot\nabla\bb{u})\cdot\boldsymbol\xi-\int_{Q_{S,T}} \frac{\partial_t \bfeta  - \tau_h \partial_t {\bfeta }}h \cdot \boldsymbol\phi \nonumber\\
    &&\quad-\int_0^T\langle DE_\varepsilon(\bfeta),\boldsymbol\phi \rangle- \int_0^T\langle D_2R_\varepsilon(\bfeta,\partial_t \bfeta), \boldsymbol\phi \rangle - \int_0^T \langle DK_\varepsilon (\bfeta),\boldsymbol\phi \rangle \nonumber\\
    &&= \int_\Omega(\rho\mathbf{u})_{0,\varepsilon}\cdot\boldsymbol\xi(0,\cdot) +\int_{\Omega_S} \bb{v}_0\cdot \boldsymbol\phi(0,\cdot) \label{coupledmomapp}
\end{eqnarray}%%----------------------------%%
holds for all $\boldsymbol\xi \in C_0^\infty([0,T)\times \overline{\Omega})$, where $ \boldsymbol\phi=\boldsymbol\xi\circ \bfeta$ on $Q_{S,T}$;
\item The following energy inequality
\begin{eqnarray*}%%%%%%%%%%%%%
    &&\int_{\Omega} \Big( \frac{1}{2} \rho |\bb{u}|^2 + \frac{\rho^\gamma}{\gamma-1} \Big)(t)  + \int_{0}^t\int_{\Omega} \mathbb{S}_{\varepsilon,\bfeta}(\nabla \mathbf{u}):\nabla \mathbf{u}\\
    &&\quad+\int_{0}^T \int_{\Omega_S^\bfeta(t)} \left(\frac{1}{\gamma-1}\rho^{\gamma+1}  +\frac{\varepsilon}{\beta-1}\rho^{\beta+1}\right)+\varsigma \int_{0}^T |\nabla \rho|^2\left( \gamma\rho^{\gamma-2}+ \varepsilon\beta \rho^{\beta-2} \right) \\
    &&\quad+E_\varepsilon(\bfeta)(t)+\int_{0}^t 2 R_\varepsilon( \bfeta,\partial_t \bfeta)+K_\varepsilon(\bfeta)(t)\\
    &&\quad+\frac1{2h}\int_{nh}^t\int_{\Omega_S}| \bb{U}|^2+ \frac{1}{2h}\int_{0}^t\int_{\Omega_S}\Big(|\bb{U} - \partial_t \bfeta|^2+ | \partial_t \bfeta-\tau_h\bb{U}|^2 \Big) \nonumber\\
    &&\leq \int_\Omega \Big(  \rho_0|\bb{u}_0|^2 +  \frac{\rho_0^\gamma}{\gamma-1}  \Big) + \frac{1}{2h}\int_{\Omega_S} |\bb{v}_0|^2+ E_\varepsilon(\bfeta_0)
\end{eqnarray*}%%----------------------------%%
for all $0\leq n \leq N-1$ and $t\in [nh,(n+1)h)$.
\end{enumerate}

\end{mydef}

\begin{figure}[h!]
    \centering
    \includegraphics[scale=0.25]{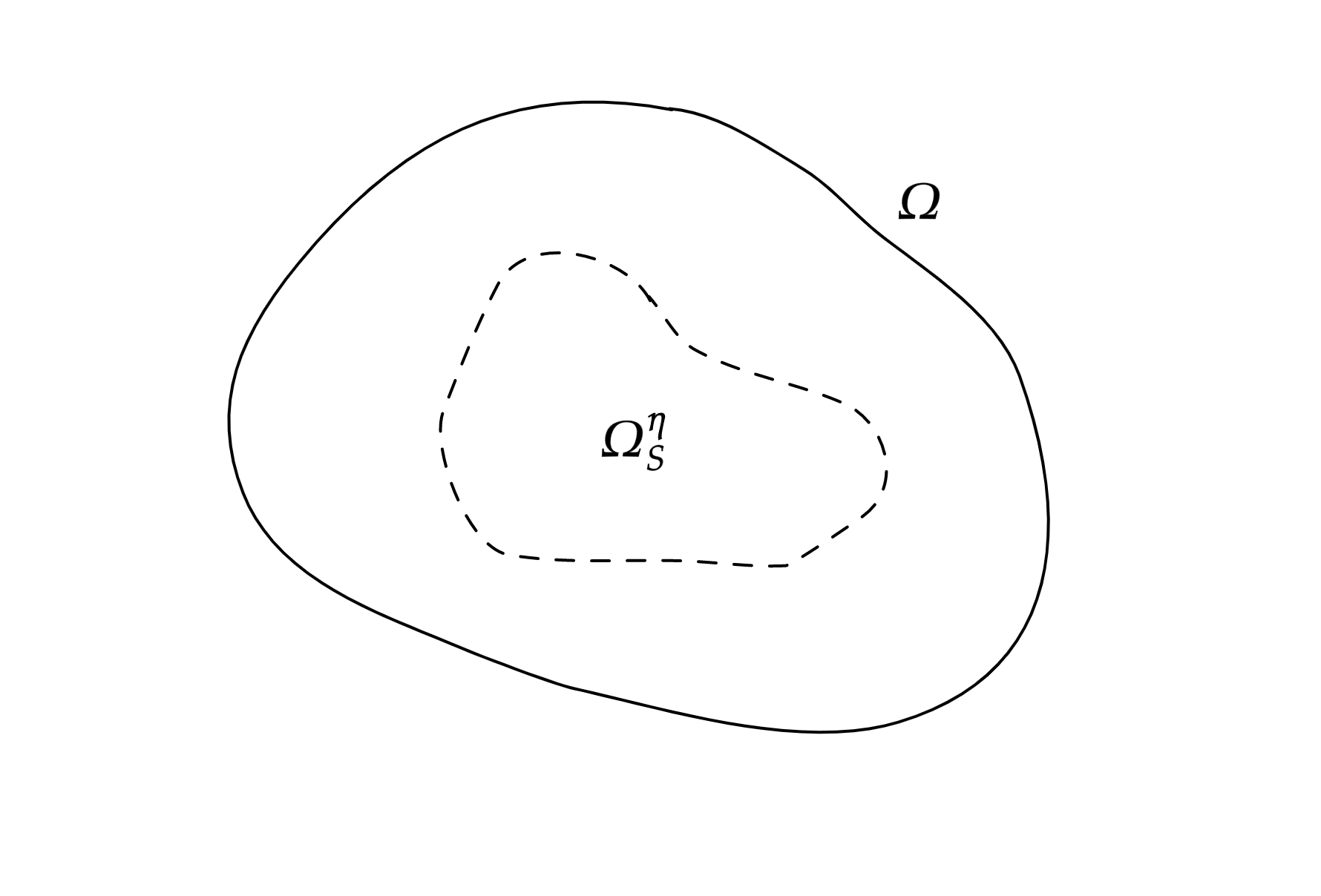}
    \caption{At this approximate level, the fluid is allowed to pass through the elastic solid, which is emphasized with a dashed line. After the limit $h\to 0$, the kinematic coupling is achieved and the fluid cannot pass through the interface.}
    \label{fig:my_label2}
\end{figure}

\begin{cor}
The above weak solution satisfies the following estimates:
\begin{eqnarray}%%%%%%%%%%%%%
    &&\sqrt{\varsigma}||\nabla \rho||_{L^2((0,T)\times \Omega)} \leq C(\varepsilon), \label{nabla:rho:est}\\
    &&||\bb{u}||_{L^2(0,T; H^1(\Omega))} \leq C(\varepsilon),\label{u:H1:est} \\
    &&||\bb{u}||_{L^2(0,T; L^6(\Omega))} \leq C(\varepsilon),\nonumber\\
    &&||\rho \bb{u}||_{L^\infty(0,T; L^{\frac{2\gamma}{\gamma+1}}(\Omega))}\leq C(\varepsilon),\nonumber \\
    &&||\rho \bb{u}||_{L^2(0,T; L^{\frac{6\gamma}{\gamma+6}}(\Omega))}\leq C(\varepsilon),\nonumber\\
    &&||\rho \bb{u}\otimes \bb{u}||_{L^2(0,T; L^{\frac{6\gamma}{4\gamma+3}}(\Omega))}\leq C(\varepsilon),\nonumber\\
    && ||\bfeta||_{L^\infty(0,T; W^{k_0,2}(\Omega_S))}+||\partial_t\bfeta||_{L^2(0,T; W^{k_0,2}(\Omega_S))}+\sup\limits_{t\in(0,T)} K_\varepsilon(t) \leq C(\varepsilon). \nonumber
\end{eqnarray}%%----------------------------%%
\end{cor}

\section{The interaction limit $h\to 0$ and the density dissipation limit $\varsigma\to 0$}\label{sec:h:lim}
The solution obtained in the previous section for a given $\varsigma, h,\varepsilon>0$ in the sense of Definition $\ref{coupledbackproblem}$ is denoted as $(\bfeta_h,\rho_h,\bb{u}_h)$. The goal is to pass to the limit $\varsigma, h\to 0$ (we can choose w.l.o.g.\ that $\varsigma=h$) and to prove that the limiting functions $(\bfeta,\rho,\bb{u})$ satisfy the following problem:

\begin{mydef}\label{weaksolh}
We say that $({\bfeta },\rho,\bb{u})$ is a weak solution of a regularized problem on an extended domain if 
\begin{eqnarray*}%%%%%%%%%%%%%
    &&\bfeta \in L^\infty(0,T; \mathcal{E}^q\cap W^{k_0,2}),~ \partial_t \bfeta \in L^2(0,T;W^{k_0,2}(\Omega_S)),\\
    &&\bb{u} \in L^2(0,T; H^1(\Omega)),~ \rho \in L^
    \infty(0,T;L^\gamma(\Omega)),
\end{eqnarray*}%%----------------------------%%
and
\begin{enumerate}
\item The fluid and solid velocities match on the solid domain, i.e. $\partial_t \bfeta = \mathbf{u} \circ \bfeta$ on $Q_{S,T}$;
\item The renormalized continuity equation
\begin{eqnarray}\label{reconteqweakh}%%%%%%%%%%%%%
    &&\int_{Q_{F,T}^{\bfeta }} \rho B(\rho)( \partial_t \varphi +\bb{u}\cdot \nabla \varphi)  \nonumber \\
    &&=\int_{Q_{F,T}^{\bfeta }} b(\rho)(\nabla\cdot \bb{u}) \varphi +\int_{\Omega_F^{\bfeta_{0,\varepsilon}}} \rho_{0,\varepsilon}  B(\rho_{0,\varepsilon}) \varphi(0,\cdot) , 
\end{eqnarray}%%----------------------------%%
holds for all $\varphi \in C^\infty([0,T)\times \mathbb{R}^3)$ and any $b\in L^{\infty}(0,\infty) \cap C[0,\infty)$ such that $b(0)=0$ with $B(x)=B(1)+\int_1^x \frac{b(z)}{z^2}dz$;
\item The coupled momentum equation
\begin{eqnarray}%%%%%%%%%%%%%
    &&\int_{Q_{F,T}^{\bfeta }} \rho \bb{u} \cdot\partial_t \boldsymbol\xi + \int_{Q_{F,T}^{\bfeta }}(\rho \bb{u} \otimes \bb{u}):\nabla\boldsymbol\xi +\int_{Q_{F,T}^{\bfeta }}(\rho^\gamma+\varepsilon \rho^\beta )\nabla \cdot \boldsymbol\xi -  \int_{Q_{F,T}^{\bfeta }} \mathbb{S}_{\varepsilon,\bfeta}(\nabla\bb{u}): \nabla \boldsymbol\xi \nonumber\\
    &&\quad+\int_{Q_{S,T}} \partial_t {\bfeta }\cdot \partial_t \boldsymbol\phi -\int_0^T\langle DE_\varepsilon(\bfeta),\boldsymbol\phi \rangle-\int_0^T\langle D_2R_\varepsilon(\bfeta,\partial_t \bfeta), \boldsymbol\phi \rangle -\int_0^T \langle DK_\varepsilon (\bfeta),\boldsymbol\phi \rangle\nonumber\\
    &&=\int_{\Omega^{\bfeta_0}}(\rho\mathbf{u})_{0,\varepsilon}\cdot\boldsymbol\xi(0,\cdot)+ \int_{\Omega_S} \bb{v}_0 \cdot \boldsymbol\phi(0,\cdot)\nonumber \\ \label{momeqweakh}
\end{eqnarray}%%----------------------------%%
holds for all $\boldsymbol\xi \in C_0^\infty([0,T)\times \overline{\Omega})$, where $\boldsymbol\phi = \boldsymbol\xi \circ \bfeta$ on $Q_{S,T}$;
\item Energy inequality 
\begin{eqnarray*}%%%%%%%%%%%%%
    &&\int_{\Omega_F^{\bfeta}(t)} \Big( \frac{1}{2} \rho |\bb{u}|^2 + \frac{\rho^\gamma}{\gamma-1}+\varepsilon \frac{\rho^\beta}{\beta-1}\Big)(t)  + \int_0^t\int_{\Omega} \left(\mathbb{S}_{\varepsilon,\bfeta}(\nabla \mathbf{u}):\nabla \mathbf{u}\right)(\tau)~d\tau  \\
    &&\quad+  \frac{1}{2}\int_{\Omega_S} |\partial_t {\bfeta }|^2(t) + E_\varepsilon(\bfeta(t)) +\int_0^t 2R_\varepsilon(\bfeta,\partial_t \bfeta) +K_\varepsilon(\bfeta)(t) \\
    &&\leq \int_{\Omega_F^{\bfeta_0}} \Big( \frac{1}{2\rho_{0,\varepsilon}} |(\rho\bb{u})_{0,\varepsilon}|^2 + \frac{\rho_{0,\varepsilon}^\gamma}{\gamma-1}   \Big) + \frac{1}{2}\int_{\Omega_S} |\bb{v}_0|^2 + E_\varepsilon(\bfeta_{0,\varepsilon}), \label{enineqh}
    \end{eqnarray*}%%----------------------------%%
holds for all $t\in (0,T]$.

\end{enumerate}

\end{mydef}

\begin{cor}
The above weak solution satisfies additional estimates independent from $\varepsilon$:
\begin{eqnarray*}%%%%%%%%%%%%%
    &&||\rho \bb{u}||_{L^\infty(0,T; L^{\frac{2\gamma}{\gamma+1}}(\Omega_F^\bfeta(t)))}\leq C,\nonumber \\
    &&||\rho \bb{u}||_{L^2(0,T; L^{\frac{6\gamma}{\gamma+6}}(\Omega_F^\bfeta(t)))}\leq C,\nonumber\\
    &&||\rho \bb{u}\otimes \bb{u}||_{L^2(0,T; L^{\frac{6\gamma}{4\gamma+3}}(\Omega_F^\bfeta(t)))}\leq C,\nonumber\\
    &&  ||\bfeta||_{L^\infty(0,T; W^{2,q}(\Omega_S))}+||\partial_t\bfeta||_{L^2(0,T; H^1(\Omega_S))}\leq C.\nonumber
\end{eqnarray*}%%----------------------------%%

\end{cor}

\subsection{Limit in the continuity equation and vanishing density inside the structure domain}
Based on the uniform estimates, we have
\begin{eqnarray*}%%%%%%%%%%%%%
    \bb{u}_h \rightharpoonup \bb{u}, \qquad \text{in } L^2(0,T; H^1(\Omega)), \\
    \rho_h \to \rho, \qquad \text{in } L^\infty(0,T; L^\gamma(\Omega)).
\end{eqnarray*}%%----------------------------%%
Moreover, noticing that the coupled momentum equation $\eqref{coupledmomapp}$ can be tested by any $\boldsymbol\phi \in L^2(0,T; W_0^{k_0,2}(\Omega_S))$, where $\boldsymbol\xi=0$ on $[0,T]\times \overline{\Omega_F^\bfeta}(t)$, one can deduce that $\partial_{tt}\bfeta_h$ is uniformly bounded in $L^2(0,T; W^{-k_0,2}(\Omega_S))$, so one can use the Aubin-Lions lemma to deduce that 
\begin{eqnarray}\label{eta:t:strong}%%%%%%%%%%%%%
    \partial_t\bfeta_h \to  \partial_t\bfeta,\qquad &\text{in } L^2(0,T; L^2(\Omega_S)),
\end{eqnarray}%%----------------------------%%
which combined with the bound $\frac1{2h}\int_{0}^T\int_{\Omega_S^{\bfeta_h}(t)}|\bb{u}_h - \partial_t \bfeta_h\circ \bfeta_h^{-1}|^2\leq C$ implies
\begin{eqnarray*}%%%%%%%%%%%%%
    \bb{u}_h \to \partial_t \bfeta\circ \bfeta^{-1}, \qquad \text{in } L^2((0,T)\times \Omega_S^\bfeta(t)).
\end{eqnarray*}%%----------------------------%%
By using $||\bb{u}_h||_{L^2(0,T; L^6(\Omega))}\leq C(\varepsilon)$ and $||\partial_t \bfeta_h\circ \bfeta_h^{-1}||_{L^2(0,T; L^6(\Omega_S^{\bfeta_h}(t)))}\leq C(\varepsilon)$, one then has
\begin{eqnarray}\label{str:mom:conv2}%%%%%%%%%%%%%
    \bb{u}_h \to \partial_t \bfeta\circ \bfeta^{-1}, \qquad \text{in } L^2(0,T; L^p(\Omega_S^\bfeta(t))), \quad \text{ for any } p<6,
\end{eqnarray}%%----------------------------%%
which gives us
\begin{eqnarray}\label{str:mom:conv}%%%%%%%%%%%%%
    \rho_h \bb{u}_h  \to \rho \partial_t \bfeta\circ \bfeta^{-1}, \qquad \text{in } L^2(0,T; L^{p}(\Omega_S^\bfeta(t))),\quad \text{ for any } p<\frac{6\gamma}{\gamma+6}.
\end{eqnarray}%%----------------------------%%
From the coupled momentum equation, one can also deduce that
\begin{eqnarray}\label{time:der:est}%%%%%%%%%%%%%
    ||\partial_t(\rho_h\bb{u}_h)||_{ L^2(0,T;W^{-k_0,2}(\Omega_F^\bfeta(t)))} \leq C(\varepsilon),\qquad \text{for some } p>1,
\end{eqnarray}%%----------------------------%%
and since $\rho_h \bb{u}_h$ is bounded in $L^2(0,T; L^{\frac{6\gamma}{4\gamma+3}}(\Omega))$ and $L^{\frac{6\gamma}{4\gamma+3}}(\Omega) \hookrightarrow\hookrightarrow H^{-s}(\Omega)$ for some $s<1$, we deduce by the Aubin-Lions lemma
\begin{eqnarray}\label{fl:mom:conv}%%%%%%%%%%%%%
    \rho_h \bb{u}_h \to \rho \bb{u} \qquad \text{in } L^2(0,T; H^{-1}(\Omega_F^\bfeta(t))).
\end{eqnarray}%%----------------------------%%

Finally, due to the estimate $\sqrt{\varsigma}||\nabla\rho_h||_{L^2((0,T)\times \Omega)}\leq C$, one has that $\varsigma \nabla \rho\rightharpoonup 0$. This means that the limiting functions $\rho,\bb{u}$ satisfy the continuity equation in the weak form. Note that at this point, there is still a term in the continuity equation $\overline{{\rho}^{2}}$, which is the weak limit of $\rho_h^2$ in $(0,T)\times\Omega_S(t)$. Moreover, this equation holds on the entire domain $\Omega$. In order to show that $\rho$ satisfies $\eqref{reconteqweakh}$, let us prove that the density vanishes inside the structure domain:
\begin{lem}\label{vanish:density}
We have
\begin{eqnarray*}%%%%%%%%%%%%%
    \rho=0 \text{ and } \overline{{\rho}^{2}} = 0, \qquad \text{a.e. in } (0,T)\times\Omega_S(t),
\end{eqnarray*}%%----------------------------%%
where $\overline{{\rho}^{2}}$ is the weak limit of $\rho_h^2$ in $(0,T)\times\Omega_S(t)$.
\end{lem}
\begin{proof}
Let $\tau\in(0,T)$. We choose $B(\rho)=1$ and $\varphi\in H^1(0,\tau; H_0^1(\Omega_S^\bfeta(t)))$ in the continuity equation to obtain
\begin{eqnarray*}%%%%%%%%%%%%%
    \int_0^\tau \int_{\Omega_S^\bfeta(t)} \rho (\partial_t \varphi+\bb{u}\cdot \nabla \varphi) =  \int_{\Omega_S^\bfeta(\tau)} (\rho \varphi)(\tau)- \underbrace{\int_{\Omega_S^{\bfeta_{0,\varepsilon}}(0)} \rho_{0,\varepsilon} \varphi}_{=0} + \int_0^\tau \int_{\Omega_S^\bfeta(t)} \overline{\rho^{2}} \varphi
\end{eqnarray*}%%----------------------------%%
since $\rho_{0,\varepsilon}=0$ on $\Omega_S(0)$, by construction. Now, recalling that $\bb{u} = \partial_t\bfeta\circ\bfeta^{-1}$ on $(0,T)\times \Omega_S(t)$, one has
\begin{eqnarray*}%%%%%%%%%%%%%
    (\partial_t \varphi+\bb{u}\cdot \nabla \varphi)(t,\bfeta(t,x)) = \frac{d}{dt} \varphi(t,\bfeta(t,x))
\end{eqnarray*}%%----------------------------%%
so one can compose both integrals with $\bfeta$
\begin{eqnarray*}%%%%%%%%%%%%%
    \int_0^t \int_{\Omega_S}\det\nabla\bfeta \tilde\rho \partial_t\tilde\varphi=  \int_{\Omega_S}(\det\nabla\bfeta \tilde\rho \tilde\varphi)(t) + \int_0^t\int_{\Omega_S}(\det\nabla\bfeta \tilde{\overline{{\rho}^{2}}} \tilde\varphi)
\end{eqnarray*}%%----------------------------%%
where $\tilde\rho = \rho\circ\bfeta$, $\tilde\varphi=\varphi\circ\bfeta$ and $\tilde{\overline{{\rho}^{2}}} = \overline{{\rho}^{2}}\circ \bfeta$. Now, for any $\tilde\varphi\in H^1(\Omega_S)$ independent of time (note that then $\varphi = \tilde\varphi\circ\bfeta^{-1}$ is a suitable test function), one has
\begin{eqnarray*}%%%%%%%%%%%%%
    \int_{\Omega_S}(\det\nabla\bfeta \tilde\rho\tilde\varphi)(\tau)+ \int_0^t\int_{\Omega_S}(\det\nabla\bfeta \tilde{\overline{{\rho}^{2}}}\tilde\varphi)=0.
\end{eqnarray*}%%----------------------------%%
We can now easily construct a sequence of functions $\tilde\varphi$ which converge to $1_{|\Omega_S}$, and since $\tilde\rho, \tilde{\overline{{\rho}^{2}}}\geq 0$ and $\det\nabla\bfeta>0$, this implies that $\tilde\rho = 0$ and $\tilde{\overline{{\rho}^{2}}} = 0$ on $\Omega_S$ and the conclusion follows.
\end{proof}

\bigskip
Finally, let us point out that since $\rho \in L^{\infty}(0,T;L^\beta(\Omega_F^\bfeta(t)))$, it means that $\rho$ is square integrable so one can use \cite[Lemma 6.9]{NovStr} to deduce that the limiting functions also satisfy the renormalized continuity equation $\eqref{reconteqweakh}$.

\subsection{Limit in the coupled momentum equation}
First, by using $\eqref{fl:mom:conv}$ and the weak convergence of $\bb{u}_h$ in $L^2(0,T;H^1(\Omega_F^\bfeta(t)))$, one deduces that 
\begin{eqnarray*}%%%%%%%%%%%%%
    \rho_h \bb{u}_h \otimes \bb{u}_h \rightharpoonup \rho \bb{u}\otimes \bb{u}, \qquad \text{in } L^1((0,T) \times \Omega_F^\bfeta(t)),
\end{eqnarray*}%%----------------------------%%
while combining $\eqref{str:mom:conv}$ and $\eqref{str:mom:conv2}$ implies
\begin{eqnarray*}%%%%%%%%%%%%%
    \rho_h \bb{u}_h \otimes \bb{u}_h \rightharpoonup \rho \partial_t \bfeta\circ \bfeta^{-1}\otimes \partial_t \bfeta\circ \bfeta^{-1} = \rho \bb{u}\otimes \bb{u}, \qquad \text{in } L^1((0,T) \times \Omega_S^\bfeta(t)),
\end{eqnarray*}%%----------------------------%%
so
\begin{eqnarray*}%%%%%%%%%%%%%
    \rho_h \bb{u}_h \otimes \bb{u}_h \rightharpoonup \rho \bb{u}\otimes \bb{u}, \qquad \text{in } L^1((0,T) \times \Omega).
\end{eqnarray*}%%----------------------------%%
Next, due to the uniform estimate on $\bb{u}_h$ in $L^2(0,T; H^1(\Omega))$ and the strong convergence of $\bfeta_h$ in $C([0,T]\times \Omega_S)$, one deduces
\begin{eqnarray*}%%%%%%%%%%%%%
    \mathbb{S}_{\varepsilon,\bfeta_h}(\nabla \bb{u}_h)\to \mathbb{S}_{\varepsilon,\bfeta}(\nabla \bb{u}), \qquad \text{in } L^2(0,T; L^2(\Omega)).
\end{eqnarray*}%%----------------------------%%

Finally, due to $\eqref{time:der:est}$ and uniform boundedness of $\rho_h\bb{u}_h$ in $L^\infty(0,T; L^{\frac{2\gamma}{\gamma+1}}(\Omega))$, one deduces 
\begin{eqnarray*}%%%%%%%%%%%%%
    \rho_h \bb{u}_h \to \rho \bb{u}, \qquad \text{in } L^\infty(0,T;L^{\frac{2\gamma}{\gamma+1}}(\Omega_F^\bfeta(t))).
\end{eqnarray*}%%----------------------------%%
To show the convergence of the structure terms, first note that the test function $\boldsymbol\phi$ in $\eqref{coupledmomapp}$ depends on $\bfeta_h$, so the derivatives of $\boldsymbol\phi$ depend on the derivatives of $\bfeta_h$. The following convergences hold:
\begin{eqnarray*}%%%%%%%%%%%%%
    \partial_t\bfeta_h \rightharpoonup  \partial_t\bfeta,& \qquad &\text{in } L^2(0,T; W^{k_0,2}(\Omega_S)), \\
    \bfeta_h\rightharpoonup  \bfeta,& \qquad &\text{in } L^\infty(0,T; W^{k_0,2}(\Omega_S))\\
    \bfeta_h\to \bfeta,& \qquad &\text{in } L^2(0,T; W^{k_0,2}(\Omega_S)),
\end{eqnarray*}%%----------------------------%%
where the last convergence is due to the Minty property and it follows in the same way as in \cite[Proposition 2.23]{benevsova2020variational}. 
This combined with $\eqref{eta:t:strong}$ allows us to pass to the limit all of the structure terms.

Finally, it remains to show the convergence of pressure.

\subsubsection{Convergence of pressure}
In order to prove that $\rho_h^\gamma+\varepsilon \rho_h^\beta$ has a weak limit $\overline{p}$ in $L^1(0,T \times \Omega_S^\bfeta(t))$, we utilize the bound $\int_0^T\int_{\Omega_S^\bfeta(t)}(\rho_h^{\gamma+1}+\varepsilon\rho_h^{\beta+1}) \leq C$. Due to Lemma $\ref{vanish:density}$, one has that $\overline{p}=0$ on $(0,T)\times \Omega_S^\bfeta(t)$. so it remains to prove the weak convergence of $\rho_h^\gamma+\varepsilon \rho_h^\beta$ on the fluid domain $(0,T)\times \Omega_F^\bfeta(t)$:

\begin{lem}\label{improve:est:int}
Let $I$ be an interval and $S\Subset \Omega_F^\bfeta(t)$ for all $t\in I$ be a set with regular boundary. Then, there is an $h_S>0$ such that for all $0<h<h_S$, one has the following improved pressure estimates:
\begin{eqnarray*}%%%%%%%%%%%%%
    \int_{I\times S} \rho_h^{\gamma+1}+\varepsilon \rho_h^{\beta+1} \leq C(S),
\end{eqnarray*}%%----------------------------%%
where $C(S)$ depends on the initial energy and blows up as $\text{dist}(S,\partial\Omega_F^{\bfeta_h}(t))$ goes to zero.
\end{lem}
\begin{proof}
The proof is done by testing $\eqref{momeqweakh}$ with $\boldsymbol\xi:=\varphi \nabla \Delta_{S'}^{-1}[\rho_h]$, where $S\subset S'\Subset \Omega_F^\bfeta(t)$ for all $t\in I$ is a set with regular boundary, $\varphi$ is a smooth cut-off function which is equal to $1$ on $S$ and $0$ outside of $S'$, and $\Delta_S^{-1}$ is the inverse Laplace operator on $S$ with Dirichlet data. At this point, this proof is standard and can be found in \cite[Lemma 6.3]{compressible}.
\end{proof}

\begin{lem}\label{improve:est:bnd}
There is a $h_0>0$ such that for every $\varepsilon'>0$, there is a $\delta'>0$, such that
\begin{eqnarray*}%%%%%%%%%%%%%
     \int_{0}^{T} \int_{\{d(x, \partial\Omega_F^{\bfeta_h}(t))<\delta'\}} \rho_h^{\gamma}+\varepsilon \rho_h^\beta \leq C(\varepsilon) \varepsilon',
\end{eqnarray*}%%----------------------------%%
for all $\varepsilon\leq \varepsilon_0$.
\end{lem}
\begin{proof}
The proof follows the approach from \cite{kukucka}, so only the main steps are presented. Let $h>0$ and $t_0\in [0,T]$ be fixed. First note that, due to the regularity of $\bfeta_h$ and $\partial\Omega$, for all $t\in(t_0-s,t_0+s)\cap (0,T)$ where $s>0$, the boundary $\partial\Omega_F^{\bfeta_h}(t)$ can be represented as a union of graphs of functions $\phi_i: (t_0-s,t_0+s)\times V_i \to \mathbb{R}$ in local coordinate systems $(x_i^1,x_i^2,x_i^3)$, where $1\leq i \leq m$, so that the fluid domain lies below all the graphs. In other words, for every $(t,x)\in (t_0-s,t_0+s)\times (\partial\Omega_S\cup \partial\Omega)$, there is an $1\leq i \leq m$ and $y\in V_i$ such that $\bfeta_h(t,x)=\bb{s}_i((y,0)+\phi_i(t)(y)\bb{e}_3^i)$ (and similarly for every $(t,x)\in (t_0-s,t_0+s)\times \partial\Omega$), where $\bb{e}_3^i$ is the unit vector in $x_i^3$ direction in the corresponding local coordinate system, and $\bb{s}_i(\bb{v})=\bb{r}_i+A_i \bb{v}$ is a combination of translation $\bb{r}_i$ and rotation $A_i$. Note that we can choose a uniform $h_0>0$ such that these graphs are well-defined on $(t_0-s,t_0+s)$ for all $0<h<h_0$. We define
\begin{eqnarray*}%%%%%%%%%%%%%
    v_i(t,x_i^1,x_i^2,x_i^3):=\begin{cases} (x_i^3 - \phi_i(t,x_1,x_2))^\lambda,&  \quad x_i^3 \leq  \phi_i(t,x_1,x_2), \\
    0,& \quad \text{elsewhere},
    \end{cases}
\end{eqnarray*}%%----------------------------%%
for $\frac{2\gamma-3}{6\gamma}<\lambda<1$, and
\begin{eqnarray*}%%%%%%%%%%%%%
    \boldsymbol\xi_h(t,x):= \sum_{i=1}^m 
\psi_i(x) v_i (t, \bb{s}_i^{-1}(x))\bb{s}_i (\bb{e}_3^i).
\end{eqnarray*}%%----------------------------%%
Here $\{\psi_i\}_{1\leq i \leq m+2}$ is a partition of unity, subordinate to sets $\Omega_1,...,\Omega_{m+2}$ such that $\cup_{i=1}^{m+2} \Omega_i\supseteq \Omega$, for all $t\in (t_0-s,t_0+s)$ the sets $\Omega_1,...,\Omega_m$ contain the graphs of $m$ functions $\phi_1(t),...,\phi_m(t)$, $\Omega_{m+1}\subset \text{int} \cap_{t\in (t_0-s,t_0+s)} \Omega_F^{\bfeta_h}(t)$ and $\Omega_{m+2}\subset \text{int} \cap_{t\in (t_0-s,t_0+s)} \Omega_S^{\bfeta_h}(t)$. Additionally, each of the sets $\Omega_i$ for $1\leq i \leq m$ can only have a non-empty intersection with $\Omega_{m+1}, \Omega_{m+2}$ and $\Omega_j$ such that $\phi_i$ and $\phi_j$ are neighbouring graphs.

Note that $\boldsymbol\xi_h$ vanishes outside of $\Omega_F^{\bfeta_h}$. Moreover, $\boldsymbol\xi_h\in L^\infty(0,T; W^{1,q}(\Omega))\cap H^1(0,T; L^4(\Omega))$ for any $q<\frac1{1-\lambda}$, due to the regularity of ${\bfeta_h}_{|\partial\Omega_S^{\bfeta_h}}$ and the fact that it behaves as a distance function to power $\lambda$ near the boundary. Due to the specific lower bound for $\lambda$, we also have that $\boldsymbol\xi_h\in L^\infty(0,T; W^{1,\frac{6\gamma}{4\gamma+3}}(\Omega))$. 

Finally, note also that we can extend this construction to the entire interval $[0,T]$ by covering $[0,T]$ with a finite number of time intervals of length $2s$ and using the partition of unity in time $\{\chi_i\}_{1\leq i\leq k}$ (here the key point is that $s$ can be chosen uniformly due to the regularity of $\bfeta_h$). Let the new function defined on $[0,T]$ be also denoted by $\boldsymbol\xi_h(t,x) = \sum_{j=1}^k \sum_{i=1}^m \chi_j(t)
\psi_i^j(x) v_i^j (t, (\bb{s}_i^j)^{-1}(x))\bb{s}_i^j (\bb{e}_3^{i,j})$. Now, by testing $\eqref{momeqweakh}$ with $\boldsymbol\xi_h $, one obtains
\begin{eqnarray*}%%%%%%%%%%%%%
   &&\int_0^T \int_{\Omega_F^{\bfeta_h}(t)} (\rho_h^{\gamma}+\varepsilon \rho_h^\beta) \nabla\cdot \boldsymbol\xi_h \\
   &&=-\int_0^T\int_{\Omega_F^{\bfeta_h}(t)}  \rho_h \bb{u}_h \cdot\partial_t \boldsymbol\xi_h -\int_0^T\int_{\Omega_F^{\bfeta_h}(t)} (\rho_h \bb{u}_h \otimes \bb{u}_h):\nabla\boldsymbol\xi_h \nonumber \\
   &&\quad  + \int_0^T\int_{\Omega_F^{\bfeta_h}(t)}  \mathbb{S}_{\varepsilon,\bfeta_h}(\nabla\bb{u}_h): \nabla \boldsymbol\xi_h - \frac{1}2 \int_0^T\int_{\Omega_F^{\bfeta_h}(t)}  \rho_h \bb{u}_h \cdot \boldsymbol\xi_h+\int_{\Omega_F^{\bfeta_h}(t)}  \rho_h\mathbf{u}_h\cdot\boldsymbol\xi_h\Big|_0^T.
\end{eqnarray*}%%----------------------------%%
Due to the estimates on $\rho_h,\bb{u}_h$ and $\boldsymbol\xi_h$ (note that the term $\rho_h \bb{u}_h \cdot \partial_t \boldsymbol\xi_h$ is uniformly bounded due to $\gamma>12/7$ even w.r.t. $\varepsilon$), the right-hand side is uniformly bounded for a fixed $\varepsilon$, so
\begin{eqnarray*}%%%%%%%%%%%%%
    \int_0^T \int_{\Omega_F^{\bfeta_h}(t)} (\rho_h^{\gamma} +\varepsilon \rho_h^\beta)\nabla\cdot \boldsymbol\xi_h \leq C(\varepsilon).
\end{eqnarray*}%%----------------------------%%
The $\varepsilon$ here appears due to the penalization $K_\varepsilon$, which directly affects the shape of the fluid domain (i.e. how close to contact and self-contact can the elastic body be) and therefore the construction of $\boldsymbol\xi_h$ and its bounds. This implies
\begin{eqnarray*}%%%%%%%%%%%%%
    &&\int_0^T \int_{\Omega_F^{\bfeta_h}(t)} \underbrace{(\rho_h^{\gamma} +\varepsilon \rho_h^\beta)\sum_{j=1}^k \sum_{i=1}^m \chi_j(t)
    \psi_i^j(x)\nabla\cdot \big(v_i^j (t, (\bb{s}_i^j)^{-1}(x))\bb{s}_i^j (\bb{e}_3^{i,j})}_{\geq 0
    }\big) \\
    && \leq C(\varepsilon) - \int_0^T \int_{\Omega_F^{\bfeta_h}(t)} (\rho_h^{\gamma}+\varepsilon \rho_h^\beta) \sum_{j=1}^k \sum_{i=1}^m \chi_j(t)
    \nabla\psi_i^j(x) \cdot v_i^j (t, (\bb{s}_i^j)^{-1}(x))\bb{s}_i^j (\bb{e}_3^{i,j})\\
    &&\leq C(\varepsilon).
\end{eqnarray*}%%----------------------------%%
The key point here is that $\nabla\cdot (v_i^j (t, (\bb{s}_i^j)^{-1}(x))\bb{s}_i^j (\bb{e}_3^{i,j})) = \lambda d_i^j(t,x)^{\lambda-1}$, where $d_i^j(t,x)$ is the distance of the point $x$ and the boundary $\partial\Omega_S^{\bfeta_h}(t)$ in the $\bb{s}_i^j (\bb{e}_3^{i,j})$-direction. These divergences blow up as one approaches the boundary, so one has that for every $\varepsilon'>0$, there is a $\delta'>0$ such that
\begin{eqnarray*}%%%%%%%%%%%%%
    \sum_{i=1}^m 
    \psi_i(x)\nabla\cdot \big(v_i (t, \bb{s}_i^{-1}(x))\bb{s}_i (\bb{e}_3^i)\big) \geq \frac{1}{\varepsilon'}, \quad \text{ for all } t \in [0,T] \text{ and } x \text { such that } d(x, \partial\Omega_F^{\bfeta_h}(t)) < \delta',
\end{eqnarray*}%%----------------------------%%
which implies
\begin{eqnarray*}%%%%%%%%%%%%%
    &&\frac{1}{\varepsilon'}\int_0^T \int_{\{d(x, \partial\Omega_F^{\bfeta_h}(t)) < \delta'\}} (\rho_h^{\gamma}+\varepsilon \rho_h^\beta)\\
    &&\leq \int_0^T \int_{\{d(x, \partial\Omega_F^{\bfeta_h}(t)) < \delta'\}} (\rho_h^{\gamma} +\varepsilon \rho_h^\beta)\sum_{j=1}^k \sum_{i=1}^m \chi_j(t)
    \psi_i^j(x)\nabla\cdot \big(v_i^j (t, (\bb{s}_i^j)^{-1}(x))\bb{s}_i^j (\bb{e}_3^{i,j}) \leq C(\varepsilon),
\end{eqnarray*}%%----------------------------%%
so the proof is finished.
\end{proof}

\begin{cor}\label{cor:lim}
There is a function $p$ such that
\begin{eqnarray*}%%%%%%%%%%%%%
    \rho_h^\gamma+\varepsilon \rho_h^\beta \rightharpoonup \overline{p} \text{ in } L^1(0,T; L^1(\Omega_F^\bfeta(t))).
\end{eqnarray*}%%----------------------------%%
\end{cor}
\begin{proof}
For any $t_0\in [0,T]$ and all sets $S\Subset \Omega_F^\bfeta(t_0)$, there is an $h_0>0$ and a time interval $(t_0-s,t_0+s)$ such that for all $t\in (t_0-s,t_0+s)\cap [0,T]$ and $0<h<h_0$ one has $S\Subset \Omega_F^\bfeta(t)$ and $S\Subset \Omega_F^{\bfeta_h}(t)$. On each of these cylinders $(t_0-s,t_0+s)\times S$, one can use Lemma $\ref{improve:est:int}$ to conclude that there is a function $\overline{p}$ such that
\begin{eqnarray*}%%%%%%%%%%%%%
     \rho_h^\gamma+\varepsilon \rho_h^\beta \rightharpoonup \overline{p} \text{ in } L^1((t_0-s,t_0+s)\times S).
\end{eqnarray*}%%----------------------------%%
This defines the function $\overline{p}$ on the entire set $(0,T) \times \Omega_F^\bfeta(t)$. Thus, for every $\boldsymbol\xi \in C^\infty([0,T]\times \Omega)$ and every $\varepsilon'>0$ and $\delta'>0$ as in Lemma $\ref{improve:est:bnd}$, one has
\begin{eqnarray*}%%%%%%%%%%%%%
    \int_0^T \int_{\Omega_F^{\bfeta_h}(t)} (\rho_h^\gamma+\varepsilon \rho_h^\beta) \varphi  &=& \int_0^T \int_{{\{d(x, \partial\Omega_F^{\bfeta_h}(t))\geq \delta'\}}} (\rho_h^\gamma+\varepsilon \rho_h^\beta) \varphi+\int_0^T \int_{\{d(x, \partial\Omega_F^{\bfeta_h}(t))<\delta'\}} (\rho_h^\gamma+\varepsilon \rho_h^\beta) \varphi \\
     &\to& \int_0^T \int_{\{d(x, \partial\Omega_F^\bfeta(t))\geq \delta'\}}\overline{p} \varphi+\int_0^T\langle \pi, \varphi\rangle_{[\mathcal{M}^+,C]({\{d(x, \partial\Omega_F^\bfeta(t))< \delta'\}})}, \quad \text{as } h\to 0,
\end{eqnarray*}%%----------------------------%%
where $\pi$ is the weak limit of $\rho_h^\gamma$ on the set $[0,T]\times {\{d(x, \partial\Omega_F^\bfeta(t))< \delta'\}}$. Due to Lemma $\ref{improve:est:bnd}$, one has
\begin{eqnarray*}%%%%%%%%%%%%%
    \int_0^T\langle \pi, \varphi\rangle_{[\mathcal{M}^+,C]({\{d(x, \partial\Omega_F^\bfeta(t))< \delta'\}})} \leq C(\varepsilon)\varepsilon',
\end{eqnarray*}%%----------------------------%%
so by passing to the limit $\varepsilon'\to
 0$ the conclusion follows.
\end{proof}

One can use the standard weak compactness arguments which are based on the convergence of effective viscous flux, 
the renormalized continuity equation and the convexity of the function $x\mapsto x\ln x$ to conclude that $\rho_h\to \rho$ a.e. in $(0,T)\times \Omega_F^\bfeta(t)$. Note that here we use the fact that $\rho_h$ and $\bb{u}_h$ satisfy the following damped renormalized continuity equation
\begin{eqnarray}%%%%%%%%%%%%%   
    &&\int_0^T \int_\Omega  \rho B(\rho_h)( \partial_t \varphi +\bb{u}_h\cdot \nabla \varphi) -  \int_{0}^{T} \bint_{\Omega_S^\bfeta(t)}\rho_h^{2} \left(B(\rho_h)+\frac{b(\rho_h)}{\rho_h}\right) \varphi  \nonumber\\
    &&\quad- \varsigma \int_0^T \left(B(\rho_h)+\frac{b(\rho_h)}{\rho_h}\right) \nabla \rho_h \cdot \nabla \varphi  -  \varsigma \int_0^T \rho_h B(\rho_h) |\nabla \rho_h|^2  \varphi  =\int_{\Omega_F^{\bfeta_0}}\rho_{0,\varepsilon} B(\rho_{0,\varepsilon}) \varphi(0,\cdot). \nonumber
\end{eqnarray}%%----------------------------%%
The last term on the LHS (which is the only non-standard one) does not affect this proof since it is positive and vanishes in the limit $h\to 0$. This approach is standard nowadays and is due to Lions \cite{LionsBook} for $\gamma\geq \frac95$ and Feireisl et. al. \cite{FeireislCompressible01} for $\gamma \in (\frac32, \frac95)$. These arguments are used in localized form away from the boundary $\partial\Omega_F^{\bfeta_h}$, and as such are completely unaffected by the presence of the solid (see for example \cite[Section 6.2 \& 6.4]{compressible}). \\

\subsection{Global Eulerian velocity field}
Now that we have passed to the limit $h\to 0$, we have $\bb{u} = \partial_t\bfeta\circ\bfeta^{-1}$ on $Q_{S,T}^\bfeta$, so it is convenient to introduce a global velocity field
\begin{eqnarray*}%%%%%%%%%%%%%
    \bb{v}(t,x):=\begin{cases}
    \bb{u}(t,x),& \text{for } x\in \Omega_F^\bfeta(t),\\
    \partial_t\bfeta\circ\bfeta^{-1}(t,x), &\text{for } x\in \Omega_S^\bfeta(t).
    \end{cases}
\end{eqnarray*}%%----------------------------%%
First, note that due to the regularity and invertibility of $\bfeta$, one has
\begin{eqnarray*}%%%%%%%%%%%%%
    || \mathbb{S}(\nabla  (\partial_t\bfeta\circ\bfeta^{-1})) ||_{L^2((0,T)\times \Omega_S^\bfeta(t))} \leq C,
\end{eqnarray*}%%----------------------------%%
which combined with the bound from the energy inequality
\begin{eqnarray*}
     || \mathbb{S}(\nabla \bb{u} ) ||_{L^2((0,T)\times \Omega_F^\bfeta(t))} \leq C
\end{eqnarray*}
implies
\begin{eqnarray*}%%%%%%%%%%%%%
    || \mathbb{S}(\nabla \bb{v} ) ||_{L^2((0,T)\times \Omega)} \leq C.
\end{eqnarray*}%%----------------------------%%
Thus, taking into consideration that $\bb{u}$ (and consequently $\bb{v}$) vanishes on $\partial\Omega$, one has by the Korn identity
\begin{eqnarray}%%%%%%%%%%%%%
    ||\bb{v}||_{L^2(0,T; H^1(\Omega))} \leq C, \label{v:H1:est}
\end{eqnarray}%%----------------------------%%
and consequently by imbedding
\begin{eqnarray*}%%%%%%%%%%%%%
    ||\bb{v}||_{L^2(0,T; L^6(\Omega))} \leq C,
\end{eqnarray*}%%----------------------------%%
where the constant $C$ only depends on the initial data and domain $\Omega$. Note that this constant does not depend on the distance between $\partial\Omega_F^\bfeta(t)$ and $\partial\Omega$ or possibly the shape of $\Omega_F^\bfeta(t)$, so this regularity is preserved even when $\varepsilon$ goes to zero and a possible contact appears.

\section{Vanishing artificial fluid diffusion, artificial pressure, structure regularization and vanishing contact-penalization limit $\varepsilon\to 0$}\label{sec:last:limit}
The solution obtained in the previous section for a given $\varepsilon>0$ in the sense of Definition $\ref{weaksolh}$ will now be denoted as $(\bfeta_\varepsilon,\rho_\varepsilon,\bb{u}_\varepsilon)$. The goal is to pass to the limit $\varepsilon\to 0$ and to prove that the limiting functions $(\bfeta,\rho,\bb{u})$ are a weak solution in the sense of Definition \ref{weaksol}. First, let us point out that due to convergence of the global Eulerian velocity $\bb{v}_\varepsilon$ in $L^2(0,T; H^1(\Omega))$, both kinematic coupling between fluid and the structure holds and the velocities at contact points match. \\

Most convergences are the same as in Section $\ref{sec:h:lim}$, and are thus omitted. We will first prove that the force appearing from the contact-penalization $DK_\varepsilon$ vanishes. This will result in possible contact of the solid with the rigid boundary and self-contact, which a priori creates irregularities in the fluid domain, such as for example cusps. Thus the convergence of pressure is different in this case, since we cannot exclude concentrations at these cusps.

\subsection{Convergence of the solid terms}
\subsubsection{Convergence of the contact-penalization}

In order to find the limit of the contact-penalization, we will observe the two terms separately.  %we combine two observations both a direct consequence of the fact that the test function of the solid is derived from a global Eulerian test function and thus, as it turns out is not able to sense the contact force (see Remark \ref{rem:testCoupling} on why this is not a problem for the model). ???

For the first term, corresponding to the collisions with the exterior boundary $\partial\Omega$, we observe that since the global test-function $\boldsymbol{\xi}$ has zero-boundary values, it is enough to consider the dense set of compactly supported functions instead. But then, for any fixed $\boldsymbol{\xi}$, at some point we have $\varepsilon < \dist(\supp \boldsymbol{\xi}, \partial \Omega)$, which means that for all $x\in \Omega_S$ such that $\kappa_\varepsilon'(\operatorname{dist}(\bfeta_\varepsilon(t,x),\partial \Omega)) \neq 0 $, we have $\boldsymbol{\xi}(\boldsymbol \bfeta (t,x)) = 0$. In other words, for any compactly supported test function $\boldsymbol{\xi}$ there is an $\varepsilon_0>0$ such that the first term resulting from the contact-penalization vanishes for all $\varepsilon\leq \varepsilon_0$.

For the self-contact-penalization the same is true, once we employ its symmetry. Fix $\delta > 0$ and $t_0\in [0,T]$. For any pair of points $x,y \in \Omega_S$ there are two possibilities. The first possibility is that $|\bfeta (t_0,x) - \bfeta(t_0,y)| =: \delta > 0$. In this case, for all $\varepsilon$ small enough and all $t, \tilde{x},\tilde{y}$ in a neighborhood of $t_0$, $x$ and $y$, respectively, one has $|\bfeta_\varepsilon (t,\tilde{x}) - \bfeta_\varepsilon(t,\tilde{y})| > \frac{\delta}{2}$ and consequently $\kappa_\varepsilon'(|\bfeta_\varepsilon (t,\tilde{x}) - \bfeta_\varepsilon(t,\tilde{y})|) = 0$ once $\varepsilon < \frac{\delta}{2}$. Therefore, these pairs $x,y$ do not contribute to the limit. The second possibility is that $|\bfeta (t_0,x) - \bfeta(t_0,y)| = 0$. In this case, assuming $x\neq y$, for any small neighborhood $I\times U \times V$ of $(t_0,x,y) \in [0,T]\times\Omega^S \times \Omega^S$, where w.l.o.g. $\dist(U,V) > \sqrt{\varepsilon}$, we can estimate by using $\eqref{eq:kappaGrowth}$
\begin{align*}
    \phantom{{}={}}& \left| \int_I\int_{U \times V} \kappa_\varepsilon'(|\bfeta_\varepsilon (t,\tilde{x}) - \bfeta_\varepsilon(t,\tilde{y})|) \frac{\bfeta_\varepsilon (t,\tilde{x}) - \bfeta_\varepsilon(t,\tilde{y})}{|\bfeta_\varepsilon (t,\tilde{x}) - \bfeta_\varepsilon(t,\tilde{y})|} \cdot ( \boldsymbol\phi(t,\tilde{x})- \boldsymbol\phi(t,\tilde{y})) d \tilde{x} d \tilde{y} dt \right|\\
     &\leq \int_I\int_{U \times V}  |\kappa_\varepsilon'(|\bfeta_\varepsilon (t,\tilde{x}) - \bfeta_\varepsilon(t,\tilde{y})|)|~ |\boldsymbol\xi(\bfeta_\varepsilon(t,\tilde{x}))- \boldsymbol\xi(\bfeta_\varepsilon(t,\tilde{x})) | d \tilde{x} d \tilde{y}dt\\
    &\leq \int_I\int_{U\times V}  |\kappa_\varepsilon'(|\bfeta_\varepsilon (t,\tilde{x}) - \bfeta_\varepsilon(t,\tilde{y})|)| \sup_{I\times\Omega} | \nabla \boldsymbol \xi | ~|\bfeta_\varepsilon (t,\tilde{x}) - \bfeta_\varepsilon(t,\tilde{y})| d \tilde{x} d \tilde{y} dt\\
    &\leq c\varepsilon \sup_{I\times\Omega} | \nabla \boldsymbol \xi |\int_I \int_{U\times V} \Big(\kappa_\varepsilon(|\bfeta_\varepsilon (t,\tilde{x}) - \bfeta_\varepsilon(t,\tilde{y})|) +1 \Big)d\tilde{x} d\tilde{y}dt \\
    & \leq c \sup_{I\times\Omega} | \nabla \boldsymbol \xi | \varepsilon \left( \int_I K_\varepsilon  +1 \right)  \leq c \sup_{I\times\Omega} | \nabla \boldsymbol \xi | \varepsilon (T E_0+1) \to 0,\quad \text{ as }\varepsilon\to 0,
\end{align*}
where $E_0$ is the uniform energy bound. Therefore, $\int_0^T\langle DK_\varepsilon(\bfeta_\varepsilon), \boldsymbol\phi \rangle\to 0$, for any test function as in $\eqref{momeqweak}$.

\subsubsection{Convergence of the solid regularization}

The convergence of the potential solid energy terms for vanishing regularization is more or less identical to that of those terms in \cite{benevsova2020variational} or \cite{breit2021compressible}. We will thus only sketch the proof.

First, using the same Aubin-Lions argument as before, we can obtain a strong convergence of $\partial_t\bfeta_\varepsilon \to \partial_t \bfeta$ in $L^2(Q_{S,T})$. With this in addition to the other convergences, we can then consider
\begin{align} \label{eq:minty}
    \lim_{\varepsilon\to 0} \int_0^T\langle DE_{\varepsilon}(\bfeta_\varepsilon)-DE(\bfeta), \psi (\bfeta_\varepsilon -\bfeta) \rangle
\end{align}
for some cut-off $\psi \in C_c(\Omega_S,[0,1])$. In expanding this, the crucial term will end up to be $\langle DE(\bfeta_\varepsilon), \psi(\bfeta_\varepsilon-\bfeta)\rangle$ as both sides depend on $\varepsilon$. All other terms converge using the standard weak limits already derived. For that last term, we now employ the weak equation. Due to the cutoff, we do not have to deal with fluid terms and all the solid-terms will either vanish or again be of lower order and thus subject to strong convergence. Thus the limit of \eqref{eq:minty} will be zero, which using the assumptions on $E$ then allows us to conclude convergence of the $DE$-term.

\subsection{Convergence of pressure}
\begin{lem}\label{int:est:pr:eps}
Let $I\subset [0,T]$ be an interval and $S\Subset \Omega_F^\bfeta(t)$ for all $t\in I$ be a set with regular boundary. Then, there is a $\varepsilon_S>0$ such that for all $0<\varepsilon<\varepsilon_S$, one has the following improved pressure estimates:
\begin{eqnarray*}%%%%%%%%%%%%%
    \int_{I\times S} \rho_\varepsilon^{\gamma+a}+\varepsilon\rho^{\beta+a} \leq C(S),
\end{eqnarray*}%%----------------------------%%
where $a=\frac23\gamma-1$ and $C(S)$ depends on the initial energy and blows up as $\text{dist}(S,\partial\Omega_F^\bfeta(t))$ goes to zero.
\end{lem}
\begin{lem}\label{loc:est:bnd}
Let $t_0\in [0,T]$ and $x_0\in \partial \Omega_F^\bfeta(t_0)\setminus \bfeta(\partial I_\bfeta(t_0))$. Then, there is an interval $I\ni t_0$ and a ball $B = B(x_0,r)$, for some $r>0$, such that the following holds. There is a $\varepsilon_{I,B}>0$ such that for every $\varepsilon'>0$, there is a $\delta'>0$, such that
\begin{eqnarray*}%%%%%%%%%%%%%
     \int_I \int_{B\cap\{d(x, \partial\Omega_F^{\bfeta_\varepsilon}(t))<\delta'\}} \rho_\varepsilon^{\gamma}+\varepsilon\rho^\beta \leq C(I,B) \varepsilon',
\end{eqnarray*}%%----------------------------%%
for all $\varepsilon<\varepsilon_{I,B}$, where constant $C(I,B)$ blows up as $\sup\limits_{t\in I}\text{dist}(x_0,\partial I_\bfeta(t))$ goes to zero.
\end{lem}
\begin{proof}
The proof is similar to the one of Lemma $\ref{improve:est:bnd}$, however we can't use the same construction since we cannot ensure that there is no contact with the rigid boundary or self-contact. The function that is constructed here is similar, but uses only one graph function instead of multiple ones.

First, by Lemma $\ref{lemma:contact1}$ Claim 4, one has that either $x_0\in \bfeta(t_0,I_\bfeta(t_0))$ or $x_0\in \partial\Omega\setminus \bfeta(t_0,C_\bfeta(t_0))$. Since the later option is simpler, we will choose the former, so there is an $y_0\in I_\bfeta(t_0)$ such that $\bfeta(t_0,y_0)=x_0$. Now, we can choose $r>0$ and an interval $I''\ni t_0$ so that for all $t\in I''$, $B(x_0,3r)\cap \partial \Omega_F^\bfeta(t)$ can be represented as a graph of a function $\phi:I''\times V\to \mathbb{R}$ so that fluid domain lies below it. Note that $B(x_0,3r)$ then intersects the boundary $\partial \Omega_F^\bfeta(t)$ only once, and moreover $\sup\limits_{t\in I''}\text{dist}(x_0,\bfeta(\partial I_\bfeta(t)))\geq 3r$. Now, the corresponding coordinate system and the coordinate transform are denoted as $(x^1,x^2,x^3)$ and $\bb{s}(\bb{v}):=\bb{r}+A\bb{v}$, with the later consisting of translation vector $\bb{r}$ and rotation matrix $A$. Then, due to the strong convergence of $\bfeta_{\varepsilon}$ in $C^{0,a}([0,T]; C^{1,b}(\Omega_S))$ for some $a,b>0$, there is an interval $t_0 \in I'\subset I''$ and an $\varepsilon_0>0$ so that for all $0<\varepsilon<\varepsilon_0$, one has that $B(x_0,2r)\cap \partial \Omega_F^{\bfeta_\varepsilon}(t)$ can be represented as a graph of a function $\phi_\varepsilon: I'\times V\to \mathbb{R}$ in the same coordinate system $(x^1,x^2,x^3)$ as $\phi$. Finally, we choose some interval $t_0\in I\subset I'$ and define a cut-off function $\psi(t,x)$ that is $1$ on $I\times B(x_0,r)$, and vanishes outside of $I'\times B(x_0,2r)$. Similarly as in Lemma $\ref{improve:est:bnd}$, we define 
\begin{eqnarray*}%%%%%%%%%%%%%
    v_\varepsilon(t,x^1,x^2,x^3):=\begin{cases} (x^3 - \phi(t,x^1,x^2))^\lambda,&  \quad x^3 \leq  \phi(t,x^1,x^2), \\
    0,& \quad \text{elsewhere},
    \end{cases}
\end{eqnarray*}%%----------------------------%%
for $\frac{2\gamma-3}{6\gamma}<\lambda<1$, and
\begin{eqnarray*}%%%%%%%%%%%%%
    \boldsymbol\xi_\varepsilon(t,x):=
\psi(t,x) v (t, \bb{s}^{-1}(x))\bb{s} (\bb{e}_3).
\end{eqnarray*}%%----------------------------%%
Testing the momentum equation $\eqref{momeqweakh}$ with $\boldsymbol\xi_\varepsilon$, the desired inequality follows in the same way as in Lemma $\ref{improve:est:bnd}$.
\end{proof}
We can now conclude:
\begin{cor}
As $\varepsilon\to 0$, we have $ \rho_\varepsilon^\gamma+\varepsilon\rho^\beta\rightharpoonup \overline{p}+\pi$ in $ L_{w^*}^\infty(0,T; \mathcal{M}^+(\overline{\Omega_F^\bfeta(t)}))$, where $\overline{p}\in L^\infty(0,T; L^1(\Omega_F^\bfeta(t)))$ and $\pi \in  L_{w^*}^\infty(0,T; \mathcal{M}^+(\overline{\Omega_F^\bfeta(t)}))$ is supported on the set $(0,T]\times  \bfeta(\partial I_\bfeta(t))$.
\end{cor}
\begin{proof}
Denote by $\pi$ the limit of $\rho_\varepsilon^\gamma+\varepsilon\rho^\beta$ in $ L_{w^*}^\infty(0,T; \mathcal{M}^+(\overline{\Omega_F^\bfeta(t)}))$ and let the function $\overline{p}\in L^\infty(0,T; L^1(\Omega_F^\bfeta(t)))$ be defined by weak convergence on all sets $I\times S$ as in Lemma $\ref{int:est:pr:eps}$. 

Fix a $\delta>0$. Since for all $t\in[0,T]$, $\bfeta(I_\bfeta(t))$ is a set of measure $\leq 2$, we can cover $\cup_{t\in (0,T]} \{t\}\times\bfeta(\partial I_\bfeta(t))$ with an open set $S_\delta\subset \mathbb{R}^4$ such that $|S_\delta|<\delta$. Now, for any pair $(t_i,x_i)$ such that $t_i\in (0,T)$ and $x_i\in \partial \Omega_F^\bfeta(t)\setminus \bfeta(\partial I_\bfeta(t_i))$, one has by Lemma $\ref{lemma:contact1}$ Claim 4 that either $x_i\in \bfeta(t_i,I_\bfeta(t_i))$ or $x_0\in \partial\Omega\setminus \bfeta(t_i,C_\bfeta(t_i))$. Therefore, by Lemma $\ref{loc:est:bnd}$, there is an $\varepsilon_i>0$, an interval $I_i\ni t_i$ and a ball $B_i= B(\bfeta(t_i,x_i),r_i)$, for some $r_i>0$, such that for all $\varepsilon'>0$, there is a $\delta'>0$ such that
\begin{eqnarray*}%%%%%%%%%%%%%
    \int_I \int_{B_i\cap\{d(x, \partial\Omega_F^{\bfeta_\varepsilon}(t))<\delta'\}} \rho_\varepsilon^{\gamma}+\varepsilon\rho^\beta \leq C_i \varepsilon',
\end{eqnarray*}%%----------------------------%%
for all $0<\varepsilon<\varepsilon_i$. Since $\big([0,T]\times\partial\Omega_F^\bfeta(t)\big)\setminus S_\delta$ is compact, we can choose a finite family $\{I_{i_k}\times B_{i_k}\}_{1\leq k\leq m}$ which covers this set. Denoting $\varepsilon_m=\min\{\varepsilon_{i_k}\}_{1\leq k \leq m}$ and $C_m=\max\{C_{i_k}\}_{1\leq k \leq m}$, one obtains
\begin{eqnarray*}%%%%%%%%%%%%%
    \int_{([0,T]\times\{d(x, \partial\Omega_F^{\bfeta_\varepsilon}(t))<\delta' \})\setminus S_\delta} \rho_\varepsilon^{\gamma}+\varepsilon\rho^\beta
  \leq \sum_{j=1}^m\int_{I_j} \int_{B_j\cap\{d(x, \partial\Omega_F^{\bfeta_\varepsilon}(t))<\delta'\}} \rho_\varepsilon^{\gamma} \leq m C_m \varepsilon',
\end{eqnarray*}%%----------------------------%%
for all $\varepsilon<\varepsilon_m$. Similarly as in Corollary $\ref{cor:lim}$, one has 
\begin{eqnarray*}%%%%%%%%%%%%%
    \int_0^T \int_{([0,T]\times\Omega_F^{\bfeta_\varepsilon}(t))\setminus S_\delta} (\rho_\varepsilon^\gamma +\varepsilon\rho^{\beta+a})\boldsymbol\xi &=&   \int_{([0,T]\times\{d(x, \partial\Omega_F^{\bfeta_\varepsilon}(t))\geq \delta'\})\setminus S_\delta} (\rho_\varepsilon^\gamma +\varepsilon\rho^\beta)\varphi\\
    &&\quad +\int_{([0,T]\times\{d(x, \partial\Omega_F^{\bfeta_\varepsilon}(t))<\delta'\})\setminus S_\delta} (\rho_\varepsilon^\gamma+\varepsilon\rho^{\beta})\varphi \\
    && \to \int_0^T \int_{([0,T]\times\{d(x, \partial\Omega_F^\bfeta(t)\geq \delta'\}))\setminus S_\delta}\overline{p} \varphi\\
    &&\quad +\int_0^T\langle \pi, \varphi\rangle_{[\mathcal{M}^+,C](\{d(x, \partial\Omega_F^\bfeta(t))<\delta'\}]\setminus S_\delta(t))}, \quad \text{as } \varepsilon\to 0,
\end{eqnarray*}%%----------------------------%%
for all $\boldsymbol\xi \in C^\infty([0,T]\times \Omega)$, where $S_\delta(t)=\{x\in \mathbb{R}^3: (t,x)\in S_\delta\}$ is the slice at time $t$, so once again
\begin{eqnarray*}%%%%%%%%%%%%%
    \int_0^T\langle \pi, \varphi\rangle_{[\mathcal{M}^+,C](\{d(x, \partial\Omega_F^\bfeta(t))<\delta'\}]\setminus S_\delta(t))} \leq  m C_m \varepsilon',
\end{eqnarray*}%%----------------------------%%
so by passing to the limit $\varepsilon'\to 0$, one has that $\pi = \overline{p}$ on $([0,T]\times \Omega_F^\bfeta(t))\setminus S_\delta$ for any\footnote{Here $j$ and $C_m$ might depend on $\delta$. This however does not affect the limit $\varepsilon'\to 0$.} $\delta>0$. Finally, passing to the limit $\delta \to 0$, the desired conclusion follows and the proof is complete.
\end{proof}
\begin{rem}
Note that the defect measure $\pi$ in Definition $\ref{weaksol}$ is supported only on $(0,T]\times  \bfeta(\partial N_\bfeta(t))$ because the test functions are compactly supported in $\Omega$, and as such is zero on $(0,T]\times \bfeta(\partial C_\bfeta(t))$ since $C_\bfeta(t))\subset \partial\Omega$. 
\end{rem}
The function $\overline{p}$ can once again be identified as $\rho^\gamma$. Note that when $\gamma<\frac95$, one needs to use the approach developed by Feireisl which utilizes so called oscillations defect measures (see \cite{FeireislSquared} and also \cite{FeireislCompressible01}).

\vspace{.1in}
\noindent{\bf Acknowledgments:} 
We would like to express our gratitude to the anonymous referees for their valuable suggestions and comments, which greatly contributed to the improvement of the paper.
\\
M.K. was partially supported by the Czech Science Foundation (GAČR) under
grant No. 23-04766S and ERC-CZ grant LL2105 as well as the Charles
University research program No. UNCE/SCI/023. S.T. was supported by Provincial Secretariat for Higher Education and Scientific Research of Vojvodina, Serbia, grant no 142-451-2593/2021-01/2.

\bibliography{ElasticBodyFSI}
\bibliographystyle{plain}

\Addresses

\end{document}